\documentclass[reqno]{amsart}

\title{A report on an ergodic dichotomy}

\author[A. Sambarino]{Andr\'es Sambarino}

\usepackage[colorlinks = true,backref = page,hyperindex,breaklinks]{hyperref} % Links in the pdf. Backref option: each bibliographical entry denotes where it was cited
\usepackage{ stmaryrd }
\usepackage{amssymb}
\usepackage{mathtools}      % improves amsmath
\usepackage[bb = fourier,cal = euler,scr = boondox]{mathalfa}	% selecting fancy math fonts
\usepackage{enumitem}       % special itemize with custom tags and labels that work 
\usepackage[english]{babel}
\usepackage{tikz}			% figures
\usepackage{tikz-cd}	% commutative diagrams
\usepackage{tkz-euclide}
\usetikzlibrary{shapes.geometric}
\tikzcdset{arrows=dash}
\usepackage[font = small]{caption}	% smaller captions for figures
\usepackage{nicefrac}
\usetikzlibrary{calc}
\usepackage[percent]{overpic}
\usepackage{mathdots}
\usepackage{array,multirow}

\usepackage{afterpage}
\usepackage[section]{placeins}
\usepackage{ wasysym }
\usepackage{ upgreek }
\usepackage{graphicx}

\usepackage{perpage}
\MakePerPage{footnote}

\newcommand{\Z}{\mathbb{Z}}

\newcommand{\R}{\mathbb{R}}
\newcommand{\C}{\mathbb{C}}
\newcommand{\K}{\mathbb{K}}
\newcommand{\N}{\mathbb{N}}

\newcommand{\DD}{\mathbb{D}}
\renewcommand{\SS}{\mathbb{S}}
\renewcommand{\P}{\mathbb{P}}

\renewcommand{\/}{\backslash}

\newcommand{\eps}{\varepsilon}

\newcommand{\G}{\Upgamma}

\newcommand{\<}{\left<}
\renewcommand{\>}{\right>}
\newcommand{\E}{\sf E}

\newcommand{\g}{\gamma}

\newcommand{\bord}{\partial}

\newcommand{\cone}{\scr L}

\newcommand{\bus}{\beta}

\newcommand{\grupo}{\Delta}

\newcommand{\posgen}{\cal F^{(2)}}

\renewcommand{\t}{\vartheta}
\newcommand{\vt}{\theta}
\newcommand{\dirac}{\updelta}

\newcommand{\Gr}{\mathscr G}
\newcommand{\cono}{{\cal C}}

\newcommand{\h}{\scr h}

\newcommand{\LL}{\cal J}
\newcommand{\II}{\mathbf{I}}

\newcommand{\1}{\mathbf 1}
\renewcommand{\L}{\Lambda}

\renewcommand{\a}{{\frak a}}
\newcommand{\simple}{{\sf\Delta}}
\renewcommand{\root}{{\sf\Phi}}
\renewcommand{\ge}{{\frak g}}
\renewcommand{\aa}{\alpha}
\newcommand{\Weyl}{\cal W}
\newcommand{\wk}{\check}
\newcommand{\peso}{\varpi}
\newcommand{\BM}{\Omega}

\newcommand{\sroot}{\sigma}

\newcommand{\Fund}{\upphi}
\newcommand{\RR}[1]{{\upchi^{#1}}}
\newcommand{\U}{\upxi}

\newcommand{\tube}{{\mathbb T}}
\newcommand{\pp}{{\sf p}}
\newcommand{\rr}{{\sf r}}

\newcommand{\cartan}{a}

\newcommand{\conodual}{\big(\cone_{\t,\rho}\big)^*}

\newcommand{\df}{\upomega}
\newcommand{\Po}{\cal P}
\newcommand{\bb}[1]{\upbeta_{#1}}
\newcommand{\bd}[1]{\bar{\upbeta}_{#1}}
\newcommand{\susp}{\Upsigma}

\newcommand{\scr}{\mathscr}
\renewcommand{\sf}[1]{{\mathsf{#1}}}

\newcommand{\cal}{\mathcal}
\renewcommand{\frak}{\mathfrak}

\renewcommand{\angle}{\measuredangle}
\renewcommand{\L}{\mathrm{L}}

\DeclareMathOperator{\ii}{i}

\DeclareMathOperator{\inte}{int}

\DeclareMathOperator{\GL}{\sf{GL}}
\DeclareMathOperator{\Leb}{Leb}
\DeclareMathOperator{\grassman}{Gr}
\DeclareMathOperator{\clase}{C}

\DeclareMathOperator{\id}{id}

\DeclareMathOperator{\PGL}{\sf{PGL}}

\DeclareMathOperator{\holder}{Holder}

\DeclareMathOperator{\spa}{span}
\DeclareMathOperator{\rk}{rank}
\DeclareMathOperator{\Hess}{Hess}

\newcommand{\sus}{\bar\Upomega}
\newcommand{\EE}{\Sigma}
\newcommand{\cjto}{\G\/\big(\bord^2\G\times V\big)}
\newcommand{\cjtot}{\G\/\big(\bord^2\G\times \E_\t\big)}

\newcommand{\PP}{\mathbf{P}}
\newcommand{\QQ}{\mathbf{Q}}
\newcommand{\cdc}{\big(\cone_c\big)^*}
\newcommand{\medidas}{\cal M}
\newcommand{\ann}{\mathrm{Ann}}

\newcommand{\BB}{\mathrm{B}}
\newcommand{\con}[1]{\bord_{#1}\G}
\renewcommand{\ss}{\mathrm{s}}
\newcommand{\uu}{\mathrm{u}}

\newcommand{\cu}{\mathrm{cu}}

\newcommand{\dual}{\wp}
\newcommand{\lb}{\llbracket}
\newcommand{\rb}{\rrbracket}

\newcommand{\nocontentsline}[3]{}
\newcommand{\tocless}[2]{\bgroup\let\addcontentsline=\nocontentsline#1{#2}\egroup}

\newtheorem{thmA}{Theorem}

\newtheorem{thm1}[subsubsection]{Theorem}
\newtheorem{prop1}[subsubsection]{Proposition}
\newtheorem{lema1}[subsubsection]{Lemma}
\newtheorem{cor1}[subsubsection]{Corollary}

\newtheorem{thm}{Theorem}[section]
\newtheorem{lema}[subsection]{Lemma}
\newtheorem{cor}[thm]{Corollary}

\theoremstyle{definition}

\newtheorem{obs}[subsection]{Remark}

\newtheorem{defi1}[subsubsection]{Definition}

\newtheorem{obs*}[subsubsection]{Remark}

\newtheorem{assu}{Assumption}

\newtheorem*{nott}{Setting}
\newtheorem*{notacion}{Notation}
\theoremstyle{remark}

\hypersetup{
    colorlinks=true,       % false: boxed links; true: colored links
    linkcolor=blue,          % color of internal links
urlcolor=blue,
citecolor=blue
}

\thanks{A.S. was partially financed by ANR DynGeo ANR-16-CE40-0025.}

\begin{document}

\begin{abstract} We establish (some directions) of a \emph{Ledrappier correspondence} between H\"older cocycles, Patterson-Sullivan measures, etc for word-hyperbolic groups with metric-Anosov Mineyev flow. We then study Patterson-Sullivan measures for $\t$-Anosov representations over a local field and show that these are parametrized by the $\t$-\emph{critical hypersurface} of the representation. We use these Patterson-Sullivan measures to establish a dichotomy concerning directions in the interior of the $\t$-limit cone of the representation in question: if $\sf u$ is such a half-line, then the subset of $\sf u$-\emph{conical limit points} has either total-mass if $|\t|\leq2$ or zero-mass if $|\t|\geq4.$ The case $|\t|=3$ remains unsettled.\end{abstract}

\maketitle

\tableofcontents

\section{Introduction}

Let $\sf G$ be the real points of a semi-simple real-algebraic group of the non-compact type. The (Riemannian) globally symmetric space $\sf X$ associated to $\sf G$ is non-positively curved, its visual boundary $\bord_\infty \sf X$ is a union of compact $\sf G$-orbits, para\-me\-tri\-zed by directions in a (fixed beforehand) closed Weyl chamber $\a^+$ of $\ge.$  The $\sf G$-orbit associated to a direction $\sf u\in\P(\a^+)$ is $\sf G$-equivariantly identified with the flag space $\cal F_{\t_\sf u}$ of $\sf G,$ where $\t_{\sf u}$ is the subset of simple roots that do not vanish on $\sf u.$

Let now $\G$ be a finitely generated group and $\rho:\G\to\sf G$ a representation with discrete image. A fundamental object of study is the \emph{limit set} $\L_{\rho}$ of $\rho(\G)$ on the visual boundary $\bord_\infty\sf X,$ defined as the set of accumulation points of a (any) orbit $\rho(\G)\cdot o$ on the natural compactification $\sf X\cup\bord_\infty\sf X.$

When $\rho(\G)$ is Zariski-dense, this object has the following topological description by Benoist \cite{limite}: the action of $\rho(\G)$ on each flag space $\cal F_\t$ has a smallest closed invariant set, called the \emph{limit set on $\cal F_\t$} and denoted by $\L_\rho^\t$; on the other hand one has the \emph{limit cone} $\cone_\rho\subset\a^+$ of $\rho(\G),$ defined as the subset of $\a^+$ of accumulation points of sequences of the form $$t_n\cartan\big(\rho(\g_n)\big),$$ where $t_n\in\R_+$ converges to $0,$ $\g_n\in\G$ goes to infinity and $\cartan:\sf G\to\a^+$ is the \emph{Cartan projection}. It is a convex cone with non-empty interior and the limit set $\L_{\rho(\G)}$ on $\bord_\infty\sf X$ is the "fibration" over $\P(\cone_\rho),$ whose fiber over a given direction $\sf u\in\P(\cone_\rho)$ is the limit set $\L_\rho^{\t_\sf u}$ of $\rho(\G)$ on $\cal F_{\t_\sf u}.$ 

Inspired by the rank 1 case, as in Sullivan \cite{sullivan}, one may seek to distinguish the subset of \emph{conical points}  of $\L_\rho$, i.e. points on the limit set that are approached in a uniform manner by elements of the orbit $\rho(\G)\cdot o.$ However, the definition of uniform depends on:\begin{itemize}\item[-] the type of $\sf G$-orbit the point lies in: a point $x\in\L_\rho^\t$ is \emph{conical} if there exists a (to be called \emph{conical}) sequence $\{\g_n\}\subset\G$ converging to $x$ such that for every $y\in\L^{\ii\t}_\rho$ in general position\footnote{here we let $\ii:\a\to\a$ be the opposition involution and $\ii\t:=\t\circ\ii,$} with $x$ the sequence $\g_n^{-1}(y,x)$ has compact closure on the space of pairs of flags in general position $\cal F_\t^{(2)},$
\item[-] the specific direction $\sf u\in\P(\cone_\rho),$ associated to the given point: fix a norm $\|\,\|$ on $\a$ and define the \emph{tube of size} $r>0$ as the $r$-tubular neighborhood $$\tube_r(\sf u)=\{v\in\a:B(v,r)\cap\sf u\neq\emptyset\},$$ then $x\in\L^{\t_{\sf u}}_\rho$ is $\sf u$-\emph{conical} if there exists $r>0$ and a conical sequence $\g_n\to x$ such that for all $n$ one has $$\cartan\big(\rho(\g_n)\big)\in\tube_r(\sf u).$$\end{itemize}

A measurable description has been recently established by Burger-Landesberg-Lee-Oh \cite{BLLH} for $\sf u$-conical points of Zariski-dense subgroups: under some extra assumptions, the Patterson-Sullivan measure associated to the direction $\sf u$ charges totally the subset of $\sf u$-conical points iff $\sf G$ has rank $\leq3,$ if $\rk\sf G\geq4$ then the subset of $\sf u$-conical points has zero mass.

In this paper we will also study a measurable description of $\sf u$-conical limit points, but for general Anosov representations, a class introduced by Labourie \cite{labourie} for fundamental groups of closed negatively-curved manifolds and generalized by Guichard-Wienhard \cite{olivieranna} for arbitrary (finitely generated) word-hyperbolic groups. Thanks to the recent work by Kapovich-Leeb-Porti \cite{KLP-Morse} (see also Bochi-Potrie-S. \cite{BPS} and Gu\'eritaud-Guichard-Kassel-Wienhard \cite{GGKW}) we can define them as follows, see \S\,\ref{coarse}.

\begin{defi1} Let $\t\subset\simple$ be a non-empty subset of simple roots and denote by $|\,|$ the word length on $\G$ for some (fixed) symmetric generating set. A representation $\rho:\G\to\sf G$ is $\t$-\emph{Anosov} if there exist positive constants $c,\mu$ such that for all $\g\in\G$ and $\sroot\in\t$ one has $$\sroot\Big(\cartan\big(\rho(\g)\big)\Big)\geq \mu|\g|-c.$$ 
\end{defi1}

A key feature of a $\t$-Anosov representation $\rho$ is that $\G$ is necessarily word-hyperbolic and there exist continuous $\rho$-equivariant limit maps (Proposition \ref{conical}) defined on its Gromov-boundary, \begin{alignat*}{2}\xi^\t&:\bord\G\to\cal F_\t\\ \xi^{\ii\t}&:\bord\G\to\cal F_{\ii\t},\end{alignat*} such that the flags $\xi^{\ii\t}(x)$ and $\xi^\t(y)$ are in general position whenever $x\neq y.$

We begin by studying the Patterson-Sullivan theory for these groups. Fix then $\t\subset\simple,$ let $$\a_\t=\bigcap_{\sroot\in\simple-\t}\ker\sroot$$ be the center of the associated Levi group and let $p_\t:\a\to\a_\t$ be the projection invariant under the subgroup of the Weyl group point-wise fixing $\a_\t$ (see \S\,\ref{center}). The dual space $(\a_\t)^*$ sits naturally as the subspace of $\a^*$ of $p_\t$-invariant linear forms. It is spanned by the fundamental weights of the elements in $\t$: $$(\a_\t)^*=\<\big\{\peso_\sroot|\a_\t:\sroot\in\t\big\}\>.$$ Let us write $\cartan_\t$ for the composition $a_\t=p_\t\circ\cartan:\sf G\to\a_\t.$

Let $\bus:\sf G\times\cal F_\simple\to\a$ be the \emph{Buseman-Iwasawa cocycle} of $\sf G$ introduced by Quint \cite{quint1} (see \S\,\ref{BI}). The map $\bus_\t=p_\t\circ\bus$ factors as a cocycle $\bus_\t:\sf G\times \cal F_\t\to\a_\t.$

\begin{defi1}A \emph{Patterson-Sullivan measure for $\rho$ on $\cal F_\t$} is a probability measure $\nu$ on $\cal F_\t$ such that there exists $\varphi\in\a_\t^*$ with, for every $\g\in\G,$  $$\frac{d\rho(\g) _*\nu}{d\nu}(\cdot)=q^{-\varphi\big(\bus_\t(\rho(\g)^{-1},\cdot)\big)}.$$ \end{defi1}

For $\varphi\in(\a_\t)^*$ denote by $$\delta^\varphi =\lim_{t\to\infty}\frac1t\log\#\Big\{\g\in\G:\varphi\Big(\cartan\big(\rho(\g)\big)\Big)\leq t\Big\}\in[0,\infty]$$ and, inspired by Quint's growth indicator \cite{quint2}, consider the $\t$-\emph{critical hypersurface} $$\cal Q_{\t,\rho} =\big\{\varphi\in(\a_\t)^*:\delta^\varphi=1\big\}.$$ Let us define the $\t$-\emph{limit cone} of $\rho,$ denoted by $\cone_{\t,\rho},$ as the asymptotic cone of the projections $$\big\{\cartan_\t\big(\rho(\g)\big):\g\in\G\},$$ i.e. all limits of sequences of the form $t_n \cartan_\t\big(\rho(\g_n)\big),$ where $\g_n\to\infty$ in $\G$ and $t_n\to0$ in $\R_+.$ 

In the real case, if $\rho(\G)$ is Zariski-dense, then Benoist's aforementioned result implies that $\cone_{\t,\rho}$ has non-empty interior. However, for arbitrary local fields this is no longer the case\footnote{(even assuming Zariski-density and Anosov)}. We aim to work on this more general context, so let us assume now that $\sf G$ is (the $\K$-points of) a semi-simple algebraic group over a local field $\K,$ we refer the reader to \S\,\ref{localfield} for the analogous definitions, where $\a_\t$ is replaced by the real vector space $\E_\t,$ etc.

Let $\ann(\cone_{\t,\rho})$ be the annihilator of the $\t$-limit cone and denote by $\pi^\t_\rho:(\E_\t)^*\to(\E_\t)^*/\ann(\cone_{\t,\rho})$ the quotient projection.

\begin{thmA}\label{A} Let $\rho:\G\to\sf G$ be $\t$-Anosov. Then, $\cal Q_{\t,\rho}$ is a closed co-dimension-one analytic sub-manifold of $(\E_\t)^*$ that bounds a convex set;  moreover the projection $\pi^\t_\rho\big(\cal Q_{\t,\rho}\big)$ is also a closed co-dimension-one analytic sub-manifold, boundary of a strictly convex set. For each $\varphi\in\cal Q_{\t,\rho}$ there exists a unique Patterson-Sullivan measure $\nu^\varphi$ with support on $\xi^\t(\bord\G).$ The map $\varphi\mapsto\nu^\varphi$ is an analytic homeomorphism between the projection $\pi^\t_\rho\big(\cal Q_{\t,\rho}\big)$ and the space of Patterson-Sullivan measures on $\cal F_\t$ whose support is contained in $\xi^\t(\bord\G).$ Such Patterson-Sullivan measures are ergodic and pairwise mutually singular. 
\end{thmA}

We refer the reader to Corollary \ref{existe} and Proposition \ref{corr} for the proofs of the above statements.

The fact that both $\cal Q_{\t,\rho}$ and $\pi^\t_\rho\big(\cal Q_{\t,\rho}\big)$ are closed analytic hypersurfaces  can be found in Potrie-S. \cite[Proposition 4.11]{exponentecritico} for $\K=\R$ with essentially the same arguments. The parametrization of Patterson-Sullivan measures by $\pi^\t_\rho\big(\cal Q_{\t,\rho}\big)$ was previously stablished by Lee-Oh \cite[Theorem 1.3]{HoLee} for $\K=\R,$ $\t=\simple$ and assuming Zariski-density of $\rho(\G).$ Existence and ergodicity was previously stablished, for $\K=\R$ by Dey-Kapovich \cite[Main Theorem]{Dey-Kapovich} for $\ii$-invariant functionals $\varphi\in(\a^+)^*\cap(\a_\t)^*$ and $\ii$-invariant subsets $\t;$ and S. \cite[Corollary 4.22]{orbitalcounting} for arbitrary functionals but Zariski-dense representations of fundamental groups of negatively curved manifolds. Existence of Patterson-Sullivan measures has also been stablished by Canary-Zhang-Zimmer \cite{DAT} in the real case for \emph{relative} Anosov representations.

We keep the discussion for $\K=\R$ since this is essential in the following result. Consider $\varphi\in\cal Q_{\t,\rho}$ with associated Patterson-Sullivan measure $\mu^\varphi.$ Via the duality $$\grassman_{\dim\a_\t-1}\big((\a_\t)^*\big)\to\P(\a_\t),$$ the tangent space $\sf T_{\varphi}\cal Q_{\t,\rho}$ gives a direction $\sf u_\varphi$ of $\P(\a_\t)$ contained in the relative interior of the limit cone $\cone_{\t,\rho}$ (Corollary \ref{e=1}). We then further investigate the $\mu^\varphi$-mass of $\sf u_\varphi$-conical points on $\xi^\t(\bord\G).$ 

Since we are dealing with the limit cone on $\a_\t$ (and not on $\a$ as before) $\sf u_\varphi$-conical are points are yet to be defined. It is standard that every point $\xi^\t(x)\in\xi^\t(\bord\G)$ is conical\footnote{(This follows from the fact that every point $x\in\bord\G$ is conical and the existence of the equivariant limit maps for $\rho.$)}, let us say it is further $\sf u_\varphi$-\emph{conical} if there exists a conical sequence (for $x$) as above and $r>0$ such that $\cartan_\t(\rho(\g_n))\in\tube_r(\sf u_\varphi).$ Denote by $\con\varphi\subset\bord\G$ the subset $$\con\varphi=\big\{x\in\bord\G:\xi^\t(x) \textrm{ is $\sf u_\varphi$-conical}\big\}.$$

\begin{thmA}[Theorem \ref{thmB1}]\label{B}Let $\K=\R$ and assume $\rho$ is $\t$-Anosov and Zariski-dense. If $|\t|\leq2$ then $\mu^\varphi\big(\xi^\t(\con\varphi)\big)=1,$ if $|\t|\geq4$ then $\mu^\varphi\big(\xi^\t(\con\varphi)\big)=0.$
\end{thmA}

The case $|\t|=3$ is sadly presently untreated, the missing fact that would make our technique directly apply is an ergodicity result for translation skew-products over metric-Anosov flows where the abelian group is isomorphic to $\R^2=\R^{|\t|-1},$ more precisely, we need equivalence between ergodicity and $\dim V\leq2$ in Corollary \ref{dicoflujos}.

%an ergodicity result for translation skew-products over metric-Anosov flows where the abelian group is isomorphic to $\R^2=\R^{|\t|-1},$ see subsection \ref{guivarch} for the precise required statement.

When $\t=\simple,$ a stronger version of Theorem \ref{B} dealing also with the case $|\simple|=3$ was previously established by Burger-Landesberg-Lee-Oh \cite[Theorem 1.6]{BLLH}. It is likely that the combination of their techniques and ours settles the missing $|\t|=3$ case.

\subsection{General strategy for Theorem \ref{B}} Let us briefly explain the proof of Theorem \ref{B}, which we believe is the main contribution of this work. The main ingredient is a precise description of the \emph{$\t$-parallel sets dynamics} of $\sf G.$ If $(x,y)\in\cal F_{\ii\t}\times\cal F_\t$ are in general position, the associated \emph{parallel set} is a subset of $\sf X $ consisting on the union of totally geodesic maximal flats $\sf p$ of $\sf X$ whose associated complete flags in the Furstenberg boundary $\sf p(-\a^+)$ and $\sf p(\a^+)$ contain, respectively, $x$ and $y$ as a partial flag. This parallel set is a reductive symmetric space, and the associated dynamical system consists on moving along its center.

More concisely, if one considers the space $\posgen_\t\subset\cal F_{\ii\t}\times\cal F_\t$ of transverse flags, then the space $\posgen_\t\times\a_\t$ carries a $\sf G$-action (on the left) given by $$g(x,y,v)=\big(gx,gy,v-\bus_\t(g,y)\big),$$ and an $\a_\t$-action (on the right) by translation on the last coordinate.

Observe however that the left-action of $\rho(\G)$ on $\posgen_\t\times\a_\t$ need not be proper. For $\t$-Anosov groups though one finds an $\a_\t$-invariant subset which is also $\rho(\G)$-invariant and on which this latter action is proper (\S\,\ref{enderezarV1} and \S\,\ref{beta<infty}). Its quotient by $\rho(\G)$ will be denoted, throughout this introduction, by $\cal O_{\t,\rho}.$ 

For each $\varphi\in\cal Q_{\t,\rho}$ the space $\cal O_{\t,\rho}$ will carry a $\varphi$-\emph{Bowen-Margulis measure} $\BM^\varphi$ invariant under the \emph{directional flow} $\df^\varphi:\cal O_{\t,\rho}\to\cal O_{\t,\rho}$ along $u_\varphi\in\sf u_\varphi,$ defined by (the induced on the quotient by $\rho(\G)$ of) $$(x,y,v)\mapsto(x,y,v-tu_\varphi).$$ The idea generalizing S. \cite{orbitalcounting} is that $\df^\varphi$ is conjugated to a skew-product over a metric-Anosov flow $\phi^\varphi=\big(\phi_t^\varphi:\RR\varphi\to\RR\varphi\big)_{t\in\R}$ on a compact metric space $\RR\varphi.$ This is stablished in \S\,\ref{fibered2} and previously stablished by Carvajales \cite[Appendix]{lyon2} when $\t=\simple.$ 

\begin{obs*}The flow $\phi^\varphi$ plays a central role in this work. We propose to name it \emph{the $\varphi$-refraction flow} of $\rho,$ because one may think as if the projection on the base $\RR\varphi$ \emph{refracts} the orbits of $\df^\varphi$ (almost all of them wondering\footnote{in spite of being topologically mixing, these flows are wondering in a measureable sense i.e. almost every point belongs to a subset of positive measure with bounded return times,} when $|\t|\geq4$) in order to bind them in a compact space $\RR\varphi$ and obtain non-trivial dynamical behavior. On the other hand, the term \emph{geodesic flow} has too many meanings on this setting (the geodesic flow of $\G,$ the geodesic flow of the locally symmetric space $\rho(\G)\/\sf X,$ the geodesic flow of a projective-Anosov representation associated to $\rho$ by Plucker embedings...).\end{obs*}

An ergodicity result for skew-products over metric-Anosov flows (see \S\,\ref{guivarch}) gives an \emph{ergodic vs totally dissipative} dichotomy for $\df^\varphi$ according to $|\t|\leq2$ or $|\t|\geq4,$ here the base field $\K=\R$ and Zariski-density of $\rho$ are essential, since Benoist's \cite{benoist2} density of Jordan projections does not hold for non-Archimedean $\K.$ This dichotomy is reminiscent of Sullivan's \cite{sullivan} \emph{conservative vs totally dissipative} dichotomy in rank 1. Observe again the untreated case $|\t|=3.$

These dynamical properties of $\df^\varphi$ imply the following. The set $\cal K(\df^\varphi)$ of points in $\cal O_{\t,\rho}$ whose future orbit returns unboundedly to some open bounded set, has either zero $\BM^\varphi$-mass if $|\t|\geq4$ or its complement has zero $\BM^\varphi$-mass if $|\t|\leq2.$

The key feature now is to relate $\sf u_\varphi$-conical points with the set $\cal K(\df^\varphi),$ this is attained in Lemma \ref{thmB1} where it is shown that a triple $(x,y,v)\in\posgen_\t\times\a_\t$ projects to $\cal K(\df^\varphi)$ \emph{if and only if} $y$ is $\sf u_\varphi$-conical. The previous dynamical dichotomy gives then the dichotomy on the $\mu^\varphi$-measure on conical points:\begin{equation}\label{dico1}|\t|\leq2\Rightarrow\mu^\varphi(\con\varphi)=1,\qquad |\t|\geq4\Rightarrow\mu^\varphi(\con\varphi)=0.\end{equation}

The global strategy of our proof is different from the analog result in Burger-Landesberg-Lee-Oh \cite{BLLH}. While, inspired by them, we also use a mixing result, the use of Dirichlet-Poincar\'e series along tubes does not play any role in the proof of Theorem \ref{B}, nor on the ergodicity dichotomy for directional flows.

Let us end this Introduction be observing that both Burger-Landesberg-Lee-Oh \cite{BLLH} and Chow-Sarkar \cite{Chow-Sarkar} prove dynamical statements on $\rho(\G)\/\sf G$ (as opposed to $\rho(\G)\/\sf G/\sf M$).

\subsection{Plan of the paper} In \S\,\ref{dynamics} we recall some basic facts about the Ergodic Theory of metric-Anosov flows, and then study translation cocycles over them. Section \ref{piedrapie} deals with a Ledrappier correspondence for word-hyperbolic groups whose Gromov-Mineyev geodesic flow is metric-Anosov. We will mainly apply these results to the Buseman-Iwasawa cocycle of $\sf G,$ for applications to other cocycles the reader may check Carvajales \cite{lyon1,lyon2}.

We then recall in \S\,\ref{localfield} necessary definitions on semi-simple algebraic groups over a local field and deal with Anosov representations on \S\,\ref{AnosovReps}. We explain in this section how the Ledrappier correspondence applies in this setting to give, mainly:\begin{itemize}\item[-] uniqueness results on the Patterson-Sullivan measures,\item[-] precise dynamical information on the directional flows $\df^\varphi.$\end{itemize}
The proof of Theorem \ref{B} can be found in \S\,\ref{medidadico}.

\subsection{Acknowledgements} This paper grew from a question asked to the author by Hee Oh on the ergodicity of directional flows for Anosov representations (now Theorem \ref{dicoAnosov}). He would like to thank her for asking the question. He would also like to thank Marc Burger for encouraging him to pursue the ideas sketched on a very early version of this work, together with the remaining authors O. Landesberg, M. Lee and H. Oh of the paper \cite{BLLH}, whose results inspired Theorem \ref{B}.

\section{Skew-products over metric-Anosov flows}\label{dynamics}

Throughout this section we let $X$ be a compact metric space and $V$ a finite dimension real vector space.

\subsection{Thermodynamic Formalism and reparametrizations}\label{thermo1}

Let $\phi=(\phi_t:X \to X)_{t\in\R}$ be a continuous flow without fixed points. The space of $\phi$-invariant probability measures on $X$ is denoted by $\medidas^\phi.$ It is a  convex, weakly-compact subset of $C^*(X),$ the dual space to the space of continuous functions equipped with the uniform topology. The \emph{metric entropy} of $m\in\medidas^\phi$ will be denoted by $h(\phi,m),$ its definition can be found in Aaronson \cite{AaronsonLibro}. Via the variational principle, we will define the \emph{topological pressure} (or just \emph{pressure}) of a function $f:X\to\R$ as the quantity \begin{equation}\label{pressuredefi}P(\phi,f)= \sup_{m\in\medidas^{\phi}} \Big(h(\phi,m)+\int_X fdm\Big).\end{equation}  A probability measure $m$ realizing the least upper bound is called an \emph{equilibrium state} of $f.$ An equilibrium state for $f\equiv0$ is called a \emph{measure of
maximal entropy}, and its entropy is called the \emph{topological entropy} of $\phi,$ denoted by $h(\phi).$

Let $f:X\to\R_{>0}$ be continuous. For every $x\in X$ the function $k_f:X\times\R\to \R,$ defined by $k_f(x,t)=\int_0^tf(\phi_sx)ds,$ is an increasing homeomorphism of $\R.$ There is thus a continuous function $\alpha_f:X\times\R\to\R$ such that for all $(x,t)\in X\times\R,$

$$\alpha_f\big(x,k_f(x,t)\big)=k_f\big(x,\alpha_f(x,t)\big)=t.$$

\noindent
The \emph{reparametrization} of $\phi$ by $f:X\to\R_{>0}$ is the flow $\phi^f=(\phi^f_t:X \to X)_{t\in\R}$ defined, for all $(x,t)\in X\times\R$ by $$\phi^f_t(x)=\phi_{\alpha_f(x,t)}(x).$$ 

The \emph{Abramov transform} of $m\in\medidas^\phi$ is the probability measure $m^\#\in\medidas^{\phi^f}$ defined by \begin{equation}\label{abramovdef}m^\#=\frac{f\cdot m}{\int fdm}.\end{equation} One has the following:

\begin{lema1}[{S. \cite[Lemma 2.4]{quantitative}}]\label{Abramov} Let $f:X\to\R_{>0}$ be a continuous function. Assume the equation $$P(\phi,-sf)=0\qquad s\in\R$$ has a finite positive solution $h,$ then $h$ is the topological entropy of $\phi^f.$ Conversely if $h(\phi^f)$ is finite then it is a solution to the last equation. In this situation, the Abramov transform induces a bijection between the set of equilibrium states of $-hf$ and the set of probability measures maximizing entropy for $\phi^f.$
\end{lema1}

Two continuous maps $f,g:X\to V$ are \emph{Liv\v sic-cohomolo\-gous} if there exists a $U:X\to V,$ of class $\clase^1$ in the direction of the flow\footnote{i.e. such that if for every $x\in X,$ the map $t\mapsto U(\phi_tx)$ is of class $\clase^1,$ and the map $x\mapsto \left.\frac{\partial }{\partial t}\right|_{t=0}U(\phi_tx)$ is continuous}, such that for all $x\in X$ one has $$f(x)-g(x)=\left.\frac{\partial}{\partial t}\right|_{t=0} U(\phi_tx).$$

\begin{obs*}\label{LivP} If $f$ and $g$ are real-valued and Liv\v sic-cohomologous then $P(\phi,f)=P(\phi,g).$
\end{obs*}

\subsection{Metric-Anosov flows I: Liv\v sic-cohomology}\label{anosovtop}
\emph{Metric-Anosov flows} are a metric version of what is commonly known as \emph{hyperbolic flows}. The former are called Smale flows by Pollicott \cite{smaleflows}, who transferred to this more general setting the classical theory carried out for the latter. We recall here their definition and some well known facts on their Ergodic Theory needed in the sequel. Throughout this subsection we will further assume that $\phi$ is H\"older-continuous with an exponent independent of $t,$ that it is \emph{transitive}, i.e. it has a dense orbit, and that it is metric-Anosov.

For $\eps>0$ the \emph{local stable/unstable set} of $x$ are (respectively) \begin{alignat*}{2}W^\ss_\eps(x) & =\{y\in X:d(\phi_tx,\phi_t y)\leq\eps\ \forall t>0\textrm{ and }d(\phi_tx,\phi_t y)\to0\textrm{ as }t\to\infty\}\\ W^\uu_\eps(x) & =\{y\in X:d(\phi_{-t}x,\phi_{-t} y)\leq\eps\ \forall t>0\textrm{ and }d(\phi_{-t}x,\phi_{-t} y)\to0\textrm{ as }t\to\infty\}.\end{alignat*}

\begin{defi1}[Metric-Anosov]\label{defiMA}The flow $\phi$ is \emph{metric-Anosov} if the following holds:
\begin{itemize} \item[-] (Exponential decay) There exist positive constants $C,$ $\lambda$ and $\eps$ such that for every $x\in X,$ every $y\in W^\ss_\eps(x)$ and every $t>0$ one has $$d(\phi_tx,\phi_ty)\leq Ce^{-\lambda t},$$ and such that for every $y\in W^\uu_\eps(x)$ one has $d(\phi_{-t}x,\phi_{-t}y)\leq Ce^{-\lambda t}.$
\item[-] (Local product structure) There exist $\delta,\eps>0$ and a H\"older-continuous map $$\upnu:\{(x,y)\in X\times X:d(x,y)<\delta\}\to \R$$ such that $\upnu(x,y)$ is the unique value $\upnu$ such that $W^\uu_\eps(\phi_\upnu x)\cap W^\ss_\eps(y)$ is non-empty, and consists of exactly one point, called $\<x,y\>;$ and for every $x\in X$ the map $$W^\ss_\eps(x)\times W^\uu_\eps(x)\times(-\delta,\delta)\to X,$$ given by $(y,z,t)\mapsto\phi_t(\<y,z\>),$ is a H\"older-homeomorphism onto an open neighborhood of $x.$
\end{itemize}
\end{defi1}

A \emph{translation cocycle} over $\phi$ is a map $k: X\times\R \to V$ such that for every $x\in X$ and $t,s\in\R,$ one has $$k(x,t+s)=k(\phi_sx,t)+k(x,s),$$ and such that the map $k(\cdot,t)$ is H\"older-continuous with exponent independent of $t,$ and with bounded multiplicative constant when $t$ remains on a bounded set.  Two translation cocycles $k_1$ and $k_2$ are \emph{Liv\v sic-cohomologous}, if there exists a continuous map $U:X\to V,$ such that for all $x\in X$ and $t\in\R$ one has \begin{equation}\label{eq:algo}k_1(x,t)-k_2(x,t)=U(\phi_tx)-U(x).\end{equation}

If $k$ is a translation cocycle then the \emph{period for $k$} of a periodic orbit $\tau$ is  $$\ell_\tau(k)=k\big(x,p(\tau)\big),$$ for any $x\in\tau.$ The \emph{marked spectrum} $\tau\mapsto \ell_k(\tau)$ is a cohomological invariant that uniquely determines its class:

\begin{thm1}[Liv\v sic \cite{livsic}]\label{livsic3} Let $k:X\times \R\to V$ be a translation cocycle. If $\ell_k(\tau)=0$ for every periodic orbit $\tau,$ then $k$ is Liv\v sic-cohomologous to $0.$\end{thm1}

Observe that if $f:X\to V$ is H\"older-continuous then the map $$k_f(x,t)=\int_0^tf(\phi_sx)ds$$ is a translation cocycle. Two such functions are Liv\v sic-cohomologous if and only if the associated cocycles are, and the \emph{period} of $f$ on $\tau$ is, for any $x\in\tau,$ $$\ell_\tau(f)=\int_\tau f=k_f\big(x,p(\tau)\big).$$ It turns out that every cocycle is Liv\v sic-cohomologous to a cocycle of the form $k_f$:

\begin{cor1}[{S. \cite[Lemma 2.6]{orbitalcounting}}]\label{graciadivina} If $k:X\times\R\to V$ is a translation cocycle then there exists a H\"older-continuous $f:X\to V$ such that $ k$ and $k_f$ are Liv\v sic-cohomologous.
\end{cor1}

\begin{proof} For any $\kappa>0,$  the function $j(x,t)=\frac1{2\kappa}\int_{-\kappa}^\kappa k(x,t+s)ds$ is differentiable on the second variable, let $f(x)=\big(\partial/\partial s\big)|_{s=0}j(x,s).$ Then \begin{alignat*}{2}k_f(x,t)&=\int_0^tf(\phi_ux)du =\int_0^t\frac{\partial}{\partial s}\Big|_{s=0}j(\phi_ux,s)du\\ & =\int_0^t\frac{\partial}{\partial s}\Big|_{s=0}j(x,s+u)du = j(x,t)-j(x,0),\end{alignat*} so the period $k_f(x,p(\tau))=j(x,p(\tau))-j(x,0)=k(x,p(\tau)).$ By Theorem \ref{livsic3} the cocycles $k$ and $k_f$ are thus Liv\v sic-cohomologous.\end{proof}

We record also the following immediate consequence of Liv\v sic's Theorem:

\begin{obs*}The space of functions Liv\v sic-cohomologous to a strictly positive function is an (open cone on an) infinite dimensional space.
\end{obs*}

In this context much more information can be stated about the pressure function. Recall that the space $\holder^\alpha(X)$ of real valued $\alpha$-H\"older functions is naturally a Banach space when equipped with the norm $$\|f\|_\alpha=\|f\|_\infty+\sup_{x\neq y}\frac{|f(x)-f(y)|}{d(x,y)^\alpha}.$$

\begin{prop1}[{Bowen-Ruelle \cite{bowenruelle} and Parry-Pollicott \cite[Prop. 4.10]{parrypollicott}}]\label{pder} The function $P(\phi,\cdot)$ is analytic on $\holder^\alpha(X).$ If $f,g\in\holder^\alpha(X),$ then $$\frac{\partial}{\partial t}\Big|_{t=0}P(\phi,f+tg)=\int gdm_f,$$ where $m_f$ is the equilibrium state of $f,$ and the funcion $t\mapsto P(\phi,f+tg)$ is strictly convex unless $g$ is Liv\v sic-cohomologous to a constant. Finally, one also has \begin{equation}\label{PrPer}P(\phi,f)=\limsup_{t\to\infty}\frac1t\log\sum_{\tau:p(\tau)\leq t}e^{\ell_\tau(f)}.\end{equation} \end{prop1}

Let $f\in\holder^\alpha(X)$ have non-negative (and not all vanishing) periods and define its \emph{entropy} by $$\h_f=\limsup_{s\to\infty}\frac1s\log\#\big\{\tau\textrm{ periodic}:\int_\tau f\leq s\big\}\in(0,\infty].$$ 

\begin{obs*}\label{h>0}Observe that $\h_f$ is necessarily $>0$ since $f$ must have a positive maximum and $h(\phi)>0.$\end{obs*}

One has the following lemma.

\begin{lema1}[{Ledrappier \cite[Lemma 1]{ledrappier}+S. \cite[Lemma 3.8]{quantitative}}]\label{positiva}
 Consider a H\"older-continuous function $f:X\to\R$ with non-negative periods. Then the following statements are equivalent: \begin{itemize}\item[-] the function $f$ is Liv\v sic-cohomologous to a positive H\"older-continuous function,  \item[-] there exists $\kappa>0$ such that $\int_\tau f>\kappa p(\tau)$ for every periodic orbit $\tau,$ \item[-] the entropy $\h_f$ is finite,\item[-]the function $t\mapsto P(\phi, -tf)$ has a positive zero, in which case is $\h_f.$\end{itemize}
\end{lema1}

Let us fix an exponent $\alpha$ and consider the cone $\holder^\alpha_+(X,\R)$ of H\"older-continuous functions that are Liv\v sic-cohomologous to a strictly positive function. The implicit function theorem for Banach spaces (see Akerkar \cite{akerkar}) and the explicit formula for the derivative of pressure (Proposition \ref{pder}) give the following corollary.

\begin{lema1}\label{entropyAn} The entropy map $\h:\holder^\alpha_+(X,\R)\to\R_+$ is analytic.
\end{lema1}

\begin{proof}Indeed Lemma \ref{positiva} gives the equation $P(\phi,-\h_ff)=0$ and equation (\ref{pder}) gives that the non-vanishing derivative $$d_{-\h_ff}P(\phi,f)=\int fdm_{-h_ff}>0,$$  so the implicit function completes de result.\end{proof}

\subsection{Metric-Anosov flows II: Ergodic Theory}\label{defi:markovpatition}

A fundamental tool for studying the Ergodic Theory of metric-Anosov systems is the existence of a Markov coding.

Let $\EE$ be an irreducible sub-shift of finite type equipped with its shift transformation $\sigma:\EE\to\EE,$ and $r:\EE\to\R_{>0}$ be H\"older-continuous. Let $\hat r:\EE\times\R\to\EE\times\R$ be defined as $$\hat r(x,s)=(\sigma x, s-r(x)),$$ and consider the quotient space $\susp_r=\EE\times\R/\<\hat r\>.$ It is equipped with the flow $\upsigma^r=\big(\upsigma^r_t:\susp_r\to\susp_r\big)_{t\in\R}$ induced on the quotient by the translation flow.

\begin{defi1}[Markov coding] A triplet $(\EE,\pi,r)$ is a \emph{Markov coding} for $\phi$ if $\EE$ and $r$ are as above, $\pi:\EE\to X$ is H\"older-continuous and the function $\pi_r:\EE\times\R\to X$ defined as $$\pi_r(x,t)=\phi_t\pi(x)$$ verifies the following conditions: 
\begin{itemize}\item[i)] $\pi_r$ is H\"older-continuous, surjective and $\hat r$-invariant; it passes then to the quotient $\susp_r,$\item[ii)] $\pi_r:\susp_r\to X$ is bounded-to-one and injective on a residual set which is of full measure for every ergodic $\upsigma^r$-invariant measure of total support, \item[iii)] for every $t\in\R$ one has $\pi_r\upsigma^r_t=\phi_t\pi_r.$
\end{itemize}
\end{defi1}

The following result has a long history, see for example Bowen \cite{bowen1,bowen2}, Ratner \cite{ratner2}, Pollicott \cite{smaleflows} and more recently Constantine-Lafont-Thompson \cite{Lafont}.

\begin{thm1}[Existence of coding]\label{particiones} A transitive metric-Anosov flow admits a Markov coding.
\end{thm1}

The above is a fundamental tool to obtain the following, see for example Bowen-Ruelle \cite{bowenruelle}, Parry-Pollicott \cite{parrypollicott} and more recently Giulietti-Kloeckner-Lopes-Marcon \cite{GKLM}, recall from \S\,\ref{thermo1} the definition of equilibrium state.

\begin{thm1}[Uniqueness of equilibrium states]\label{equilibrios}Let $f:X\to\R$ be H\"older-continuous, then there exists a unique equilibrium state for $f,$ denoted by $m_f;$ it is an ergodic measure. If $g:X\to\R$ is also H\"older, then $m_g\ll m_f$ if and only if $f-g$ is Liv\v sic-cohomologous to a constant function, in which case $m_g=m_f.$ The function $f\mapsto m_f,$ defined on the space of H\"older-continuous functions with fixed exponent, is analytic.\footnote{We emphasize that the space of measures is endowed with the differentiable structure induced by being the dual space of continuous functions.} \end{thm1}

%Let us say a word on analyticity of $f\mapsto m_f.$ Via the Markov codings, one translates the question to a H\"older potential F, defined on a subshift of finite type. One defines then the associated Ruelle operator \cal L_F on the Banach space of continuous functions 

A final fact we will require on this setting (introduced by Margulis \cite{margulistesis}) is the decomposition  of the measure of maximal entropy along the stable/ central-unstable sets of $\phi.$

The \emph{stable/unstable leaf} of $x$ is \begin{alignat*}{2} W^\ss(x) & = \bigcup_{t\in\R_+} \phi_{-t}\big(W^\ss_\eps(\phi_tx)\big)\\ W^\uu(x) & = \bigcup_{t\in\R_+} \phi_{t}\big(W^\uu_\eps(\phi_{-t}x)\big),\end{alignat*} and the \emph{central} stable/unstable leaf is (respectively) the $\phi$-orbit of $W^\ss(x)$ (resp. $W^\uu(x)$). These sets are independent of (any small enough) $\eps$ (i.e. smaller than the $\eps$ given by Definition \ref{defiMA}).

One has the following, see for example Margulis \cite{MargulisMedida}, Pollicott \cite{Pollicott-MM} for a construction via Markov codings or Katok-Hasselblatt's book \cite[\S\,5 of Chapter 20]{katokh} for the discrete-time case.

\begin{thm1}[Margulis description]\label{mm}For each $x\in X$ there exists a measure $\mu_x^{\ss}$ on the stable leaf $W^\ss(x)$ and a measure $\mu_x^{\cu}$ on the central unstable leaf such that\begin{itemize}\item[-] for all $t>0$ and all measurable $U\subset W^{\ss}(x)$ one has\begin{equation}\label{contraccion}\mu_{\phi_tx}^\ss(\phi_tU)=e^{-h(\phi)t}\mu_{x}^\ss(U),
\end{equation}\item[-] the local product structure on Definition \ref{anosovtop} induces a local isomorphism between the measure $\mu_x^\ss\otimes\mu^\cu_x$ and the measure of maximal entropy of $\phi.$
\end{itemize}The family of measures is unique in the following sense. If $\nu_x^\ss$ and $\nu_x^\cu$ are also a family of measures along the stable and central-unstable leafs of $x$ such that the product $\nu_x^\ss\otimes\nu_x^\cu$ is locally isomorphic to a $\phi$-invariant measure and $\nu^\ss_x$ verifies equation (\ref{contraccion}) with some arbitrary fixed $\delta>0$ (instead of $h(\phi)$) then $\delta=h(\phi),$ $\nu^\ss_x=\mu^\ss_x$ and $\mu^\cu_x=\nu^\cu_x.$
\end{thm1}

\subsection{Skew-products over sub-shifts}\label{erg1} Consider now the two-sided subshift $\EE$ and let $\sf K:\EE\to V$ be H\"older-continuous.

\begin{defi1}\label{nnA}We say that $\sf K$ is \emph{non-arithmetic} if the group spanned by the periods of $\sf K$ is dense in $V.$
\end{defi1}

The \emph{skew-product system} is defined by $f=f^{\sf K}:\EE\times V\to\EE\times V$  \begin{equation}\label{trsl}f(p,v)=\big(\sigma(p),v-\sf K(p)\big).\end{equation}  If $\nu$ is a $\sigma$-invariant probability measure on $\EE$ then the measure $\Upomega=\Upomega_\nu=\nu\otimes \Leb$ is $f$-invariant.

The following proposition seems to be well known but we haven't been able to find a specific reference in the literature, for completeness we added a short proof Appendix \ref{dim1}.

\begin{prop1}\label{d=1} Let $\EE$ be a two-sided sub-shift, $\nu$ be an equilibrium state of some H\"older potential, and $\sf K:\EE\to\R$ a non-arithmetic H\"older-continuous function with $\int \sf Kd\nu=0.$ Then the skew-product $f^{\sf K}:\EE\times\R\to\EE\times\R$ is ergodic w.r.t. $\Upomega_\nu.$
\end{prop1}

We record also the following classical lemma. Let us say that a subset of $\EE\times V$ is \emph{bounded} if it has compact closure, and that it has \emph{total mass} (w.r.t. $\Upomega$) if its complement has measure zero. As the space $\EE\times V$ is non-compact, it is natural to study the subset of points of $\EE\times V$ whose future orbit returns infinitely many times to a fixed open bounded set: $$\cal K(f)=\big\{p\in \EE\times V: \exists \BB\textrm{ open bounded set and $n_k\to\infty$ with }f^{n_k}(p)\in\BB\big\}.$$ One can be more specific. If $\BB_1,\BB_2\subset\EE\times W$ we want to understand the measure of $$\cal K(\BB_1,\BB_2)=\big\{p\in\BB_1:\,\exists n_k\to\infty\textrm{ with $f^{n_k}(p)\in\BB_2$}\big\},$$ to this end one considers the sum $\sum_0^\infty\Upomega\big(\1_{\BB_1}\cdot\1_{\BB_2}\circ f^n\big)$:

\begin{lema1}\label{dico} If ${\sum_{n=0}^\infty \Upomega\big(\1_{\BB_1}\cdot\1_{\BB_2}\circ f^n\big)<\infty}$ then $\Upomega\big(\cal K(\BB_1,\BB_2)\big)=0.$ On the other hand if $\nu$ has no atoms and $f$ is ergodic w.r.t. $\Upomega$ then $\cal K(f)$ has total mass and for every pair $\BB_1,\BB_2$ one has $\cal K(\BB_1,\BB_2)$ has total mass on $\BB_1.$
\end{lema1}

\begin{proof}This is a standard argument valid for any measure preserving transformation. The first assertion follows by looking at the tail of the series in question $$\sum_{n=k}^\infty \Upomega\big(\1_{\BB_1}\cdot\1_{\BB_2}\circ f^n\big)\geq\Upomega(E_k),$$ where, for each $k\in\N,$ $E_k =\{p\in\BB_1:\,\exists N\geq k\textrm{ with }f^N(p)\in\BB_2\big\}.$ The second assertion can be found in, for example, Aaronson's book \cite[page 22]{AaronsonLibro}.\end{proof}

\subsection{An ergodic dichotomy}\label{guivarch}

As in \S\,\ref{defi:markovpatition}, let $\upsigma^r$ be the suspension of the shift on $\EE$ by the function $r.$ If $\nu$ is a $\sigma$-invariant probability measure, then $\nu\otimes dt/\int rd\nu$ is invariant under the translation flow $\EE\times\R$ and induces thus a $\upsigma^r$-invariant probability measure on $\susp_r,$ denoted by $\hat\nu.$ 

\begin{obs*}\label{equientro}It is a classical fact, to be found for example in Bowen-Ruelle \cite{bowenruelle}, that if $\nu$ is the equilibrium state of $-h(\upsigma^r)r $ then $\hat\nu$ realizes the topological entropy of $\upsigma_r.$
\end{obs*}

Let $K:\EE\times\R \to V$ be a $\<\hat r\>$-invariant H\"older-continuous function and consider the flow \begin{alignat*}{2}\psi=\big(\psi_t & :\susp_r \times V\to \susp_r \times V)_{t\in\R}\\  \psi_t\big(p,v\big) & =\big(\upsigma^r_t(p),v-\int_0^t K(\upsigma^r_sp)ds\big).\end{alignat*} 
\noindent

Consider the measure on $\Upsigma_r\times V$ defined by $\sus_\nu=\hat\nu\otimes\Leb.$

\begin{thm1}[{S. \cite[Theorem 3.8]{orbitalcounting}}]\label{mixing0} Assume the group generated by the periods of $(r,K)$ is dense in $\R\times V$ and that $\int Kd\hat\nu=0$ for the equilibrium state $\nu$ of $-h(\upsigma^r)r.$ Then there exists $\kappa>0$ such that given two compactly supported continuous functions $g_1,g_2:\Upsigma_r\times V\to\R,$ one has $$t^{\dim V/2}\sus_\nu\big(g_1\cdot g_2\circ\psi_t\big)\rightarrow \kappa \sus_\nu(g_1)\sus\nu(g_2),$$ as $t\to\infty.$
\end{thm1}

We include the main outline of its proof in Appendix \ref{mixingprueba}. As it is classical, the above result holds for characteristic functions of open bounded sets whose boundary has measure zero.

\begin{cor1} Under the same assumptions of Theorem \ref{mixing0}, if $\dim V=1$ then $\psi$ is ergodic w.r.t. $\sus_\nu.$ If $\dim V\geq 3$ then $\sus_\nu\big(\cal K(\psi)\big)=0.$
\end{cor1}

\begin{proof} The skew-product system $f^{\sf K}:\EE\times V\to \EE\times V$ of equation (\ref{trsl}), where $\sf K:\EE\to V$ is defined by \begin{equation}\label{sfK}\sf K(x)=\int_0^{r(x)}K(x,s)ds,\end{equation} is the first-return map of $\psi$ to its global section $\big(\EE\times\{0\}\big)/_\sim\times V.$ Consequently, following the flow-lines until reaching the section, one finds a natural (measurable) bijection between $\psi$-invariant subsets and $f$-invariant subsets. If $\mu$ is a $\sigma$-invariant probability measure on $\EE,$ then this bijection preserves the class of invariant-zero-sets between the measures $\hat\mu\otimes\Leb$ on $\susp_r\times V$ and $\mu\otimes\Leb$ on $\EE\times V.$ Thus we can translate ergodicity results from $f$ to the flow $\psi$ and vice-versa. 

The case $\dim V=1$ is hence settled by Proposition \ref{d=1}. 

For $\dim V\geq3$ one considers an open $A\subset\EE$ with $\nu(\partial A)=0,$ an open interval $I\subset\R$ with length $<\min r$ and $B\subset V$ an open ball. Applying Theorem \ref{mixing0} to $$\bar \BB=\bar\BB_1=\bar\BB_2=A\times I\times B$$ gives a positive $C$ such that for large $t$ one has $t^{\dim V/2}\sus\big(\1_{\bar\BB}\cdot\1_{\bar\BB}\circ\psi_t\big)\leq C.$ Thus, for a fixed $t_0>0$ one has that $$\int_{t_0}^\infty\sus\big(\1_{\bar\BB}\cdot\1_{\bar\BB}\circ\psi_t\big)dt\leq C\int_{t_0}^\infty \frac1{t^{\dim V/2}}.$$ 

If $\dim V\geq3$ then $\int_{t_0}^\infty\sus\big(\1_{\bar\BB}\cdot\1_{\bar\BB}\circ\psi_t\big)dt<\infty$ and Lemma \ref{dico} gives $\sus\big(\cal K(\psi)\big)=0,$ in particular the system is not ergodic.
\end{proof}

\begin{obs*} An ergodicity dichotomy for $f^{\sf K}$ has been previously established by Guivarc'h \cite[Corollaire 3 on page 443]{guivarch} under the stronger assumption that $\sf K$ is \emph{aperiodic}.
\end{obs*}

By means of Markov partitions (Theorem \ref{particiones}), the above corollary immediately translates in the following. As in the previous section, let $X$ be a compact metric space equipped with a topologically transitive, H\"older-continuous, metric-Anosov flow $\phi.$ Let $F:X\to V$ be H\"older-continuous and consider the flow $\Phi=\big(\Phi_t:X\times V\to X\times V\big)_{t\in\R}$ $$\Phi_t(p,v)=\big(\phi_tp,v-\int_0^s F(\phi_sp)ds\big).$$ It is convenient to call the flow $\Phi$ by \emph{the skew product of $\phi$ by $F$}.

\begin{cor1}[Dichotomy]\label{dicoflujos} Assume the group spanned by the periods of $(1,F)$ is dense in $\R\times V$ and that $\int Fdm=0$ for the measure of maximal entropy $m$ of $\phi.$ Then $\Phi$ is mixing as in Theorem \ref{mixing0}, moreover $$\dim V\leq1 \Rightarrow \Phi\textrm{ is ergodic w.r.t. }m\otimes\Leb\Rightarrow \dim V\leq2.$$ If $\dim V\geq3$ then $\cal K(\Phi)$ has zero measure.\end{cor1}

\begin{proof}Follows from the corresponding results for subshifts and Remark \ref{equientro} describing the measure of maximal entropy of $\phi.$
\end{proof}

\subsection{The critical hypersurface}\label{E=1} We recall here two results from Babillot-Ledra\-ppier \cite{babled}. Their paper concerns differentiable Anosov flows but, as one checks the proofs, only the existence of a Markov coding is required for both their results below. We take the liberty to state them in our broader generality and refer the reader to \emph{loc. cit.} whose proofs work verbatim. 

As before, let $F:X\to V$ be H\"older-continuous.

\begin{assu}\label{assuA}We will assume throughout the remainder of \S\,\ref{dynamics} that the closed group $\grupo$ spanned by the periods of $F$ has rank $\dim V,$ (i.e. $\grupo\simeq \R^k\times\Z^{\dim V-k}$ for some $k\in\lb0,\dim V\rb$) and that moreover the group spanned by $$\Big\{\big(p(\tau),\int_\tau F\big):\tau\textrm{ periodic}\Big\}$$ is isomorphic to $\R\times \grupo.$
\end{assu}

The compact convex subset of $V$ $$ \medidas^\phi(F)=\Big\{\int_XFd\mu:\mu\in\medidas^\phi\Big\}$$ has hence non-empty interior. On the other hand, for each $\varphi\in V^*$ one can consider the pressure of the function $\varphi(F):X\to\R$: $$\PP(\varphi)=P(\phi,-\varphi\circ F).$$ By Assumption \ref{assuA} the function $\PP:V^*\to\R$ is analytic and strictly convex (Proposition \ref{pder}). Using the formula for the derivative of pressure (Proposition \ref{pder}), and the natural identification $\grassman_{\dim V-1}(V^*)\to\P(V),$ one has, for $\varphi\in V^*$ that \begin{equation}\label{dpres}d_\varphi\PP=\int Fdm_{-\varphi(F)},\end{equation} where $m_{-\varphi(F)}$ is the equilibrium state of $-\varphi(F).$ One has the following.

\begin{prop1}[{Babillot-Ledrappier \cite[Prop. 1.1]{babled}}]\label{babled1} The map $\dual:V^*\to V$ defined by $\varphi\mapsto d_\varphi\PP$ is a diffeomorphism between $V^*$ and the interior of $\medidas^\phi(F).$\end{prop1}

Let us denote by $\cone_F=\R_+\cdot\medidas^\varphi(F)$ the closed cone generated by the periods of $F.$ If $0$ does not belong to $\medidas^\phi(F)$ then $\cone_F$ is a sharp cone (i.e. does not contain a hyperplane of $V$) and its interior is $$\inte\cone_F=\R_+\cdot\inte\big(\medidas^\phi(F)\big).$$ One has moreover the following.

\begin{prop1}[{Babillot-Ledrappier \cite[Prop. 3.1]{babled}}]\label{babled2} Assume $0\notin\medidas^\phi(F),$ then the set $\dual\big(\{\varphi\in V^*:\PP(\varphi)=0\}\big)$ generates the cone $\inte\cone_F.$ 
\end{prop1}

Observe that if $\varphi\in V^*$ is such that $\sum_{\tau\textrm{ periodic}}e^{-\ell_\tau(\varphi\circ F)}<\infty$ then the formula for pressure on Proposition \ref{pder} gives that $\PP(\varphi)\leq0.$ There exists then $s\in(0,1]$ such that $\PP(s\varphi)=0.$ The variational principle (equation (\ref{pressuredefi})) implies that $$\int_\tau\varphi(F)\geq0$$ for every periodic orbit $\tau,$ and thus Lemma \ref{positiva} applies to give that $\varphi(F)$ is Liv\v sic-cohomologous to a strictly positive function and $\h_{\varphi(F)}\in(0,1].$

Consequently $\varphi$ is strictly positive on the cone $\cone_F,$ i.e. $\varphi\in\inte\big(\cone_F^*\big).$ We are thus interested in the \emph{convergence domain} of $F$ \begin{alignat*}{2}\cal D_F & = \big\{\varphi\in V^*: \sum_{\tau\textrm{ periodic}} e^{-\ell_\tau(\varphi\circ F)}<\infty\big\}\\ & =\big\{\varphi\in\inte\big(\cone_F^*\big): \h_{\varphi(F)}\in(0,1]\big\}, \end{alignat*} and the \emph{critical hypersurface}\footnote{also usually called the \emph{entropy-one set} or \emph{the Manhattan curve},} (whose name is justified by the next corollary) \begin{alignat*}{2}\cal Q_F & =\big\{\varphi\in\inte\big(\cone_F^*\big): \h_{\varphi(F)}=1\big\}.\end{alignat*}

\begin{cor1}\label{E=1Cor} Assume $0\notin\medidas^\phi(F),$ then $\cal D_F=\PP^{-1}(-\infty,0)$ and $\cal Q_F=\PP^{-1}(0)$ Consequently $\cal D_F$ is a strictly convex set whose boundary coincides with $\cal Q_F.$ The latter is a closed analytic co-dimension-one submanifold of $V^*.$ The map $$\varphi\in\cal Q_F\mapsto \sf T_\varphi\cal Q_F$$ induces a diffeomorphism between $\cal Q_F$ and directions in the interior of the cone $\cone_F.$\end{cor1}

\begin{proof}We have already shown the inclusions $\PP^{-1}(-\infty,0)\subset\cal D_F$ and $\PP^{-1}(0)\subset\cal Q_F,$ the other ones follow at once from Lemmas \ref{positiva} and \ref{Abramov}. Since $0\notin\medidas^\varphi(F)$ Proposition \ref{babled1} implies that $\PP$ has no critical points, thus $\PP^{-1}(0)=\cal Q_F$ is an analytic sub-manifold of $V^*.$ Strict convexity follows from that of $\PP,$ and the last assertion follows by observing that the tangent space $\sf T_\varphi\cal Q_F$ equals $\ker d_\varphi\PP.$\end{proof}

We now focus on the variation of the critical hypersurface when $F$ varies. To this end consider the Banach space $\holder^\alpha(X,V)$ of $V$-valued H\"older continuous functions with exponent $\alpha.$ The pressure can be considered as an analytic map $\PP:\holder^\alpha(X,V)\times V^*\to\R$ defined as $\PP(G,\psi):=P(\phi,-\psi(G)).$ Its differential at the point $(F,\varphi)$ on the vector $(G,\psi)$ is $$d_{(F,\varphi)}\PP(G,\psi)=-\int\big(\psi(F)+\varphi(G)\big)dm_{-\varphi(F)},$$ and vanishes identically only if $(F,\varphi)=(0,0).$ The pre-image $\PP^{-1}(0)$ is thus a Banach-manifold. 

If $F\in\holder^\alpha(X,V)$ is such that $0\notin\medidas^\phi(F),$ then its critical hypersurface $$\cal Q_F=\big\{\varphi\in V^*:(F,\varphi)\in\PP^{-1}(0)\big\}$$ is the intersection of $\{F\}\times V^*$ with the level set $\PP^{-1}(0).$ This intersection will vary analytically on compact sets with $F$ as long as the tangent space $\ker d_{(F,\varphi)}\PP$ (for fixed $(F,\varphi)$ with $\PP(F,\varphi)=0$) is transverse to the vector space $\{0\}\times V^*.$ Since $\ker d_{(F,\varphi)}\PP$ has co-dimension 1, transversality is implied by $\ker d_{(F,\varphi)}\PP\cap\{0\}\times V^*$ being co-dimension $1$ on $V^*.$ However by Corollary \ref{E=1Cor}, this latter intersection is, as long as $F$ verifies assumption \ref{assuA} and $0\notin\medidas^\phi(F),$ the tangnet space $\sf T_\varphi\cal Q_F,$ which has co-dimension $1.$ We have thus established the following.

\begin{cor1}\label{Qanaly}The critical hypersurface $\cal Q_F$ varies analytically on compact sets when varying the function $F$ among H\"older functions verifying the hypothesis of Corollary \ref{E=1Cor}.
\end{cor1}

\subsection{Dynamical Intersection and the critical hypersurface}\label{DynIn} We recall here a concept from Bridgeman-Canary-Labourie-S.  \cite{pressure} (see also references therein for related concepts). Let $f:X\to\R_+$ be a positive H\"older-continuous function and, for $t>0,$ consider the finite set $\sf R_t(f)=\big\{\tau\textrm{ periodic}:\ell_\tau(f)\leq t\big\}.$ Let $g:X\to\R$ be H\"older-continuous (but not necessarily positive), then the \emph{dynamical intersection} between $f$ and $g$ is defined by $$\II(f,g)=\lim_{t\to\infty}\frac1{\#\sf R_t(f)}\sum_{\tau\in \sf R_t(f)}\frac{\ell_\tau(g)}{\ell_\tau(f)}.$$ Then one has the following.

\begin{prop1}[{\cite[\S\,3.4]{pressure}}]\label{ineq} One has $$\II(f,g)=\frac{\int gdm_{-\h_ff}}{\int fdm_{-\h_ff}},$$ in particular $\II$ is well defined and varies analytically with $f$ and $g$ among H\"older-continuous functions with fixed H\"older exponent. If $g$ is moreover positive then one has $\II(f,g)\geq \h_f/\h_g.$ \end{prop1}

We now place ourselves in the context of the previous subsection, i.e. we consider a H\"older-continuous $F:X\to V$ and we assume moreover that $0\notin\medidas^\phi(F).$ For $\varphi\in\cal Q_F$ we consider the map $\II_{\varphi}:V^*\to\R$ defined by \begin{equation}\label{IIQ}\II_{\varphi}(\psi):= \II(\varphi(F),\psi(F))=\frac{\int \psi(F)dm_{-\varphi(F)}}{\int\varphi(F)dm_{-\varphi(F)}},\end{equation} where the last equality comes from Proposition \ref{ineq} and the fact that $\h_{\varphi(F)}=1$. Observe that it is a linear map. We then have the following explicit interpretation of the tangent space to the critical hypersurface purely in terms of periods.

\begin{cor1}\label{TI}Let $F:X\to V$ be as in \S\,\ref{E=1} and such that $0\notin\medidas^\phi(F).$ Then for $\varphi\in\cal Q_F$ one has $\sf T_\varphi\cal Q_F=\ker\II_\varphi.$
\end{cor1}

\begin{proof} Since by Corollary \ref{E=1Cor} one has $\cal Q_F=\PP^{-1}(0),$ the tangent space $\sf T_\varphi\cal Q_F=\ker d_\varphi\PP=\ker\II_\varphi,$ where the last equality comes from the combination of Equations \eqref{dpres} and \eqref{IIQ}.\end{proof}

\section{A Ledrappier correspondence}\label{piedrapie}

In \cite{ledrappier} Ledrappier establishes, for a closed negatively-curved manifold $M,$ bijections between Liv\v sic-cohomology classes of \emph{pressure zero functions} on $\sf T^1M,$ \emph{normalized H\"older cocycles} for the action of $\pi_1M$ on the visual boundary $\bord_\infty\tilde M$ of the universal cover of $M,$ \emph{quasi-invariant measures} on $\bord_\infty\tilde M$, among other objects. The purpose of this section is to establish, in the context of word-hyperbolic groups with metric-Anosov geodesic flow, some of these correspondences. We also extend results from S. \cite{exponential,quantitative,orbitalcounting} to this setting. Some ideas from Bridgeman-Canary-Labourie-S. \cite{pressure,Liouvillepressure} and Carvajales \cite[Appendix]{lyon2} are used. The reader can also check Paulin-Pollicott-Schapira \cite{PPS} for similar results in situations allowing cusps.

Let $\G$ be a finitely generated, non-elementary, word-hyperbolic group (see Ghys-de la Harpe \cite{ghysharpe} for a definition). Denote by $\sf g=\big(\sf g_t:\sf U\G\to\sf U\G\big)_{t\in\R}$ the \emph{Gromov-Mineyev geodesic} flow of $\G$ (see Gromov \cite{gromov} and Mineyev \cite{mineyev}). The total space $\sf U\G$ is the quotient of $\bord^2\G\times\R$ by a properly discontinuous co-compact $\G$-action (defined in \emph{loc. cit.}). This action restricted to $\bord^2\G$ coincides with the induced $\G$-action on its Gromov boundary, and commutes with the $\R$-action by translations, giving on the quotient the desired flow $\sf g.$ We will save the notation $$\widetilde{\sf U\G}$$ for the pair consisting on the space $\bord^2\G\times\R$ equipped with the above $\G$-action.

\begin{assu}\label{assuB}We will assume throughout \S\,\ref{piedrapie} that $\sf g$ is metric-Anosov (recall that in general $\sf g$ is transitive (see Remark \ref{tran})) and that the lamination induced on the quotient by $\widetilde{\cal W}^{cu}=\{(x,\cdot,\cdot)\in\widetilde{\sf U\G}\}$ is the central-unstable lamination of $\sf g.$\end{assu}

Let us emphasize that, in what follows, the Gromov-Mineyev geodesic flow is merely an auxiliary object. The whole discussion works verbatim replacing $\sf g$ by the following flows known to satisfy Assumption \ref{assuB}: \begin{itemize}\item[-] the non-wandering set of the geodesic flow of a convex co-compact action of $\G$ on a $\mathrm{CAT}(-1)$ space (if this is known to exist), see Constantine-Lafont-Thompson \cite{Lafont}, \item[-] the geodesic flow of a projective-Anosov representation of $\rho$ (again if this is known to exist) as introduced in Bridgeman-Canary-Labourie-S. \cite{pressure}, see also \S\,\ref{flujoBCLS}.\end{itemize}

Recall that every \emph{hyperbolic element}\footnote{i.e. an infinite order element} $\g\in\G$ has two fixed points on $\bord\G,$ the attracting $\g_+$ and the repelling $\g_-.$ If $x\in\bord\G-\{\g_-\}$ then $\g^nx\to\g_+$ as $n\to\infty.$ The axis $(\g_-,\g_+)\times\R$ projects then to a periodic orbit of $\sf g,$ denoted by $[\g].$ If $l(\g)$ denotes the translation length of $\g$ along this axis, then $l(\g)$ is an integer multiple of the period of the periodic orbit $[\g].$ Observe we allow $[\g]$ to tour several times along the orbits it surjects to, so at least formally, we let $[\g^n]$ be the orbit $[\g]$ toured $n$-times.

\begin{obs*}\label{tran}We briefly justify why $\sf g$ is transitive. It suffices to show that given two open sets $U, V$ there exists $t\in\R$ s.t. $\sf g_t(U)\cap V\neq\emptyset,$ so the question is reduced to the same question for the action of $\G$ on $\bord^2\G$; the open sets to be considered can be reduced to be of the form $U_1\times U_2$ and $V_1\times V_2,$ where $U_i,V_i\subset\bord\G$ are open and $\overline{U_1}\cap U_2=\overline{V_1}\cap V_2=\emptyset;$ an element $\g\in\G$ with $\g_-\in U_2$ and $\g_+\in V_1$ has a positive power such that $\g^n(U_1\times U_2)\cap V_1\times V_2\neq\emptyset.$ 
\end{obs*}

\subsection{The Ledrappier potential of a H\"older cocycle}\label{LedrappierPot}

Let $V$ be a finite dimensional real vector space. A \emph{H\"older cocycle} is a function $c:\G\times\bord^2\G\to V$ such that: \begin{itemize}\item[-] for all $\g,h\in \G$ one has $c\big(\g h,(x,y)\big)=c\big(h,(x,y)\big)+c\big(\g,h(x,y)\big),$ \item[-]there exists $\alpha\in(0,1]$ such that for every $\g\in\G$ the map $c(\g,\cdot)$ is $\alpha$-H\"older continuous.\end{itemize}

The \emph{period} of a H\"older cocycle for a hyperbolic $\g\in\G$ is $\ell_c(\g):=c\big(\g,(\g^-,\g^+)\big).$ Two cocycles $c,c'$ are \emph{cohomologous} if there exists a H\"older-continuous function $U:\bord^2\G\to V$ such that for all $\g\in\G$ and $(x,y)\in\bord^2\G$ one has $$c(\g,(x,y))-c'(\g,(x,y))=U(\g(x,y))-U(x,y).$$ Two cohomolgous cocycles have the same marked spectrum $\g\mapsto\ell_c(\g).$ The following should be compared with Ledrappier \cite[Th\'eor\`eme 3]{ledrappier}.

\begin{prop1}\label{led1} For every H\"older cocycle $c$ there exists a H\"older-continuous function $\LL_c:\sf U\G\to V$ such that for every hyperbolic $\g\in\G$ one has $$\int_{[\g]}\LL_c=\ell_c(\g).$$ Cohomologous cocycles induce Liv\v sic-cohomolgous functions. 
\end{prop1}

\begin{proof} The general case follows from the case $V=\R$ by Riesz representation Theorem. Assume thus $V=\R$ and consider the trivial line bundle $\widetilde{\sf U\G}\times\R$ equipped with the bundle automorphisms $$\g\cdot(p,s):=\big(\g p,e^{-c\big(\g,(x,y)\big)}s\big),$$ where $p=(x,y,t).$ Denote by $\sf F\to \sf U\G$ the quotient line bundle. It is equipped with a flow $\big(\hat {\sf g}_t:\sf F\to\sf F\big)_{t\in\R}$ by bundle automorphisms, induced on the quotient by $$t\cdot(p,s)=(\sf g_tp,s).$$ Let $|\,|$ be a Euclidean metric on $\sf F$ and define, for $v\in\sf F_p,$ \begin{equation}\label{de}\cal T(p,t)=\log\frac{|\hat{\sf g}_tv|}{|v|}.\end{equation} It is a translation cocycle over $\sf g,$ indeed since $\sf F_p$ is one dimensional the choice of $v$ does not matter, and since $\hat{\sf g}$ is a flow one has $$\log\frac{|\hat{\sf g}_{t+s}(v)|}{|v|}= \log\frac{|\hat{\sf g}_{t}(\hat{\sf g}_sv)|}{|v|}\frac{|\hat{\sf g}_s(v)|}{|\hat{\sf g}_s(v)|}=\log\frac{|\hat {\sf g}_{t}(\hat{\sf g}_sv)|}{|\hat{\sf g}_s(v)|}+\log\frac{|\hat{\sf g}_s(v)|}{|v|}.$$
 
By Corollary \ref{graciadivina}, there exists a H\"older-continuous function $\LL_c:\sf U\G\to\R$ such that $\cal T$ and $k_{\LL_c}$ are Liv\v sic-cohomologous. We end the proof by a period computation. For every hyperbolic $\g\in\G$ one has, for all $s\in\R,$ that $\g(\g_-,\g_+,s)=(\g_-,\g_+,e^{-\ell_c(\g)}s),$ or equivalently, for any $x\in[\g]\subset\sf U \G$ and $v\in \sf F_x,$ $$\ell_c(\g)=\log\frac{|\hat{\sf g}_{p(\g)}v|}{|v|}=\ell_{[\g]}(\cal T)=\int_{[\g]}\LL_c,$$ where $p(\g)$ is the period of $[\g]$ for $\sf g.$ Since cohomologous cocycles have the same marked spectrum, the associated functions have the same periods and are thus Liv\v sic-cohomolgous by Theorem \ref{livsic3}.\end{proof}

\begin{defi1}We say that $\LL_c$ is \emph{a Ledrappier potential} of $c$ over $\sf g.$\end{defi1}

\subsection{Real cocycles and reparametrizations}\label{refraction}

Assume now $V=\R$ and consider a cocycle $c$ with non-negative (and not all vanishing) periods. Define its \emph{entropy} by $$\h_c=\limsup_{t\to\infty}\frac1t\log\#\big\{[\g]\in[\G]\textrm{ hyperbolic}:\ell_c(\g)\leq t\big\}.$$

\begin{obs*} It follows from Proposition \ref{led1} and Remark \ref{h>0} that $\h_c>0.$
\end{obs*}

%\begin{proof} One has $\h_c=\h_{\LL_c}$ and thus \S\,\ref{anosovtop} gives $\h_c>0.$\end{proof}

For such a cocycle consider the action of $\G$ on $\bord^2\G\times\R$ via $c$: \begin{equation}\label{gaction}\g\cdot(x,y,t)=\Big(\g x,\g y,t-c\big(\g,(x,y)\big)\Big).\end{equation} Let us denote by $\RR c$ the quotient space $\RR c=\G\/\big(\bord^2\G\times\R\big).$ The following can be found in S. \cite{quantitative} for fundamental groups of closed negatively-curved manifolds and in Carvajales \cite{lyon2} for the refraction cocycle of a $\simple$-Anosov representation (see Definition \ref{refr}).

\begin{thm1}\label{refractionThm} If $c$ is a H\"older cocycle with non-negative periods and finite entropy, then its Ledrappier potential is Liv\v sic-cohomologous to a strictly positive function. Moreover, the above action of $\G$ on $\bord^2\G\times\R$ is properly-discontinuous and co-compact and the flow $\phi^c=\big(\phi^c_t:\RR c\to \RR c\big)_{t\in\R}$ induced on the quotient by the $\R$-translation flow is H\"older-conjugated to the reparametrization of $\sf g$ by $\LL_c.$
\end{thm1}

The topological entropy of $\phi^c$ is thus\footnote{When $\G$ has torsion elements this fact requires some work, see Carvajales-Dai-Pozzetti-Wienhard \cite[Section 5]{CDPW} for details.} $\h_c.$

\begin{proof} The first assertion follows at once from Lemma \ref{positiva}. For the remaining statements, we continue as in the proof of the proposition but for the cocycle $-c.$ Observe first that $-\LL_c=\LL_{-c}.$ Since $\cal T$ and $k_{\LL_{-c}}$ are Liv\v sic-cohomologous, there exists $U:\sf U\G\to\R$ such that for all $t\in\R,$ $p\in\sf U\G$ and $v\in\sf F_p$ one has (recall equation (\ref{de})) $$\log\frac{|\hat{\sf g}_tv|}{|v|}-\int_0^t\LL_{-c}=U(\sf g_tp)-U(x).$$

\noindent Since $\LL_{-c}=-\LL_c$ is Liv\v sic-cohomologous to a strictly negative function, the above equation implies that the flow $\hat{\sf g}$ is contracting on $\sf F,$ i.e. there exist positive $C$ and $\mu$ such that for all $v\in \sf F$ and $t\in \R$ one has $$|\hat{\sf g}_tv|\leq Ce^{-\mu t}|v|.$$ A standard procedure (see for example Katok-Hasselblat \cite{katokh} or Bridgeman-Canary-Labourie-S. \cite[Lemma 4.3]{pressure}) provides a Euclidean metric $\|\,\|$ on $\sf F$ such that the constant $C$ equals $1.$ We denote also by $\|\,\|$ the lift of this metric to $\widetilde{\sf U\G}\times \R,$ it is a $\G$-invariant family.

Given then $(x,y,t)\in\widetilde{\sf U\G}$ we let $v_{(x,y,t)}\in(\R-\{0\})/\pm$ be such that $\|v_{(x,y,t)}\|=1.$ As in \cite[Proposition 4.2]{pressure} the map \begin{alignat}{3} \label{mapa} \widetilde{\sf U \G} & \to \bord^{2}\G\times (\R-\{0\}/\pm)  \to  \bord^{2}\G\times\R & \nonumber\\ \U: (x,y,t) & \mapsto \big(x,y,v_{(x,y,t)}\big)  \mapsto \big(x,y,\log v_{(x,y,t)}\big), & \end{alignat} is $\G$-equivariant and an orbit equivalence between de $\R$-actions. \end{proof}

\begin{defi1} The flow $\phi^c$ will be called the \emph{refraction flow} of $c.$
\end{defi1}

\subsection{Patterson-Sullivan measures}\label{PattersonC}

Let us consider now H\"older cocycles with $V=\R$ and only depending on one variable, i.e. $c:\G\times\bord\G\to\R.$ Assume moreover that $c$ has non-negative periods and finite entropy. A cocycle $\bar c:\G\times\bord\G\to\R$ is \emph{dual to $c$} if for every hyperbolic $\g\in\G$ one has $$\ell_{\bar c}(\g)=\ell_c\big(\g^{-1}\big).$$

\begin{defi1}\item\begin{itemize}\item[-]A \emph{Patterson-Sullivan measure for $c$ of exponent} $\delta\in\R_+$ is a probability measure $\mu$ on $\bord\G$ such that for every $\g\in\G$ one has \begin{equation}\label{defP}\frac{d \g_*\mu}{d\mu}(\cdot)=e^{-\delta c\big(\g^{-1},\cdot\big)}.\end{equation} \item[-]A \emph{Gromov product} for the ordered pair $(\bar c, c)$ is a function $[\cdot,\cdot]:\bord^2\G\to\R$ such that for all $\g\in\G$ and $(x,y)\in\bord^2\G$ one has $$[\g x,\g y]-[x,y]=-\big(\bar c(\g,x)+c(\g,y)\big).$$\end{itemize}
\end{defi1}

Consider a pair of dual cocycles $(\bar c,c)$ and assume a Patterson-Sullivan measure of exponent $\delta$ exists for each $c$ and $\bar c,$ denoted by $\mu$ and $\bar \mu$ respectively. Assume moreover that a Gromov product for the pair $(\bar c, c)$ exists. The measure \begin{equation}\label{tildem}\widetilde{\scr m}= e^{-\delta[\cdot,\cdot]}\bar\mu\otimes \mu\otimes dt\end{equation}

\noindent
on $\bord^2\G\times \R$ is hence $\G$-invariant\footnote{the action being via $c,$ $\g(x,y,t)=(\g x,\g y,-c(\g,y))$} and $\R$-invariant. Passing to the quotient one obtains a measure $\scr m$ on $\RR c$ invariant under the flow $\phi^c.$ Observe that we can write $\widetilde{\scr m}$ as $$d\widetilde{\scr m}(x,y,t)=\int_{\bord\G\times\R}e^{-\delta t}d\mu(y) dt\left(\int_{\bord\G} e^{\delta t}e^{-\delta[x,y]}d\bar\mu(x)\right).$$ The measure $e^{-\delta t}d\mu dt$ is $\G$-invariant\footnote{Indeed, if $f:\bord\G\times\R\to\R$ is continuous then \begin{alignat*}{3}\int f\big(\g(y,t)\big)e^{-\delta t} d\mu dt,&=\int f\big(\g y,t-c(\g,y)\big)e^{-\delta t} d\mu dt & \\ & =\int f(\g y,t)e^{-\delta \big(t-c(\g,y)\big)} d\mu dt & \qquad\textrm{(by translation invariance)}\\ & =\int f( y,t)e^{-\delta \big(t-c(\g,\g^{-1}y)\big)}e^{-\delta c(\g^{-1},y)} d\mu dt  & \qquad\textrm{(by definition of $\mu$)}\\ & =\int f(y,t)e^{-\delta t}d\mu dt.& \end{alignat*}} on $\bord\G\times\R$ and the family $$m_{(y,t)}=e^{\delta t}e^{-\delta[\cdot,y]}d\bar\mu$$ is $\G$-equivariant. This decomposition induces then a local decomposition of $\scr m$ as a product of measures on the laminations induced on the quotient by $\widetilde{\cal W}^{cu}=\big\{(x,\cdot,\cdot):x\in\bord\G\}$ and $\widetilde{\cal V}=\big\{(\cdot,y,t):y\in\bord\G,\, t\in\R\big\}.$ The local product structure induced by $\bord\G\times\bord\G\times\R$ permits to transport the measures $m_{(y,t)},$ parallel to the central stable leaf $W^{\cu}[(x,y,t)],$ to the stable leaf of $[(x,y,t)]$ for $\phi^c$ to obtain a family of measures at each strong stable leaf $\nu^{\ss}_p$ such that \begin{equation}\label{empuje}\frac{d\big(\phi^c_{-t}\big)_*\nu^{\ss}_p}{d\nu^{\ss}_{\phi^c_tp}}=e^{-\delta t}.\end{equation} Margulis's description of the measure maximizing entropy (\S\,\ref{mm}) then gives that $\delta=\h_c$ and that $\scr m$ maximizes entropy for $\phi^c.$ Thus, subject to the existence of the Patterson-Sullivan measures and the Gromov product, we summarize the above discussion in the following.

\begin{prop1}\label{entropiamax} The measure $\scr m$ on $\RR c$ maximizes entropy for the flow $\phi^c.$ The exponent $\delta$ necessarily equals the topological entropy of $\phi^c,$ $\h_c,$ if $\nu$ is another Patterson-Sullivan measure for $c$ then $\nu=\mu.$ \end{prop1}

Let us consider again the measure $\widetilde{\scr m}$ on $\bord^2\G\times\R$ from (\ref{tildem}) but let us instead study the $\G$-action on the $\R$-coordinate by the Gromov-Mineyev cocycle, so that $\G\/\big(\bord^2\G\times\R\big)=\sf U\G$ and the induced flow is $\sf g.$  The quotient measure, $\scr a,$ is $\sf g$-invariant and the orbit equivalence $\U$ from equation (\ref{mapa}) preserves, by the way it is defined, zero flow-invariant sets between $\scr a$ and $\scr m.$ Since $\U$ is a conjugation between $\sf g^{\LL_c}$ and $\phi^c,$ and $\scr m$ maximizes entropy for $\phi^c,$ the measure $\U_*\scr m$ is $\sf g^{\LL_c}$-invariant, maximizes entropy for $\sf g^{\LL_c}$ and has the same zero sets as $\scr a.$ 

One concludes that the Abramov transform (\ref{abramovdef}) ${\scr a}^\#$ (w.r.t. $\LL_c$) is a measure of maximal entropy of the flow $\sf g^{\LL_c}.$ Lemma \ref{Abramov} implies then that $\scr a/|\scr a|$ is the equilibrium state of $-\h_c\LL_c.$ Let us summarize in the following remark:

\begin{obs*}\label{equiA}The probability measure $\scr a/|\scr a|$ on $\sf U\G=\G\/\widetilde{\sf U\G}$ induced by $\widetilde{\scr m}$ is the equilibrium state of $-\h_c\LL_c.$
\end{obs*}

Since the zero sets of an equilibrium state are uniquely determined by the Liv\v sic-cohomology class of the associated potential up to an additive constant (Theorem \ref{equilibrios}), one concludes the following:

\begin{cor1}\label{unicidad} Let $\bar\kappa,\kappa$ be a pair of dual cocycles with non-negative periods and finite entropy. Assume Patterson-Sullivan measures, $\nu$ and $\bar \nu,$ exist for $\kappa$ and $\bar\kappa$ respectively, together with a Gromov product for the pair. If $\nu$ has the same zero sets as $\mu$ then the (scaled) Ledrappier potentials $\h_c\LL_c$ and $\h_\kappa\LL_\kappa$ are Liv\v sic-cohomologous and $\nu=\mu.$
\end{cor1}

\begin{proof} A final argument is required, indeed from \S\,\ref{equilibrios} there exists a constant $K$ such that $\h_c\LL_c$ and  $\h_\kappa\LL_\kappa+K$ are Liv\v sic-cohomologous. However, since $\h_c$ is the topological entropy of $\sf g^{\LL_c},$ Lemma \ref{Abramov} gives $$0=P(\sf g,-\h_c\LL_c)=P(\sf g,-\h_\kappa\LL_\kappa+K)=K,$$ where the second equality holds by Remark \ref{LivP}, and the third equality by the definition of $P$ (equation (\ref{pressuredefi})).\end{proof}

The above corollary can be found in Ledrappier \cite{ledrappier} when $\G$ is the fundamental group of a negatively curved closed manifold. The proof uses also a disintegration argument. One may also check Kaimanovich \cite{kaimanovichPSBM} and Babillot's survey \cite{babillotSurvey} specifically for the Buseman cocycle ($\G$ as in Ledrappier's aforementioned situation), and Carvajales \cite[Appendix]{lyon2} for the refraction cocycle $\bb\varphi$ (see \S\,\ref{arbitraryG} for the definition) of a $\simple$-Anosov representation of an arbitrary word-hyperbolic group.

\subsection{Vector-valued cocycles I: the critical hypersurface}\label{entropy1}\label{directions} 

Let now $c:\G\times\bord\G\to V$ be a H\"older cocycle and consider the compact convex set $\medidas^\sf g(\LL_c)\subset V.$ Since this set depends on the base flow $\sf g$ it is natural to consider the \emph{limit cone} of $c$ $$\cone_c=\overline{\bigcup_{\g\in\G}\R_+\cdot\ell_c(\g)}=\R_+\cdot\medidas^\sf g(\LL_c).$$

\begin{obs*}\label{reparam} Up to Liv\v sic -cohomology we can assume that $\LL_c$ has values in the vector space $V'=\<\medidas^{\sf g}(\LL_c)\>.$ We can moreover choose a reparametrization $\sf g^f$ of $\sf g$ so that if $\LL_c^{\sf g^f}:\sf U\G\to V'$ is the Ledrappier potential for $c$ over $\sf g^f$ then the flow $\sf g^f$ together with the potential $\LL_c^{\sf g^f}$ verify Assumption \ref{assuA} from \S\,\ref{E=1}.\end{obs*}

\begin{proof} By Remark \ref{positiva} the space of Liv\v sic-cohomology classes over $\sf g$ is infinite dimensional, so the remark follows.\end{proof}

We will work from now on with flow $\sf g^f$ given by the remark and the Ledrappier potential $\LL_c^{\sf g^f},$ we will rename these by $\sf g$ and $\LL_c$ though as to not overcharge the paper with notation and keep in mind that, when we restrict the image of $\LL_c$ to $V',$  Assumption \ref{assuA} is verified.

Let $\cdc=\{\psi\in V^*:\psi|\cone_c\geq0\}$ be the \emph{dual cone} of $c.$ For $\psi\in V^*$ denote by $c_\psi=\psi\circ c:\G\times\bord\G\to\R$ and $\h_\psi=\h_{c_\psi}.$

\begin{assu}\label{assuC}There exists $\psi\in \cdc$ such that $c_\psi$ has finite entropy.
\end{assu}

\begin{lema1}\label{h<infty} In this case $0\notin\medidas^{\sf g}(\LL_c),$ $\cdc$ has non-empty interior and $\inte\cdc=\big\{\varphi\in\cdc:\h_\varphi<\infty\big\}.$
\end{lema1}

\begin{proof} The Lemma follows essentially from \S\,\ref{LedrappierPot} and Lemma \ref{positiva}, indeed since $\h_{\psi}=\h_{\psi(\LL_c)}<\infty$ there exists $\kappa>0$ such that for all hyperbolic $\g\in\G$ it holds $\psi\big(\ell_c(\g)/p(\g)\big)>\kappa;$ by density of periodic orbits on $\medidas^{\sf g}$ one has $\inf \big\{\psi\big(\medidas^{\sf g}(\LL_c)\big)\big\}>\kappa>0.$ The remaining statements follows similarly.\end{proof}

Since $0\notin\medidas^{\phi^{c_\psi}}(\LL_c)$ we can apply Corollary \ref{E=1}. Denote by \begin{alignat*}{2}\cal Q_c  & =\Big\{\varphi\in\inte\cdc:\h_\varphi=1\Big\}\\ \cal D_c & = \Big\{\varphi\in\inte\cdc:\h_\varphi\in(0,1)\Big\}=\Big\{\varphi\in V^*:\sum_{[\g]\in[\G]}e^{-\ell_{c_\varphi}(\g)}<\infty\Big\}\end{alignat*} respectively the \emph{critical hypersurface} and the \emph{convergence domain} of $c.$

Since we haven't required the cone $\cone_c$ to have non-empty interior, consider its annihilation space $$\ann(\cone_c)=\{\psi\in V^*:\cone_c\subset\ker\psi\}.$$ If $\varphi\in\inte\cdc$ and $\psi\in\ann(\cone_c)$ then the potentials $\varphi(\LL_c)$ and $(\varphi+\psi)(\LL_c)$ are Liv\v sic-cohomologous. Let $\pi^c:V^*\to V^*/\ann(\cone_c)$ be the quotient projection.

We also import the concept of dynamical intersection of \S\,\ref{DynIn} to this setting using the Ledrappier potential of $c.$ For $\varphi\in\cal Q_c$ define the \emph{dynamical intersection map} associated to $c$ by  $\II_\varphi=\II^c_\varphi:V^*\to\R$ be defined by $$\II_\varphi(\psi)=\II\big(\varphi(\LL_c),\psi(\LL_c)\big).$$ By definition $\II\big(\varphi(\LL_c),\psi(\LL_c)\big)$ only depends on the Liv\v sic-cohomology classes of $\varphi(\LL_c)$ and $\psi(\LL_c),$ so we may  freely consider $\II$ as defined on $\cal Q_c\times V^*$ or on $\pi^c(\cal Q_c)\times V^*/\ann(\cone_c).$ One has the following:

\begin{cor1} Under Assumption \emph{\ref{assuC}} one has that $\pi^c(\cal D_c)$ is a strictly convex set whose boundary is $\pi^c(\cal Q_c).$  The latter is an analytic co-dimension-one sub-manifold. The map $\sf u:\pi^c(\cal Q_c)\to\P\big(\spa\{\cone_c\}\big)$ defined by $$\varphi\mapsto\sf u_\varphi:=\sf T_\varphi\pi^c(\cal Q_c)=\ker\II_\varphi$$ is an analytic diffeomorphism between $\pi^c(\cal Q_c)$ and directions in the relative interior of $\cone_c.$\end{cor1}

\begin{proof}By Remark \ref{reparam} we can apply \S\,\ref{E=1} to $\LL_c,$ the equality $\sf T_{\varphi}\cal Q_c=\ker\II_\varphi$ follows from Corollary \ref{TI}.\end{proof}

\subsection{Vector-valued cocycles II: skew-product structure}\label{fibered}\label{tutti2}

We remain in the situation of \S\,\ref{entropy1}, i.e. we keep Assumption \ref{assuC}. It follows at once from Theorem \ref{refraction} that the $\G$-action $\bord^2\G\times V$ $$\g(x,y,v)=\big(\g x,\g y,v-c(\g,y)\big)$$ is properly discontinuous. We aim to give a description of the $V$-action on the quotient space $\cjto$ space. 

By Lemma \ref{h<infty} and Theorem \ref{refractionThm} we can, for every $\varphi\in\inte\cdc,$  consider the refraction flow $\phi^{c_\varphi}=\big(\phi^{c_\varphi}_t:\RR{c_\varphi}\to\RR{c_\varphi}\big)_{t\in\R}.$ Such a $\varphi$ is fixed from now on.
 
\begin{obs*} Let us  still denote by $\LL_c$ the Ledrappier potential for $c$ over the flow $\phi^{c_\varphi},$ i.e. for every hyperbolic $\g\in\G$ one has $$\int_{[\g]}\LL_c=\ell_c(\g)\in V.$$ \end{obs*}

Let $\scr p$ be the probability measure of maximal entropy of $\phi^{c_\varphi}.$ The \emph{growth direction of $\varphi$} is the line of $V$ \begin{equation}\label{diru}\sf u_\varphi=\sf T_\varphi\cal Q_c= \R\cdot\int_{\RR{c_\varphi}}\LL_c d\scr p\end{equation} (the last equality follows from equation (\ref{dpres}) and Corollary \ref{E=1Cor}). Consider also the projection $\pi^\varphi:V\to\ker\varphi$ parallel to $\sf u_\varphi$ and denote by $\LL_c^\varphi:\RR{c_\varphi}\to\ker\varphi$ the composition $\LL_c^\varphi=\pi^\varphi\circ\LL_c.$ Observe that \begin{equation}\label{=0} \int_{\RR{c_\varphi}} \LL_c^\varphi d\scr p=0.\end{equation} Fix $u_\varphi\in\sf u_\varphi$ with $\varphi(u_\varphi)=1$ and define the \emph{directional flow} $\big(\df^\varphi_t:\G\/\big(\bord^2\G\times V\big)\to\G\/\big(\bord^2\G\times V\big)\big)_{t\in\R}$ by the induced on the quotient of $$t\cdot(x,y,v)=(x,y,v-tu_\varphi).$$

\begin{prop1}[S. \cite{orbitalcounting}]\label{enderezarV1} There exist a (bi)-H\"older-continuous homeomorphism $$\overline E:\cjto\to \RR{c_\varphi}\times\ker \varphi,$$ commuting with the $\ker\varphi$ action, that conjugates the flow $\df^\varphi$ with $\Phi^\varphi=\big(\Phi_t^\varphi:\RR{c_\varphi}\times\ker\varphi\to\RR{c_\varphi}\times\ker\varphi\big)_{t\in\R}$ $${\displaystyle\Phi_t^\varphi(p,v):=\Big(\phi^{c_\varphi}_t(p),v-\int_0^t\LL_c^\varphi(\phi^{c_\varphi}_s p)ds\Big).}$$
\end{prop1}

\begin{proof} This is the first item of S. \cite[Proposition 3.5]{orbitalcounting} when $\G $ is the fundamental group of a closed negatively curved manifold. The proof adapts \emph{mutatis mutandis} once \S\,\ref{refraction} is stablished. \end{proof}

Let us moreover place ourselves in the existence assumptions of subsection \ref{PattersonC} for the cocycle $c_\varphi,$ i.e. assume there exists:\begin{itemize}\item[-] a dual cocycle $\bar c_\varphi$ together with a Gromov product $[\cdot,\cdot]$ for the pair $(\bar c_\varphi,c_\varphi),$ \item[-] a Patterson-Sullivan measure for each $c_\varphi$ and $\bar c_\varphi,$ denoted by $\mu$ and $\bar\mu$ respectively. Recall from Proposition \ref{entropiamax} that necessarily the exponent of both $\mu$ and $\bar\mu$ is $\h_{c_\varphi},$ the topological entropy of $\phi^{c_\varphi}.$ \end{itemize}

By Proposition \ref{entropiamax} the measure $\scr m$ maximizes entropy for $\phi^{c_\varphi},$ one has then $\scr p=\scr m/|\scr m|.$ Consider the $\varphi$-\emph{Bowen-Margulis} measure $\BM^\varphi$ on $\G\/\big(\bord^2\G\times V\big)$ defined as the induced on the quotient by $$e^{-\h_{c_\varphi}[\cdot,\cdot]}\bar\mu\otimes\mu\otimes \Leb_{V},$$ for a Lebesgue measure on $V.$ One has the following result.

\begin{prop1}[S. \cite{orbitalcounting}]\label{enderezarV2} The (bi)-H\"older-continuous homeomorphism from Proposition \ref{enderezarV1} is a measurable isomorphism between $\BM^\varphi$ and $\scr m\otimes \Leb_{\ker\varphi}.$\end{prop1}

\begin{proof} This follows again as in S. \cite[Proposition 3.5]{orbitalcounting} once Proposition \ref{entropiamax} is established. \end{proof}

\subsection{Vector-valued cocycles III: Dynamical consequences}\label{ergoDic} 
We remain in the situation of \S\,\ref{entropy1}. Let us say that $c$ is \emph{non-arithmetic} if the periods of $c$ span a dense subgroup in $V.$

One concludes at once the following consequences:

\begin{cor1}[Ergodicity dichotomy]\label{CorergoDic} Assume $c$ is non-arithmetic, then $\df^\varphi$ is mixing as in Theorem \ref{mixing0}, consequently if $\dim V\geq4$ then $\cal K(\df^\varphi)$ has zero $\BM^\varphi$-measure. If  $\dim V\leq2$ then the directional flow $\df^\varphi$ is $\BM^\varphi$-ergodic.
\end{cor1}

\begin{proof} By Proposition \ref{enderezarV1} the flow $\df^\varphi$ is H\"older-conjugated to the skew-product of $\phi^{c_\varphi}$ with the Ledrappier potential $\LL_c^\varphi,$ the fiber being $\ker\varphi$ and thus $\dim V-1$-dimensional. Proposition \ref{enderezarV2} describes the desired measures in terms of the skew-product structure; non-arithmeticity of $c$ and equation (\ref{=0}) allow us to apply Corollary \ref{dicoflujos}.\end{proof}

\section{Algebraic semi-simple groups over a local field}\label{localfield}

This section is a collection of necessary language and basic results needed for the sequel. Most of the material covered here can be found in, for example, Borel's book \cite{Borel-Libro} and/or in the book by Benoist-Quint \cite{BQlibro}.

Let $\K$ be a local field. If $\K$ is non-Archimedean let us denote by $q$ the cardinality of its residue field, $u\in\K$ a uniformizing element, and choose the norm $|\,|$ on $\K$ so that $|u|=q^{-1}.$ In this case $\log$ denotes the logarithm on base $q,$ so that $\log q=1.$ If $\K=\R$ or $\C$ then $|\,|$ is the standard modulus, $q:=e$ and $\log$ is the usual logarithm.

Let $\sf G$ be the $\K$-points of a connected semi-simple $\K$-group, $\sf A$ the $\K$-points of a maximal $\K$-split torus and $\mathbf{X}(\sf A)$ the group of its $\K^*$-characters. Consider the real vector space $\E^*=\mathbf{X}(\sf A)\otimes_\Z\R$ and $\E$ its dual.  For every $\chi\in\mathbf X(\sf A),$ we denote by $\chi^\omega$ the corresponding linear form on $\E.$

Let $\root$ be the set of restricted roots of $\sf A$ in $\frak g,$ the set $\root^\omega$ is a root system of $\E^*.$ Let $(\root^\omega)^+$ be a system of positive roots, $\E^+$ the associated Weyl chamber and $\root^+$ and $\simple\subset\root$ the corresponding system of positive roots and simple roots respectively.

Let $\nu:\sf A\to\E$ be defined, for $z\in\sf A,$ as the unique vector in $\E$ such that for every $\chi\in\mathbf X(\sf A)$ one has $$\chi^\omega(\nu(z))=\log|\chi(z)|.$$ Denote by $\sf A^+=\nu^{-1}(\E^+).$

Let $\Weyl$ be the Weyl group of $\root,$ it is isomorphic to the quotient of the normalizer $N_{\sf G}(\sf A)$ of $\sf A$ in $\sf G$ by its centralizer $Z_{\sf G}(\sf A).$ Let $\ii:\E\to\E$ be the opposition involution: if $u:\E\to\E$ is the unique element in the Weyl group with $u(\E^+)=-\E^+$ then $\ii=-u.$

\subsection{Restricted roots and parabolic groups}

Consider $\t\subset\simple$ and let $\sf P_\t,$ resp. $\wk{\sf P}_\t,$ be the normalizers in $\sf G$ of, respectively, the Lie algebras \begin{alignat*}{2} & \bigoplus_{\aa\in\<\simple-\t\>}\ge_{-\aa}\oplus \bigoplus_{\aa\in\root^+}\ge_\aa,\\ & \bigoplus_{\aa\in\<\simple-\t\>}\ge_{\aa}\oplus \bigoplus_{\aa\in\root^+}\ge_{-\aa}. \end{alignat*}

The $\t$-\emph{flag space} is $\cal F_\t=\sf G/\sf P_\t.$ The orbit $\sf G\cdot\big([\wk{\sf P}_\t],[\sf P_\t]\big)\subset \cal F_{\ii\t}\times\cal F_\t$ is the unique open orbit on this product space, we will denote it by $\posgen_\t$ and say that $(x,y)\in\cal F_{\ii\t}\times\cal F_\t$ are \emph{transverse} if in fact $(x,y)\in\posgen_\t.$

Denote by $(\cdot,\cdot)$ a $\Weyl$-invariant inner product on $\E,$ $(\cdot,\cdot)$ the induced inner product on $\E^*,$ define $\<\,,\,\>$ on $\E^*$ by $$\<\chi,\psi\>=\frac{2(\chi,\psi) }{(\psi,\psi)}$$ and let $\{\peso_\aa\}_{\aa\in\Pi}$ be \emph{the fundamental weights} of $\root,$ defined by the equations $\<\peso_\aa,\sroot\>=d_\aa\delta_{\aa\sroot},$ where $d_\aa=1$ if $2\aa\notin\root$ and $d_\aa=2$ otherwise.

\subsection{The center of the Levi group}\label{center} We now consider the vector subspace 
$$\E_\t=\bigcap_{\aa\in\simple-\t}\ker\aa^\omega$$ 
together with the unique projection $p_\t:\E\to\E_\t$ invariant under the subgroup of the Weyl group $\Weyl_\t=\{w\in \Weyl:w|\E_\t=\id\}.$ The dual space $(\E_\t)^*$ sits naturally as the subspace of $\E^*$ of $p_\t$-invariant linear forms $$(\E_\t)^*=\big\{\varphi\in\E^*:\varphi\circ p_\t=\varphi\big\}.$$ It is spanned by the fundamental weights $\big\{\peso_\sroot|\E_ \t:\sroot\in\t\big\}.$

\subsection{Cartan decomposition}

Let $\sf K\subset\sf G$ be a compact group that contains a representative for every element of the Weyl group $\Weyl.$ This is to say, such that the normalizer $N_{\sf G}(\sf A)$ verifies $N_{\sf G}(\sf A)=(N_{\sf G}(\sf A)\cap \sf K)\sf A.$ One has $\sf G=\sf K\sf A^+\sf K$ and if $z,w\in \sf A^+$ are such that $z\in \sf Kw\sf K$ then $\nu(z)=\nu(w).$ There exists  thus a function 
$$\cartan:\sf G\to \E^+$$ such that for every $g_1,g_2\in \sf G$ one has $g_1\in \sf Kg_2 \sf K$ if and only if $\cartan(g_1)=\cartan(g_2).$ It is called the \emph{Cartan projection} of $\sf G.$

\subsection{Jordan decomposition}\label{jordan}Recall that the Jordan decomposition states that every $g\in\sf G$ has a power\footnote{($k=1$ if $\K$ is Archimedean)} $g^k$ that can be written as a commuting product $g=g_eg_hg_n,$ where $g_e$ is elliptic, $g_h$ is semi-simple over $\K$ and $g_n$ is unipotent. The component $g_h$ is conjugate to an element $z_g\in \sf A^+$ and we let $$\lambda(g)=(1/k)\nu(z_g)\in\E^+.$$ The map $\lambda:\sf G\to\sf E^+$ is \emph{the Jordan projection} of $\sf G.$ We will also denote by $\lambda_\t:\sf G\to \E_\t$ the composition $p_\t\circ\lambda.$  For $\sf G=\PGL_d(\K)$ we will denote by $\lambda_1(g)$ the logarithm of the spectral radius if $g.$

\subsection{Representations of $\sf G$}\label{representaciones}
Let $\sf V$ be a finite-dimensional $\K$-vector space and $\Fund:\sf G\to\PGL(\sf V)$ be an algebraic irreducible representation. Then the \emph{weight space} associated to $\chi\in\mathbf X(\sf A)$ is the vector space $$\sf V_\chi=\{v\in \sf V:\Fund(a) v=\chi(a) v\ \forall a\in\sf A\}$$ and if $V_\chi\neq0$ then we say that $\chi^\omega\in\E^*$ is a \emph{restricted weight} of $\Fund.$ Theorem 7.2 of Tits \cite{Tits} states that the set of weights has a unique maximal element with respect to the order $\chi\geq\psi$ if $\chi-\psi$ is a sum of simple roots with non-negative coefficients. This is called \emph{the highest weight} of $\Fund$ and denoted by $\chi_\Fund.$ 

We denote by $\|\,\|_\Fund$ a norm  on $\sf V$ invariant under $\Fund \sf K$ and such that $\Fund\sf A$ consists on semi-homotheties\footnote{i.e. diagonal on an orthonormal basis $\cal E$ of $V,$ in the classical sense if $\K$ Archimedean, and such that $\|\sum_{e\in\cal E}v_ee\|=\max \{|v_e|\}$ if $\K$ is non-Archimedean.}. If $\K$ is Archimedean the existence of such a norm is classical (see for example Benoist-Quint \cite[Lemma 6.33]{BQlibro}), if $\K$ is non-Archimedean then this is the content of Quint \cite[Th\'eor\`eme 6.1]{Quint-localFields}. 

For every $g\in \sf G$ one has then 
\begin{alignat}{2}
\log\|\Fund g\|_\Fund & =\chi_\Fund^\omega\big(\cartan(g)\big) \nonumber,\\ \log\lambda_1(\Fund g) & =\chi_\Fund^\omega\big(\lambda(g)\big).\label{eq:normayrep}
\end{alignat} 

Denote by $W_{\chi_\Fund}$ the $\Fund\sf A$-invariant complement of $\sf V_{\chi_\Fund}.$ The stabilizer in $\sf G$ of $\sf W_{\chi_\Fund}$ is $\wk{\sf P}_{\t,\K},$ and thus one has a map of flag spaces \begin{equation}\label{maps}(\Xi_\Fund,\Xi^*_\Fund):\cal F_{\t_\Fund}^{(2)}(\sf G)\to \grassman_{\dim \sf V_{\chi_\Fund}}^{(2)}(\sf V),\end{equation} where $\t_\Fund=\{\sroot\in\simple:\chi_\Fund-\sroot\textrm{ is a weight of $\Fund$}\}$. This is a proper embedding which is an homeomorphism onto its image. Here $\Gr_{\dim \sf V_{\chi_\Fund}}^{(2)}(\sf V)$ is the open $\PGL (\sf V)$-orbit in the product of the Grassmannians of $(\dim \sf V_{\chi_\Fund})$-dimensional and $(\dim \sf V-\dim \sf V_{\chi_\Fund})$-dimensional subspaces.

One has the following proposition from Tits \cite{Tits} that guarantees existence of certain representations of $\sf G.$ We say that $\phi$ is \emph{proximal} if $\dim \sf V_{\chi_\Fund}=1.$

\begin{prop1}[Tits \cite{Tits}]\label{FundTits} For every $\sroot\in\simple$ there exists an irreducible proximal representation of $\sf G$ whose highest restricted weight is $l_\sroot\peso_\sroot$ for some $l_\sroot\in\Z_{\geq1}.$\end{prop1}

\begin{defi1}We will fix and denote by $\Fund_\sroot:\sf G\to\GL(\sf V_\sroot)$ such a set of representations.\end{defi1}

\subsection{Buseman-Iwasawa cocycle}\label{BI} The \emph{Iwasawa decomposition} of $\sf G$ states that every $g\in\sf G$ can be written as a product $lzu$ with $l\in\sf K,$ $z\in\sf A$ and $u\in\sf U_\simple,$ where $\sf U_\simple$ is the unipotent radical of $\sf P_{\simple}.$ When $\K$ is non-Archimedean the Iwasawa decomposition is not unique, however if $z_1,z_2\in\sf A$ are such that $z_1\in\sf K z_2\sf U_{\simple}$ then $\nu(z_1)=\nu(z_2).$

The \emph{Buseman-Iwasawa cocycle} of $\sf G,$ $\bus:\sf G\times\cal F\to\E,$ is defined by, for all $g\in\sf G$ and $k[\sf P_\simple]\in\cal F,$ if $gk=lzu$ is an Iwasawa decomposition of $gk$ then $\bus(g,k[\sf P_\simple])=\nu(z).$ Quint proved the following.

\begin{lema1}[{Quint \cite[Lemmas 6.1 and 6.2]{quint1}}] The function $\bus_\t=p_\t\circ\bus$ factors as a cocycle $\bus_\t:\sf G\times \cal F_\t\to\E_\t.$\end{lema1}

The Buseman-Iwasawa cocycle can also be read from the representations of $\sf G.$ Indeed, Quint \cite[Lemme 6.4]{quint1} states that for every $g\in\sf G$ and $x\in\cal F_\t$ one has \begin{equation}\label{busnorma}l_\sroot\peso_\sroot(\bus(g,x))=\log\frac{\|\Fund_\sroot(g)v\|_\Fund}{\|v\|_\Fund},\end{equation} where $v\in\Xi_{\Fund_\sroot}(x)\in\P(\sf V_\sroot)$ is non-zero.

\subsection{Gromov product}\label{GryBus}

As in S. \cite{orbitalcounting}, the \emph{Gromov product} $\Gr_\t:\posgen_\t\to\E_\t$ is defined such that, for every $(x,y)\in\posgen_\t$ and $\sroot\in\t,$ one has $$l_\sroot\peso_\sroot\big(\Gr_\t(x,y)\big)= \log\frac{|\varphi(v)|}{\|\varphi\|_{\Fund_\sroot}\|v\|_{\Fund_\sroot}},$$ where $\varphi\in \Xi^*_{\Fund_\sroot}(x)$ and $v\in \Xi_{\Fund_\sroot}(y)$ are the equivariant maps from equation (\ref{maps}). 

\begin{obs*}\label{-infty}Observe that the limiting situation $l_\sroot\peso_\sroot\big(\Gr_\t(x,y)\big)=-\infty$ occurs when $v\in\ker\varphi,$ i.e. when $x$ and $y$ are no longer transverse flags, so a statement of the form $\peso_\sroot\Gr_\t(x,y) \geq-\kappa$ for all $\sroot\in\t$ is a quantitative version (that depends on $\sf K$) of the transversality between $x$ and $y.$
\end{obs*}

A straightforward computation (S. \cite[Lemma 4.12]{orbitalcounting}) gives, for all $g\in\sf G$ and $(x,y)\in\posgen_\t,$ \begin{equation}\label{forGr}\Gr_\t(gx,gy)-\Gr_\t(x,y)=-\big(\ii\bus_{\ii\t}(g,x)+\bus_\t(g,y)\big).\end{equation}

\subsection{Proximality}\label{proxBenoist}

Recall that $g\in\PGL_d(\sf V)$ is \emph{proximal} if it has a unique eigenvalue with maximal modulus and that the multiplicity of this eigenvalue in the characteristic polynomial of $g$ is $1.$ The associated eigenline is denoted by $g^+\in\P(\sf V)$ and $g^-$ is its $g$-invariant complementary subspace.

We say then that $g\in\sf G$ is $\t$-\emph{proximal} if for every $\sroot\in\t$ one has $\Fund_\sroot(g)$ is proximal. In this situation, there exists a pair $(g^-_\t,g^+_\t)\in\posgen_\t,$ defined by, for every $\sroot\in\t,$ $\Xi_{\Fund_\sroot}\big(g^+_\t\big)=\Fund_\sroot(g)^+,$ and every flag $x\in\cal F_\t$ in general position with $g^-_\t$ verifies $g^nx\to g^+_\t.$ 

It is also useful to consider a quantified version of proximality. Given $r,\eps$ positive we say that $g$ is $(r,\eps)$-\emph{proximal} if it is proximal, $$\peso_\sroot\Gr_\t(g_\t^-,g^+_\t)\geq-r$$ for all $\sroot\in\t$ and for every $x\in\cal F_\t$ with $\min_{\sroot_\in\t}\peso_\sroot\Gr_\t(g_\t^-,x)\geq- \eps^{-1}$ one has $d_{\cal F_\t}(gx,g^+_\t)\leq\eps.$ More details on the following can be found in S. \cite[Lemma 5.6]{quantitative}

\begin{prop1}[{Benoist \cite[Corollaire 6.3]{Benoist-HomRed}}]\label{proxCartan}For every $\delta>0$ there exist $r,\eps>0$ such that if $g\in\sf G$ is $(r,\eps)$-proximal then $$\big\|\cartan_\t(g)-\lambda_\t(g)-\Gr_\t(g_\t^-,g^+_\t)\big\|\leq\delta.$$
\end{prop1}

\subsection{Cartan attractors}\label{s.Basin} Consider $g\in\sf G$ and let $g=k_gz_gl_g$ be a Cartan decomposition. We say that $g\in\sf G$ has \emph{a gap at $\t$} if for all $\sroot\in\t$ one has $$\sroot\big(\cartan(g)\big)>0.$$ In that case the \emph{Cartan attractor of $g$ in $\cal F_\t$} $$U_\t(g)=k_g[\sf P_{\t}]$$ is well defined: uniquely defined if $\K$ is Archimedian; defined up to a ball of radius $q^{-\min_{\sroot\in\t}\sroot(g)}$ if $\K$ is non-Archimedean (see Pozzetti-S.-Wienhard \cite[Remark 2.4]{PSW1}).

\begin{obs*} For every $\sroot\in\t$ one has ${\displaystyle\Xi_{\Fund_\sroot}\big(U_\t(g)\big)=U_{1}\big(\Fund_\sroot(g)\big).}$
\end{obs*}

\begin{lema1}[{Bochi-Potrie-S. \cite[Lemma A.5]{BPS}}]\label{lemma1} Consider $g,h\in\sf G$ such that $h$ and $gh$ have gaps at every $\sroot\in\t,$ then one has $$d\big(U_\t(gh),gU_\t(h)\big)\leq q^{-\min_{\sroot\in\t}\sroot(h)}\cdot\max_{\sroot\in\t}\big\{\|\Fund_\sroot(g)\|\|\Fund_\sroot(g^{-1})\|\big)\big\}.$$
\end{lema1}

The \emph{Cartan basin of $g$} is defined, for $\alpha>0,$ by (recall Remark \ref{-infty})$$B_{\t,\alpha}(g)= \Big\{x\in\cal F_\t: \peso_\sroot\big(\Gr_\t\big(U_{\ii\t}(g^{-1})\big),x\big)>-\alpha\ \forall\sroot\in\t\Big\}.$$

\noindent 
It is clear from the definition that given $\alpha>0$ there exists a constant $K_\alpha$ such that if $y\in\cal F_{\t}$ belongs to $B_{\t,\alpha}(g)$ then one has
\begin{equation}\label{comparision-shadow}
\big\|\cartan_\t(g)-\bus_\t(g,y)\big\|\leq K_\alpha.
\end{equation}

\begin{lema1}[{Quint \cite[Lemme 6.6]{quint1}}]\label{lQ} For every $g\in \sf G$ one has $\cartan_\t(gh)-\cartan_\t(h)-\bus_\t\big(g,U_\t(h)\big)\to0$ as $\min_{\sroot\in\t}\sroot\big(\cartan(h)\big)\to\infty.$
\end{lema1}

\subsection{General facts on discrete subgroups}\label{generalidades} We record here some facts related to the title that we will need in the sequel.

\begin{lema1}\label{aFinito}Let $\grupo\subset\sf G$ be a discrete subgroup, then for every $\varphi\in\E^*$ strictly positive on $\E^+$ the exponential rate $$\delta^\varphi_\grupo=\limsup_{t\to\infty}\frac1t\log\#\big\{g\in\grupo:\varphi\big(\cartan(g)\big)\leq t\big\}$$ is finite.
\end{lema1}

\begin{proof}Follows from a computation of the Haar measure of $\sf G,$ to be found in Helgason \cite{helga} for the Archimedean case and in Matsumoto \cite[\S\,3.2.7]{matsumoto} for the non-Archimedean case, see Quint \cite[\S\,4]{quint2} for details.\end{proof}

We record also the following theorem from Benoist \cite{benoist2}.

\begin{thm1}[{Benoist \cite{benoist2}}]\label{densidad} Assume $\K=\R$ and let $\grupo\subset\sf G$ be a Zariski-dense subgroup, then the group spanned by the Jordan projections $\{\lambda(g):g\in\grupo\}$ is dense in $\E.$
\end{thm1}

\section{Anosov representations}\label{AnosovReps}

Anosov representations where introduced by Labourie \cite{labourie} for fundamental groups of negatively curved manifolds and extended to arbitrary finitely generated hyperbolic groups by Guichard-Wienhard \cite{olivieranna}. They originated as a tool to study higher rank Teichm\"uller Theory, and are nowadays consider as the higher-rank analog of what is known in pinched negative curvature as \emph{convex co-compact} groups.

\begin{notacion}If $\rho:\G\to\sf G$ is a representation we will simplify notation and denote, for $\g\in\G,$ by $\g_\rho=\rho(\g).$\end{notacion}

\subsection{Real projective-Anosov representations}\label{flujoBCLS}

We begin by recalling Labourie's original approach. Let $\G$ be a finitely generated word-hyperbolic group. If $\rho:\G\to\PGL_d(\R)$ is a representation then we can consider the \emph{natural flat bundle automorphism} defined as follows. Consider the flat bundle $\R^d\to\sf E_\rho\to\sf U\G$ defined by $\widetilde{\sf U\G}\times\R^d/_\sim$ where $(p,v)\sim\big(\g p,\g_\rho v\big),$ and define $\hat{\sf g}=\big(\hat{\sf g}_t:\sf E_\rho\to \sf E_\rho\big)_{t\in\R}$ as the induced on the quotient by $t\cdot(p,v)=(\widetilde{\sf g}_t p, v).$

\begin{defi1} The representation $\rho$ is \emph{projective-Anosov} if there exists a pair of continuous $\rho$-equivariant maps \begin{alignat*}{2} \xi^1 & :\bord\G\to\P(\R^d)\\ \xi^{d-1}& :\bord\G\to\P\big((\R^d)^*\big)\end{alignat*} such that:\begin{itemize}\item[-] for every $(x,y)\in\bord^2\G$ one has $\ker\xi^{d-1}(x)\oplus\xi^1(y)=\R^d;$ this induces a $\hat{\sf g}$-invariant decomposition $\Xi\oplus\Theta=\sf E_\rho$

\item[-] the decomposition $\sf E_\rho=\Xi\oplus\Theta$ is a \emph{dominated splitting} for $\hat{\sf g},$ i.e. there exist $c,\alpha$ positive such that for every $v\in\Xi_p$ and $w\in\Theta_p$ one has $$\frac{\|\hat{\sf g}_tv\|}{\|\hat{\sf g}_tw\|}\leq ce^{-\alpha t}\frac{\|v\|}{\|w\|}.$$
\end{itemize}
\end{defi1}

One has the following standard consequences, see for example Guichard-Wienhard \cite[Lemma 3.1]{olivieranna} or Bridgeman-Canary-Labourie-S. \cite[Lemma 2.5+Proposition 2.6]{pressure}. Recall from \S\,\ref{proxBenoist} that $g\in\PGL_d(\R)$ is \emph{proximal} if the Jordan block associated to the eigenvalues with maximal modulus is 1-dimensional.

\begin{lema1}\label{proximal}
 If $\rho$ is projective-Anosov then  for every hyperbolic $\g$ one has $\g_\rho$ is proximal with attracting line $\xi^1(\g_+).$ In particular the entropy $$\lim_{t\to\infty}\frac1t\log\#\big\{[\g]\in[\G]\textrm{ hyperbolic}:\lambda_1\big(\g_\rho \big)\leq t\big\}\in[0,\infty).$$ The equivariant maps $\xi^1$ and $\xi^{d-1}$ are H\"older-continuous. 
\end{lema1}

\begin{proof} Let us add a word on finiteness of entropy. Recall from Bowditch \cite{bowditch} that, since $\G$ is hyperbolic its action on the space of pairwise distinct triples $\bord^{(3)}\G$ is properly discontinuous and co-cocompact. If $\g\in\G$ is hyperbolic one can choose then $\eta\in[\g]$ (the conjugacy class of $\g$) whose fixed points are far appart by a constant independent of $\g.$ Since the image $\g_\rho$ is proximal and the equivariant maps are continuous, one has that $\eta_\rho$ is $(r,\eps)$-proximal, for constants $r,\eps$ independent of $[\g].$ By Proposition \ref{proxCartan} one has then $\big|\log\|\eta_\rho\|-\lambda_1(\eta_\rho)\big|<K$ for some $K$ independent of $\eta.$ If follows then that for every $t\in\R_+$ $$\#\big\{[\g]\in[\G]\textrm{ hyperbolic}:\lambda_1(\g_\rho)\leq t\big\}\leq\#\big\{\g\in\G:\log\|\g_\rho\|\leq t+K\big\}.$$ Finiteness of entropy then follows from Lemma \ref{aFinito}.\end{proof}

We use the equivariant maps to construct  a bundle $\R\to\widetilde{\sf F}\to \bord^2\G$ whose fiber at $(x,y)\in\bord^2\G$ is $$\widetilde{\sf F}_{(x,y)}=\big\{(\varphi,v)\in\xi^{d-1}(x)\times\xi^1(y):\varphi(v)=1\big\}/\sim,$$ where $(\varphi,v)\sim(-\varphi,-v).$ This bundle is equipped with a $\G$-action $\g(\varphi,v)=\big(\varphi\circ\g^{-1}_\rho,\g_\rho v\big)$ and an $\R$-action $\big(\widetilde{\sf g^\rho}_t:\widetilde{\sf F}\to\widetilde{\sf F}\big)_{t\in\R}$ defined by $\widetilde{\sf g^\rho}_t\cdot(\varphi,v)=(e^t\varphi,e^{-t}v).$ Let $\sf F=\G\/\widetilde{\sf F}$ and denote by $\sf g^\rho=\big(\sf g^\rho_t:\sf F\to\sf F\big)_{t\in\R}$ the induced flow on the quotient, it is usually called \emph{the geodesic flow of $\rho$}.

\begin{thm1}[Bridgeman-Canary-Labourie-S. \cite{pressure}]\label{tutti}The above $\G$-action is properly discontinuous and co-compact. The flow $\sf g^\rho$ is H\"older-continuous and metric-Anosov with stable/unstable laminations the (induced on the quotient by) \begin{alignat*}{2}\tilde W^\ss\big((x,y,(\varphi,v)\big) & =\big\{\big(x,\cdot,(\varphi,\cdot)\big)\in\widetilde{\sf F}\big\}\\ W^\uu\big((x,y,(\varphi,v)\big) & =\big\{\big(\cdot,y,(\cdot,v)\big)\in\widetilde{\sf F}\big\}.\end{alignat*} It is moreover H\"older-conjugated to the Gromov-Mineyev geodesic flow $\sf g$ of $\G,$ consequently this latter flow is also metric-Anosov. 
\end{thm1}

Consequently, hyperbolic groups admitting a real projective-Anosov representation verify Assumption \ref{assuB} and are thus subject of a Ledrappier correspondence (\S\,\ref{piedrapie}). It is stablished in Carvajales \cite[Appendix]{lyon1} that $\sf g^\rho$ is topologically mixing (regardless the Zariski closure of $\rho$) and thus mixing for any equilibrium state.

\subsection{Arbitrary $\sf G,$ coarse geometry viewpoint}\label{coarse} Let $\sf G$ be as in \S\,\ref{localfield}, we use freely the notation introduced there and fix from now on a subset $\t\subset\simple$ of simple roots.

Let $\G$ be a finitely generated group and denote, for $\g\in\G,$ by $|\g|$ the word length w.r.t. a fixed finite symmetric generating set of $\G.$ 

\begin{defi1}A representation $\rho:\G\to\sf G$ is $\t$-\emph{Anosov} if there exist $c,\mu$ positive such that for all $\g\in\G$ and $\sroot\in\t$ one has \begin{equation}\label{defA}\sroot\big(\cartan(\g_\rho)\big)\geq \mu|\g|-c.\end{equation} The constants $c$ and $\mu$ will be referred to as \emph{the domination constants} of $\rho.$
\end{defi1}

The Theorem below follows from the main result by Kapovich-Leeb-Porti \cite{KLP-Morse} and the standard facts from representation theory stated in \S\,\ref{representaciones}, a proof can also be found in Bochi-Potrie-S. \cite{BPS}.

\begin{thm1}\label{A-A} If $\rho:\G\to\sf G$ is $\t$-Anosov then $\G$ is word-hyperbolic. If moreover $\K=\R$ then for every $\sroot\in\t$ the representation $\Fund_\sroot\circ\rho:\G\to\PGL(V_\sroot)$ is projective-Anosov (as in \emph{\S\,\ref{flujoBCLS}}). \end{thm1}

The following lemma is essentially a consequence of {Bochi-Potrie-S. \cite[Lemma 4.9]{BPS}}. See Pozzetti-S.-Wienhard \cite[Proposition 3.5]{PSW1} for details concerning the non-Archimedean case. The last assertion is classical.

\begin{prop1}[{Bochi-Potrie-S. \cite[Proposition 4.9  ]{BPS}}]\label{conical}
If $\rho:\G\to\sf G$ is $\t$-Anosov, then for any geodesic ray $\{\alpha_n\}_0^\infty$ with endpoint $x$, the limits
$$\xi^\t_\rho(x):=\lim_{n\to\infty}U_\t\big(\rho(\alpha_n)\big)\quad \xi^{\ii\t}_\rho(x):=\lim_{n\to\infty}U_{\ii\t}\big(\rho(\alpha_n)\big)$$
exist and do not depend on the ray; they define continuous $\rho$-equivariant transverse maps $\xi^{\t}:\bord\G\to\cal F_\t$, $\xi^{\ii\t}:\bord\G\to\cal F_{\ii\t}$.  If $\g\in\G$ is hyerbolic, then $\g_\rho$ is $\t$-proximal with attracting point $\xi^\t(\g^+)=(\g_\rho)_\t^+.$
\end{prop1} 

The above Proposition readily implies the following Lemma (recall Remark \ref{-infty}).

\begin{lema1}\label{lejosplanos} Let $\rho:\G\to\sf G$ be $\t$-Anosov, $\{\g_n\}\subset\G$ a divergent sequence and $x\in\bord\G.$ Then, as $n\to\infty,$ one has:
$$\g_n\to x\;\Leftrightarrow\; U_\t\big(\rho(\g_n)\big)\to\xi^\t(x)\;\Leftrightarrow\;\exists\sroot\in\t\textrm{ s.t. }\peso_\sroot\Gr_\t\big(U_{\ii\t}\big(\rho(\g_n)\big),\xi^\t(x)\big)\to-\infty.$$\end{lema1}

We finally record the following useful Lemma. 

\begin{lema1}[{Pozzetti-S.-Wienhard \cite[Lemma 3.6]{PSW1}}]\label{tubo} Let $\rho:\G\to \sf G$ be $\t$-Anosov, then for every $\eps>0$ there exists $L$ such that $$\overline{\bigcup_{\g:|\g|>L} U_\t\big(\g_\rho \big)}\subset \cal N_\eps\big(\xi^\t(\bord\G)\big),$$ where $\cal N_\eps$ denotes the $\eps$-tubular-neighborhood.
\end{lema1}

\begin{obs*}[Non-Archimedean case]The existence of continuous $\rho$-equivariant maps implies, when $\K$ is non-Archimedean, that the boundary of $\G$ is necessarily a Cantor set and thus $\G$ is virtually free. The Gromov-Mineyev of $\G$ is thus a suspension of a sub-shift of finite type and is, hence, metric-Anosov.
\end{obs*}

\noindent
\fbox{\begin{minipage}{0.98\textwidth}
\begin{nott} A $\t$-Anosov representation $\rho:\G\to \sf G$ is fixed from now on. By \S\,\ref{flujoBCLS} for $\K=\R$ or $\C,$ and the preceding paragraph for non-Archimedean $\K,$ the Gromov-Mineyev flow $\sf g$ of $\G$ satisfies Assumption \ref{assuB}. \end{nott}
\end{minipage}
}

\subsection{The refraction cocycle verifies Assumption \ref{assuC}}\label{arbitraryG}

Via the equivariant boundary maps of $\rho$ one can pullback the Buseman-Iwasawa cocycle of $\sf G$ to obtain a H\"older-cocycle on the boundary of $\G$:

\begin{defi1}\label{refr}The \emph{refraction cocycle} of $\rho$ is $\bb{}:\G\times\bord\G\to\E_\t$ $$\bb{}(\g,x)=\bb{}^\rho(\g,x)=\bus_\t\big(\g_\rho ,\xi^\t_\rho(x)\big).$$\end{defi1}

The limit cone of $\bb{}$ will be denoted by $\cone_{\t,\rho}$ and referred to as \emph{the $\t$-limit cone of $\rho$}. The period computation below implies it is the smallest closed cone of $\E_\t$ that contains the projections $\big\{\lambda_\t\big(\g_\rho \big):\g\in\G\big\}.$ We prove moreover that $\bb{}$ verifies Assumption \ref{assuC} from \S\,\ref{entropy1}.

\begin{lema1}\label{beta<infty} The periods of $\bb{}$ are $\bb{}(\g,\g_+)=\lambda_\t\big(\g_\rho \big)$, consequently Assumption \emph{\ref{assuC}} holds for $\bb{}$, in particular $\inte\conodual=\big\{\varphi\in\conodual:\h_\varphi<\infty\big\}$. \end{lema1}

\begin{proof} The first assertion follows from Proposition \ref{conical}. To prove assumption \ref{assuC} holds one considers any $\sroot\in\t$ and the representation $\Fund_\sroot.$ By Theorem \ref{A-A} the composition $\Fund_\sroot\rho:\G\to\GL(\sf V_\sroot)$ is projective-Anosov and thus, by (\ref{eq:normayrep}) and Lemma \ref{proximal}, the form $\peso_\sroot\in(\E_\t)^*$ has finite entropy. The last assertion follows from Lemma \ref{h<infty}.\end{proof}

Lemma \ref{h<infty} and Theorem \ref{refractionThm} give then the following.

\begin{cor1}\label{flujoPhi} There exists a H\"older-continuous function $\LL_{\t,\rho}:\sf U\G\to\E_\t$ such that for every hyperbolic $\g\in\G$ one has $\int_{[\g]}\LL_{\t,\rho}=\lambda_\t\big(\g_\rho \big).$ For every $\varphi\in\inte\conodual$ the $\G$-action on $\bord^{2}\G\times\R$ defined by \begin{equation}\label{espacio}\g\cdot(x,y,t)=\Big(\g x,\g y,t-\bb\varphi(\g,y)\big)\end{equation} is properly discontinuous and co-compact. The $\R$-translation flow induces on the quotient a flow $\phi^\varphi=\big(\phi^\varphi_t:\RR\varphi\to\RR\varphi)_{t\in\R}$ (bi)-H\"older-conjugated to the reparametrization of $\sf g$ by $\varphi\circ \LL_{\t,\rho}.$
\end{cor1}

\begin{defi1} The function $\LL_{\t,\rho}$ will be referred to as \emph{the Ledrappier potential} of $\rho.$ The flow $\phi^\varphi$ will be called the $\varphi$-\emph{refraction flow of $\rho$}.  
\end{defi1}

\subsection{The $\t$-limit cone}\label{cerca} We mimick some celebrated results by Benoist \cite{limite} for Zariski dense subgroups and $\t=\simple.$

\begin{lema1}[{Benoist \cite[Proposition 5.1]{Benoist-HomRed}}] For every compact set $L\subset\sf G$ there exists a compact set $H\subset \E$ such that for every $g\in \sf G$ one has $\cartan(LgL)\subset\cartan(g)+H.$ 
\end{lema1}

Let us also denote by $\cartan_\t=p_\t\circ\cartan.$ 

\begin{prop1}\label{cerquita}Let $\rho:\G\to \sf G$ be a $\t$-Anosov representation, then there exists a compact set $D\subset \E_\t$ such that $\cartan_\t\big(\rho(\G)\big)\subset \lambda_\t\big(\rho(\G)\big)+D.$ 
\end{prop1}

\begin{proof} As $\G$ is finitely generated and word-hyperbolic, there exist $\kappa>0$ and two elements $u,v\subset\G$ such that for every non-torsion $\g\in\G$ there exists $f\in \{u,v\}$ such that $f\g$ verifies $$d_{\bord\G}\big((f\g)^+,(f\g)^-\big)>\kappa.$$ As $\rho$ is $\t$-Anosov, the above equation implies the element $\rho(f\g)$ is $(r,\eps)$-proximal on $\t$ for some $r$ only depending on $\kappa.$ By Proposition \ref{proxCartan} one has $$\big\|\cartan_\t\big(\rho(f\g)\big)-\lambda_\t\big(\rho(f\g)\big)\big\|\leq K,$$ for some $K$ only depending on $\kappa$ and $\rho.$ We consider the compact set $H$ from the Lemma above applied to $L=\rho(\{u^{-1},v^{-1}\})$ and we let $D:=p_\t(H)+B(0,K).$
\end{proof}

We will mainly use the following direct consequence:

\begin{cor1}\label{finito} If $\varphi\in\inte\conodual$ then the exponential rate $$\delta^\varphi:=\limsup_{t\to\infty}\frac1t\log\#\big\{\g\in\G:\varphi\big(\cartan(\g_\rho)\big)\leq t\big\}<\infty.$$
\end{cor1}

\begin{proof} If $\sroot\in\t$ then, since both intersections $\ker\peso_\sroot\cap\cone_{\t,\rho}$ and $\ker\varphi\cap\cone_{\t,\rho}$ vanish (the first one always does, the second one by the assumption on $\varphi$) the function $\varphi/\peso_\sroot$ is bounded  below away from zero on $\cone_{\t,\rho}.$ By Proposition \ref{cerquita} there exist positive $c$ and $C$ such that for all hyperbolic $\g\in\G$ one has $$\varphi\big(\cartan(\g_\rho )\big)\geq c\peso_\sroot\big(\cartan(\g_\rho) \big)-C.$$  Lemma \ref{aFinito} gives then the desired result.\end{proof}

\subsection{Patterson-Sullivan Theory along the Anosov roots: existence}\label{PS1} In this section we will construct, for each $\varphi\in\inte\conodual$ a $\bb\varphi$-Patterson-Sullivan measure\footnote{The $\t$-Anosov property is not really used until the uniqueness corollary, the existence presented here works for any discrete group whose limit cone on $\E$ does not intersect any wall associated to $\t$ and replacing $\xi^\t(\bord\G)$ by $$\bigcap_{n\in\N}\overline{\{U_\t(g):g\in\grupo \textrm{ with }\min_{\sroot\in\t}\sroot(\cartan(g))\geq n\}}.$$}. The procedure is standard and follows the original idea by Patterson.

We begin by considering the Dirichlet series $$\Po^\varphi(s)=\sum_{\g\in\G} q^{-s\varphi\big(a(\g_\rho)\big)}.$$ It is convergent for every $s>\delta^\varphi$ and divergent for every $s<\delta^\varphi.$ As it is customary when constructing Patterson-Sullivan measures, we can assume throughout this subsection that $\Po^\varphi(\delta^\varphi)=\infty,$ otherwise one would consider the series $$s\mapsto \sum_{\g\in\G} h\Big(\varphi\big(\cartan(\g_\rho )\big)\Big)q^{-s\varphi\big(\cartan(\g_\rho)\big)}$$ for some real function $h$ defined, for example, as in Quint \cite[Lemma 8.5]{quint1}.

For  $s>\delta^\varphi$ consider the probability measure on $\cal F_\t$ defined by $$\nu_s=\frac1{\Po^\varphi(s)}\sum_{\g\in\G}q^{-s\varphi\big(\cartan(\g_\rho)\big)}\dirac_{U_\t\big(\g_\rho \big)}.$$

\begin{lema1}\label{eps} For every $\eta\in\G$ the signed measure 

$$\eps(\eta,s)=(\eta_\rho)_*\nu_s-\frac1{\Po^\varphi(s)}\sum_{\g\in\G}q^{-s\varphi\big(\cartan(\g_\rho)\big)}\dirac_{U_\t\big(\eta_\rho\g_\rho\big)}$$ weakly converges to $0$ as $s\searrow \delta^\varphi.$

\end{lema1}

This is a standard argument that can be found, for example, in Pozzetti-S.-Wienhard \cite[Lemma 5.11]{PSW1}.

\begin{proof} It is sufficient to check the convergence for continuous functions. If $f:\cal F_\t\to\R$ is continuous then $$|\eps(\eta,s)(f)|\leq\frac1{\Po^\varphi(s)}\sum_{\g\in\G}q^{-s\varphi\big(\cartan(\g_\rho )\big)}\Big|f\big(\eta_\rho U_\t(\g_\rho)\big)-f\big(U_\t(\eta_\rho\g_\rho)\big)\Big|.$$ By Lemma \ref{lemma1} and uniform continuity of $f$ the convergence follows. \end{proof}

\begin{lema1}\label{eqCocyclo} Let $\nu^\varphi$ be any weak-star limit of $\nu_s$ when $s\searrow \delta^\varphi.$ Then the support of $\nu^\varphi$ is contained in $\xi^\t(\bord\G).$ Moreover, for every $\eta\in\G$ one has $$\frac{d\rho(\eta)_*\nu^\varphi}{d\nu^\varphi}(x)=q^{-\varphi\Big(\bus_\t\big(\eta^{-1}_\rho,x\big)\Big)}.$$
\end{lema1}

\begin{proof} The first statement follows at once from Lemma \ref{tubo} since we assumed $\Po^\varphi(\delta^\varphi)=\infty.$ For the second statement, consider a sequence $s_k\searrow\delta^\varphi$ such that $\nu_{s_k}\to\nu^\varphi.$ One then has \begin{alignat*}{2}\rho(\eta)_*\nu^{s_k} & =\eps(\eta,s_k)+\frac1{\Po^\varphi(s_k)}\sum_{\g\in\G}q^{-s_k\varphi\big(\cartan(\g_\rho )\big)}\dirac_{U_\t(\eta_\rho\g_\rho)}\\ &=\eps(\eta,s_k)+\frac1{\Po^\varphi(s_k)}\sum_{\g\in\G}q^{-s_k\varphi\big(\cartan(\eta^{-1}_\rho\g_\rho)\big)}\dirac_{U_\t(\g_\rho)}\\ & =\eps(\eta,s_k)+\frac1{\Po^\varphi(s_k)}\sum_{\g\in\G}q^{-s_k\varphi\big(\cartan(\eta^{-1}_\rho\g_\rho)-\cartan(\g_\rho )\big)}q^{-s_k\varphi\big(\cartan(\g_\rho )\big)}\dirac_{U_\t(\g_\rho)}\\ & = \eps(\eta,s_k)+\frac1{\Po^\varphi(s_k)}\sum_{\g\in\G}q^{-s_k\varphi\Big(\bus_\t\big(\eta^{-1}_\rho,U_\t(\g_\rho )\big)+\eps'(\eta,\g)\Big)}q^{-s_k\varphi\big(\cartan(\g_\rho )\big)}\dirac_{U_\t(\g_\rho)} \end{alignat*} 

\noindent
where, by Quint's Lemma \ref{lQ} and the fact that $\varphi\circ p_\t=\varphi,$ one has $\eps'(\eta,\g)\to0$ as $\min_{\sroot\in\t}\sroot\big(\cartan\big(\g_\rho \big)\big)\to\infty.$ Taking limit as $s_k\searrow\delta^\varphi$ one has, since we assumed $\Po^\varphi(\delta^\varphi)=\infty,$ that only elements $\g\in\G$ with arbitrary big $|\g|$ count in the sum. Since $\rho$ is $\t$-Anosov, this is equivalent to considering elements $\g\in\G$ such that $$\min_{\sroot\in\t}\sroot\big(\cartan(\g_\rho) \big)$$ is arbitrary big. The result then follows as $\eps(\eta,s_k)\to0$ by the Lemma above and $\eps'(\eta,\g)$ is arbitrary small.\end{proof}

Since Assumption \ref{assuC} holds for $\bb{}$ (Lemma \ref{beta<infty}), \S\,\ref{PattersonC} applies to give:

\begin{cor1}\label{existe}For every $\varphi\in\inte\conodual$ there exists a $\bb\varphi$-Patterson-Sullivan measure $\mu^\varphi:=(\xi^\t)_*\nu^\varphi$ of exponent $\delta^\varphi.$ Such a measure is ergodic and moreover one has $\delta^\varphi=\h_\varphi.$ If $\psi\in\inte\conodual$ is such that $\mu^\psi\ll\mu^\varphi$ then for every hyperbolic $\g\in\G$ one has $$\h_\varphi\varphi\big(\lambda_\t(\g_\rho )\big)=\h_\psi\psi\big(\lambda_\t(\g_\rho )\big)$$ and, in particular, $\mu^\psi=\mu^\varphi.$ 
\end{cor1}

The above corollary was previously stablished by Dey-Kapovich \cite[Main Theorem]{Dey-Kapovich} for real algebraic groups, $\ii$-invariant functionals $\varphi\in\inte(\a^+)^*$ and $\ii$-invariant subsets $\t.$ The equality $\delta^\varphi=\h_\varphi$ can also be found in Glorieux-Monclair-Tholozan {\cite[Theorem 2.31 (2)]{GMT}} for real groups.

\begin{obs*}\label{existenceAssumption}
We conclude by remarking that, for $\varphi\in\inte\conodual,$ the existence assumptions of \S\,\ref{PattersonC} are guaranteed for $\bb\varphi.$ Indeed Proposition \ref{existe} states the existence of a Patterson-Sullivan measure $\mu^\varphi$ for $\bb\varphi.$ On the other hand the cocycle $$\bd{}(\g,x)=\ii\bus_{\ii\t}\big(\g_\rho ,\xi^{\ii\t}(x)\big)$$ is dual to $\bb{}$ and moreover, from equation (\ref{forGr}), the function $[\cdot,\cdot]_\varphi:\bord^2\G\to\R$ \begin{equation}\label{prodGr}[x,y]_\varphi=\varphi\Big(\Gr_\t\big(\xi^{\ii\t}(x),\xi^\t(y)\big)\Big)\end{equation} is a  Gromov product for the pair $(\bd{\varphi},\bb{\varphi}).$ Finally, exchanging $\t$ with $\ii\t,$ Proposition \ref{existe} provides a Patter\-son-Sullivan measure for $\bd\varphi.$ We can thus apply the results from \S\,\ref{directions} and \S\,\ref{fibered}.

\end{obs*}

%\begin{obs*}\label{sepuede} Since $\mu^\varphi$ is ergodic, equation (\ref{irreducibility}) holds for $\mu^\varphi$ and thus we can apply Sullivan's shadow Lemma to these measures.\end{obs*}

\subsection{Cartan's basins have controlled overlaps}\label{buenC}\label{buencubrimiento} The job of understanding the overlaps of Cartan's bassins for Anosov representations has been carried out in Pozzetti-S.-Wienhard \cite{PSW2}. The idea is to compare the Cartan's basins of elements $\g_\rho ,$ for hyperbolic $\g\in\G,$ with the \emph{coarse cone type} of $\g.$ 

Let $c_0,c_1$ be positive and $I\subset\Z$ an interval, then a $(c_0,c_1)$-\emph{quasigeodesic} is a sequence $\{\alpha_i\}_{i\in I}\in\G$  such that for every pair $j,l$ in the interval $I$ one has $$\frac1{c_0}|j-l|-c_1\leq d_\G(\alpha_j,\alpha_l)\leq c_0|j-l|+c_1.$$ The  \emph{coarse cone type at infinity} of $\g\in\G$ consists of endpoints on $\bord\G$ of quasi geodesic rays based at $\g^{-1}$ passing through the identity (see Figure \ref{CCT}):   
\begin{alignat*}{2} \cono^{c_0,c_1}_\infty&(\g) =  \\ & \Big\{[\{\alpha_j\}_0^\infty]\in\bord\G:\,  \{\alpha_i\}_0^\infty \text{ is a $(c_0,c_1)$-quasi-geodesic with } \alpha_0=\g^{-1}, e\in\{\alpha_j\}\Big\}.\end{alignat*}

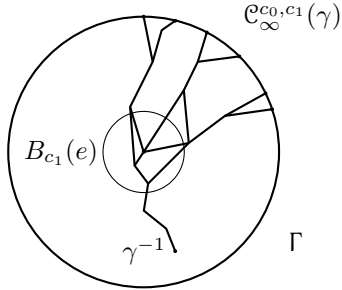
\begin{figure}[h]
\centering
\begin{tikzpicture}[scale = 0.6]
\draw [thick] circle [radius = 3];
\draw [fill] circle [radius = 0.03];	
\draw  circle [radius = 0.9];
\node [left] at (-0.8,0) {$B_{c_1}(e)$};
\draw [thick] (0.7,-2.2) -- (0.5,-1.7) -- (0.004,-1.3) -- (0.1,-0.7) -- (-0.2,-0.3) -- (0,0); %camino a g^{-1}
\node [left] at (0.7,-2.2) {$\g^{-1}$}; \draw [fill] (0.7,-2.2) circle [radius = 0.03]; %g^{-1}
\draw [thick] (1,0.2) -- (0.1,-0.7);
\draw [thick] (-0.2,-0.3) -- (-0.3,1);
%\draw [thick] (-0.1,-0.3) -- (0.97,0.45);
\draw [thick] (0,0) -- (-0.3,1) -- (0.2,2) -- (0,3); %camino mas a la izquierda
\draw [thick] (0.2,2) -- (0.4, 2.7) -- (0.7,2.91); %rama de camino a la izquierda
\draw [thick] (1,0.2) -- (0.89,1.37);
%\draw [thick] (0.2,2) -- (0.4, 2.7);
\draw [thick] (1.8,0.8) -- (2.7,1.3);
\draw [thick] (0,0) -- (1,0.2) -- (1.8,0.8) -- (2.846,0.948);
\draw [fill] (0.7,2.91) circle [radius = 0.03];
\draw [fill] (2.7,1.3) circle [radius = 0.03];
\draw [fill] (2.12,2.12) circle [radius = 0.03];
%\node [above] at (0,3) {math here};
\draw [fill] (0,3) circle [radius = 0.03];
%\node [right] at (2.846,0.948) {math here};	
\draw [fill] (2.846,0.948) circle [radius = 0.03];	
%\draw [thick] (2.846,0.948) to [out = 197,in = 270] (0,3);
\draw [thick] (0,0) -- (0.89,1.37) -- (1.2,2) -- (2.12,2.12);
\draw [thick] (1.2,2) -- (1.4,2.65);
\draw [fill] (1.4,2.65) circle [radius = 0.03];
%\draw [fill] (0.882,1.35) circle [radius = 0.03];
\node [right] at (3,-2) {$\G$};
%\draw [->] (3,-2) to [out = 100,in = -10] (0.45,0.60);
%\draw [->] (3.5,3) to [out = 180,in = 90] (1.5,1.2);
\node [right] at (2,3) {$\cono^{c_0,c_1}_\infty(\g)$};
\end{tikzpicture}
\caption{The coarse cone type at infinity, the black broken lines are $(c_0,c_1)$-quasi-geodesics.}\label{CCT}
\end{figure}

{Pozzetti-.S.-Wienhard \cite[Prop. 3.3]{PSW2}} together with Bochi-Potrie.S. \cite[Lemma 2.5]{BPS} (see also Pozzetti-S.-Wienhard \cite[Proposition 3.3]{PSW1}) give the following. The last statement can be found on {Pozzetti-S.-Wienhard \cite[Proposition 3.5]{PSW2}}:

\begin{prop1}[{Pozzetti-.S.-Wienhard \cite[Prop. 3.3]{PSW2}}]\label{CCTprop} For a given $\alpha>0$ there exist $c_0,c_1,$ depending on $\alpha$ and the domination constants of $\rho,$ such that for every hyperbolic $\g\in\G$ one has
$$(\xi^\t)^{-1}\big(B_{\theta,\alpha}(\g_\rho)\big)\subset \cono^{c_0,c_1}_\infty(\g).$$ Reciprocally, there exists $\alpha';$ only depending on $c_0,c_1$ and the domination constants of $\rho,$ such that $$\cono^{c_0,c_1}_\infty(\g)\subset(\xi^\t)^{-1}\big(B_{\theta,\alpha'}(\g_\rho)\big).$$ There exists then $N\in\N,$ only depending on $c_0,c_1$ and the domination constants of $\rho$ such that, for all $t\in\N$ the family $$\cal U_t=\big\{\g_\rho  B_{\t,\alpha}\big(\g_\rho \big):t\leq|\g|\leq t+1\big\}$$ is an open covering of $\xi^\t(\bord\G)$ and such that every element $\xi^\t(x)$ belongs to at most $N$ elements of the covering $\cal U_t.$
\end{prop1}

\subsection{Sullivan's shadow lemma}\label{Sullivansombra}

We establish now a version of Sullivan's shadow Lemma.

\begin{lema1}\label{sombralemma} Consider $\varphi\in(\E_\t)^* $ and let $\mu$ be a $\bb\varphi$-Patterson-Sullivan measure of exponent $\delta.$ Let $\nu=\xi^{\t}_*\mu,$ then given $\alpha>0$ there exist constants $C,$ $C'$ and $L\in\N$ such that for every $\g\in\G$ with $|\g|\geq L$ one has $$q^{-\delta\cdot\varphi\big(\cartan(\g_\rho)\big)}C'
\leq\nu\big(\g_\rho B_{\t,\alpha}(\g_\rho )\big)\leq Cq^{-\delta\cdot\varphi\big(\cartan(\g_\rho)\big)}.$$
\end{lema1}

\begin{proof} It suffices to stablish that there exist $\alpha$ and $\kappa>0$ such that for all large enough $\g\in\G$ one has $\nu\big(B_{\t,\alpha}(\g_\rho)\big)\geq\kappa.$ Indeed, using this fact the Lemma follows from the defining Equation \eqref{defining} and Equation \eqref{comparision-shadow}. 

In order to stablish the desired lower bound we suppose by contradiction that there exists $\alpha_n\to\infty,$ $\g_n\to\infty$ such that $\nu\big(B_{\t,\alpha_n}((\g_n)_\rho)\big)\to0$ as $n\to\infty.$ We can extract then a subsequence $(\g_{n_k})$ such that $$U_{\ii\t}\big((\g_{n_k}^{-1})_\rho\big)\to Y\in\cal F_{\ii\t},\ k\to\infty.$$ Moreover, since $\rho$ is $\t$-Anosov, Lemma \ref{tubo} guarantees that $Y=\xi^{\ii\t}(y)$ for some $y\in\bord\G.$ Also, since $\nu\big(B_{\t,\alpha_{n_k}}((\g_{n_k})_\rho)\big)\to0$ and $\alpha_{n_k}\to\infty$ we get that the complement $$\big(B_{\t,\alpha_{n_k}}((\g_{n_k})_\rho)\big)^c=\{X\in\cal F_\t:\exists\sroot\in\t\ \textrm{s.t.}\ \peso_\sroot\Gr_\t(U_{\ii\t}((\g_{n_k}^{-1})_\rho),X)\leq-\alpha_n\big\}$$ converges to the subset of $\cal F_\t$ $$\big\{X\in\cal F_\t:(X,\xi^{\ii\t}(y))\notin\posgen_\t\big\},$$ and that this subset has total $\nu$-mass. Since the support of $\nu$ is contained in $\xi^\t(\bord\G)$ and the equivariant maps are transverse (Proposition \ref{conical}), one has that $$\{\xi^\t(y)\}=\big\{\xi^\t(x):(\xi^\t(x),\xi^{\ii\t}(y))\notin\posgen_\t\big\}$$ has total $\nu$-mass. However, considering $\g\in\G$ with $\g y\neq y$ we get, since \begin{equation}\label{defining}\frac{d(\g_\rho) _*\nu}{d\nu}(\cdot)=q^{-\delta\cdot\varphi\big(\bus_\vt(\g_\rho^{-1},\cdot)\big)},\end{equation} that $\nu\{\xi^1(\g y)\}>0,$ contradicting that $\{\xi^\t(y)\}$ has total $\nu$-mass.\end{proof}

%\subsection{Divergence type along functionals}\label{divergentFunctionals}
\begin{cor1} For every $\varphi\in\inte\conodual$ one has  ${\displaystyle\sum_{\g\in\G}q^{-\delta^\varphi\varphi\big(\cartan(\g_\rho )\big)}=\infty.}$
\end{cor1}

\begin{proof} We apply Sullivan's shadow Lemma \ref{sombralemma} to the measure $\nu^\varphi$ of Lemma \ref{eqCocyclo}. Indeed, considering the coverings of $\xi^\t(\bord\G)$ given by Proposition \ref{CCTprop} one has $$1=\nu^\varphi\big(\xi^\t(\bord\G)\big)\leq \sum_{t\leq|\g|\leq t+1}\nu^\varphi\big(\g_\rho B_{\t,\alpha}(\g_\rho )\big)\leq C\sum_{t\leq|\g|\leq t+1}q^{-\delta^\varphi\varphi\big(\cartan(\g_\rho )\big)}$$ for all large enough $t,$ giving divergence of the desired series.\end{proof}

\subsection{Patterson-Sullivan Theory along the Anosov roots: surjectivity}\label{PS2} We prove here surjectivity of the map $\varphi\mapsto\mu^\varphi$ defined in \S\,\ref{PS1}. 

The following proposition should be compared with Pozzetti-S.-Wienhard \cite[Theorem 5.14]{PSW2}, where a similar result is obtained for measures on the flag space $\cal F_\vt,$ for $\vt$ not necessarily equal to $\t$ but assuming that $\t\cap\vt\neq\emptyset.$

\begin{prop1}\label{jaco}Consider $\varphi\in(\E_\t)^*.$ If there exists a $\bb\varphi$-Patterson-Sullivan measure $\mu$ of exponent $\delta$ then $\varphi\in\inte(\cone_{\t,\rho})^*,$ $\delta= \delta^\varphi$ and $\mu=\mu^\varphi.$
\end{prop1}

\begin{proof} %We have to show that for every $\eps$ the series $$\sum_{\g\in\G}e^{-(\delta+\eps)\varphi\big(\cartan(\g_\rho)\big)}<\infty.$$ 
We let $\nu=\xi^\t_*\mu.$ Using Proposition \ref{CCTprop} we get a family of covering $\cal U_t$ with bounded overlap. In combination with Lemma \ref{sombralemma} one has for $t$ large enough that $$1=\nu(\xi^\t(\bord\G))\geq K\sum_{\g:t\leq|\g|\leq t+1} e^{-\delta\varphi\big(\cartan(\g_\rho)\big)},$$ for some constant $K>0.$ This is to say, there exists $\kappa>0$ such that for all $t\in\R_+$ large one has $\sum_{\g:t\leq|\g|\leq t+1}e^{-\delta \varphi\big(\cartan(\g_\rho)\big)}\leq\kappa,$ which gives in turn that $$\sum_{\g:|\g|\leq t}e^{-\delta\varphi\big(\cartan(\g_\rho)\big)}\leq\kappa t.$$ A standard argument (using for example \S\,\ref{proxBenoist}) permits to replace Cartan projections with Jordan projections giving $$\sum_{[\g]:p([\g])\leq t}e^{-\delta\varphi\big(\lambda(\g_\rho)\big)}=\sum_{[\g]:p([\g])\leq t}e^{-\ell_{\delta\LL_{\t,\rho}^\varphi}(\g)}\leq\kappa' t,$$ for a suitable $\kappa',$ where $p(\g)$ is the $\sf g$-period of the periodic orbit associated to $[\g],$ and $\LL_{\t,\rho}^\varphi$ is the Ledrappier potential of $\bb\varphi.$ Formula \eqref{PrPer} for the pressure function gives then $$P(-\delta\LL_{\t,\rho}^\varphi)\leq0.$$

Consequently, Lemma \ref{positiva} gives that $\LL_{\t,\rho}^\varphi$ is Liv\v sic-cohomologous to a positive function, this is to say, $\varphi\in\inte\conodual.$ Finally, since Remark \ref{existenceAssumption} guarantees the existence assumptions of \S\,\ref{PattersonC} for $\bb\varphi,$ the remaining two equalities in the statement follow from Corollary \ref{entropiamax}.\end{proof}

\subsection{The critical hypersurface parametrizes Patterson-Sullivan measures}\label{Q2}

By Lemma \ref{arbitraryG} Assumption \ref{assuC} holds for $\bb{}$ and thus \S\,\ref{entropy1} applies. Define the $\t$-\emph{critical hypersurface}, resp. $\t$-\emph{convergence domain}, of $\rho$ by \begin{alignat*}{2}\cal Q_{\t,\rho}:=\cal Q_{\bb{}}  & =\Big\{\varphi\in\inte\conodual:\h_\varphi=1\Big\},  \\  \cal D_{\t,\rho}:=\cal D_{\bb{}}  & =\Big\{\varphi\in\inte\conodual:\h_\varphi\in(0,1)\Big\}  \\ & = \Big\{\varphi\in(\E_\t)^*:\sum_{[\g]\in[\G]}e^{-\varphi\big(\lambda(\g_\rho)\big)}<\infty\Big\}.\end{alignat*} Moreover, by Corollary \ref{existe} one has $\delta^\varphi=\h_\varphi$ so one has the equalities 

\begin{alignat*}{2}\cal Q_{\t,\rho}  & =\Big\{\varphi\in\inte\conodual:\delta^\varphi=1\Big\},\\\cal D_{\t,\rho} &  =\Big\{\varphi\in\inte\conodual:\delta^\varphi\in(0,1)\Big\} \\ & = \Big\{\varphi\in(\E_\t)^*:\sum_{\g\in\G}e^{-\varphi\big(\cartan(\g_\rho)\big)}<\infty\Big\}.\end{alignat*}

For $\varphi\in\cal Q_{\t,\rho}$ we consider the \emph{dynamical intersection} map $\II_\varphi=\II_\varphi^{\bb{}}:(\E_\t)^*\to\R,$ associated to the cocycle $\bb{}$ as in \S\,\ref{entropy1} and defined by 

$$\II_\varphi(\psi)=\II_\varphi^{\bb{}}(\psi)=\lim_{t\to\infty}\frac1{\#\sf R_t(\varphi)}\sum_{\g\in \sf R_t(\varphi)}\frac{\psi(\lambda(\g_\rho))}{\varphi(\lambda(\g_\rho))},$$ where $\sf R_t(\varphi)=\big\{\g\in\G\textrm{ hyperbolic}:\varphi(\lambda(\g_\rho))\leq t\}.$ Let $\ann(\cone_{\t,\rho})$ be the annihilator of the $\t$-limit cone and denote by $$\pi^\t_\rho:(\E_\t)^*\to(\E_\t)^*/\ann(\cone_{\t,\rho})$$ the quotient projection. As before, the map $\II^{\bb{}}$ is also well defined on $\pi^\t_\rho(\cal Q_{\t,\rho})\times (\E_\t)^*/\ann(\cone_{\t,\rho}).$

The following corollary was previously stablished in S. \cite{exponential} for $\K=\R$ and Zariski-dense Anosov representations of closed negatively curved manifolds (the equality $\sf T_\varphi\pi^\t_\rho\big(\cal Q_{\t,\rho}\big)=\ker\II_\varphi$ is new). 

\begin{cor1}\label{e=1} The sets $\cal Q_{\t,\rho}$ and $\pi^\t_\rho\big(\cal Q_{\t,\rho})$ are closed co-dimension-one analytic sub-manifolds, the latter bounds the strictly convex set $\pi^\t_\rho\big(\cal D_{\t,\rho}\big).$ The map $$\varphi\mapsto \sf T_\varphi\pi^\t_\rho\big(\cal Q_{\t,\rho}\big)=\ker\II_\varphi$$ is an analytic diffeomorphism between $\pi^\t_\rho\big(\cal Q_{\t,\rho}\big)$ and directions in the relative interior of $\cone_{\t,\rho}.$ \end{cor1}

We now prove the following:

\begin{prop1}\label{corr} The map $\varphi\mapsto\mu^\varphi$ is an analytic homeomorphism from the manifold $\pi^\t_\rho\big(\cal Q_{\t,\rho}\big)$ to the space of Patterson-Sullivan measures supported on $\xi^\t(\bord\G).$
\end{prop1}

\begin{proof} By uniqueness in Corollary \ref{existe} the map $\varphi\mapsto\mu^\varphi$ is well defined and injective. Regularity follows from Remark \ref{equiA} and analytic variation of equilibrium states (Theorem \ref{equilibrios}). Surjectivity follows from Proposition \ref{jaco}.\end{proof}

 %By Theorem \ref{jaco} for every $\eps>0$ one has that $\sum_{\g\in\G}q^{-(\delta+\eps)\varphi\big(\cartan(\g_\rho)\big)}<\infty.$ Consequently, for every $\eps'>0$ one has $$\#\big\{\g\in\G:\varphi(\cartan(\g_\rho))<-\eps'\big\}<\infty.$$ Proposition \ref{cerquita} then implies that $\varphi|\cone_{\t,\rho}\geq0,$ i.e. $\varphi\in\conodual$ so the cocycle $\bb{\varphi}$ has non-negative periods. A standard argument (using for example \S\,\ref{proxBenoist}) gives $\h_\varphi\leq\delta^\varphi\leq\delta<\infty,$ so by Lemma \ref{beta<infty} one has $\varphi\in\inte\conodual.$

%Finally, since $\nu\big(\xi^\t(\bord\G)\big)>0$ the restriction $\nu|\xi^\t(\bord\G)$ is, up to a constant, a $\bb\varphi$-Patterson-Sullivan measure of exponent $\delta.$ In particular, \S\,\ref{PattersonC} implies $\h_\varphi=\delta=\delta^\varphi$ and $\nu|\xi^\t(\bord\G)=\nu\big(\xi^\t(\bord\G)\big)\mu^\varphi.$ 

Proposition \ref{corr} was previously stablished by Lee-Oh \cite[Theorem 1.3]{HoLee} for $\K=\R$ and $\simple$-Anosov Zariski-dense representations. The convergence domain $\cal D_{\simple,\rho}$ is dual to Quint's \emph{growth indicator function} \cite{quint2}.

\begin{obs} Observe that, by definition, a $\bb\varphi$-Patterson-Sullivan measure has its support on $\bord\G,$ and thus on $\xi^\t(\bord\G)$ when pushed to $\cal F_\t.$ One could more generally study measures on $\cal F_\t$ verifying \begin{equation}\label{defQ}\frac{d(\g_\rho) _*\nu}{d\nu}(\cdot)=q^{-\delta\cdot\varphi\big(\bus_\vt(\g_\rho^{-1},\cdot)\big)},\end{equation} without imposing conditions on their support. Such measures exist, for example the $\sf K$-invariant measure on $\cal F_\t,$ but their exponent is too large. The question would be totally settled if the following had affirmative answer: Is the support of a measure verifying \eqref{defQ} with $\delta=\delta^\varphi$  necessarily contained on $\xi^\t(\bord\G)$?
\end{obs}

\subsection{Variation of the critical hypersurface}

We record the following consequence of Bridgeman-Canary-Labourie-S. \cite[\S\,6.3]{pressure}

\begin{cor1}[{Bridgeman-Canary-Labourie-S. \cite{pressure}}] Let $\{\rho_u:\G\to\sf G\}_{u\in D}$ be an analytic family of $\t$-Anosov representations. Then Liv\v sic-cohomology class of the Ledrappier potential $\LL_{\bb{}^{\rho_u}}:\widetilde{\sf U\G}\to\E_\t$ associated varies analytically with $u.$
\end{cor1}

Consequently we can apply Corollary \ref{Qanaly} to obtain:

\begin{cor1}Let $\{\rho_u:\G\to\sf G\}_{u\in D}$ be an analytic family of $\t$-Anosov representations, then the critical hypersurface $\cal Q_{\t,\rho_u}$ varies analytically (on compact sets of $\E_\t$) with the representation $u.$

\end{cor1}

\subsection{Consequences of the skew-product structure}\label{fibered2}  Consider $\varphi\in\inte\conodual.$ By Remark \ref{existenceAssumption} we can freely apply results from \S\,\ref{directions} and \S\,\ref{fibered} to the cocycle $\bb\varphi.$

 Let $\sf u_\varphi=\sf T_{\h_\varphi\varphi}\cal Q_{\t,\rho}\in\P(\cone_{\t,\rho})$ be the \emph{growth direction of $\varphi$}. By \S\,\ref{entropy1} the half line $\sf u_\varphi\cap\cone_{\t,\rho}$ lies in the relative interior of $\cone_{\t,\rho}$ (and every direction in this relative interior is obtained in this fashion).

Consider the $\varphi$-\emph{Bowen-Margulis} measure $\BM^\varphi$ on $\cjtot,$ defined as the induced on the quotient by \begin{equation}\label{forBM}e^{-\delta^\varphi[\cdot,\cdot]_\varphi}\bar\mu^\varphi\otimes\mu^\varphi\otimes d\Leb.\end{equation}

Consider $u_\varphi\in\sf u_\varphi$ with $\varphi(u_\varphi)=1$ and denote by $\df^\varphi=\big(\df^\varphi_t:\cjtot\to\cjtot\big)_{t\in\R}$ the \emph{directional flow}, induced on the quotient of $$t\cdot(x,y,v)=(x,y,v-tu_\varphi).$$

The ergodic dichotomy from \S\,\ref{ergoDic} gives then:

\begin{thm1}\label{dicoAnosov}Assume $\K=\R$ and $\rho$ is Zariski-dense, and let $\varphi\in\inte\big(\cone_{\t,\rho}\big)^*.$ If $|\t|\leq2$ then the directional flow $\df^\varphi$ is ergodic w.r.t $\BM^\varphi,$ in particular $\cal K(\df^\varphi)$ has total mass. If $|\t|\geq4$ then $\cal K(\df^\varphi)$ has measure $0.$
\end{thm1}

\begin{proof}The non-arithmeticity assumption for $\bb{}$ holds by Benoist's Theorem \ref{densidad} and thus Corollary \ref{CorergoDic} applies.\end{proof}

\subsection{Directional conical points}\label{medidadico} The present task is to study the set of points on $\bord\G$ that are \emph{conical} in the direction $\sf u_\varphi.$

Consider $y\in\bord\G$ and a sequence $\{\g_n\}\subset\G$ with $\g_n\to y.$ Then we say that $\g_n$ converges \emph{conically} to $y$ if for every $z\in\bord\G-\{y\}$ the sequence $\g_n^{-1}(z,y)$ remains on a compact subset of $\bord^2\G.$ 

\begin{obs*}\label{conica}Equivalently, since any compact subset of $\bord^2\G$ is contained in a compact subset of the form $\big\{(a,b):d_{\bord\G}(a,b)\geq\kappa\big\}$ for a fixed $\kappa,$ one has that $\g_n\to y$ conically if and only if there exists a geodesic ray $\{\alpha_i\}_0^\infty$ on $\G,$ converging to $y,$ such that $\{\g_n\}$ is at bounded Hausdorff distance from $\{\alpha_i\}_0^\infty.$ It follows then the existence of constants, $c_0,c_1$ such that for all $n$ one has \begin{equation}\label{ycono}\g_n^{-1}y\in\cono^{c_0,c_1}_\infty(\g_n).\end{equation}
\end{obs*}

Let us fix an (auxiliary) Euclidean norm on $\E_\t$ and denote by $B(v,r)$ the associated ball of radius $r$ about $v.$ The \emph{tube} of size $r$ about $\sf u_\varphi$ is the tubular neighborhood: $$\tube_r(\sf u_\varphi)=\{v\in\E_\t:B(v,r)\cap\sf u_\varphi\neq\emptyset\}.$$

\begin{defi1}We say that $y\in\bord\G$ is $(r,\varphi)$-\emph{conical} if there exists a conical sequence $\{\g_n\}\subset\G$ converging to $y$ such that for all $n$ $$\cartan_\t\big((\g_n)_\rho \big)\in \tube_r(\sf u_\varphi).$$ We say that $y$ is $\varphi$-conical if it is $(r,\varphi)$-conical for some $r.$ \end{defi1}

Let us denote by $\con{r,\varphi}\subset\bord\G$ the set of $(r,\varphi)$-conical points and by $\con\varphi$ the set of $\varphi$-conical points. We now establish the following dichotomy.

\begin{thm1}\label{thmB1}Assume $\K=\R$ and that $\rho$ is Zariski-dense. If $|\t|\leq2$ then $\mu^\varphi\big(\con\varphi)=1,$ if $|\t|\geq4$ then $\mu^\varphi\big(\con\varphi)=0.$
\end{thm1}

The Theorem follows directly from Theorem \ref{dicoAnosov} and the following proposition. Let us denote by $\pp:\bord^2\G\times\E_\t \to\G\/\big(\bord^2\G\times\E_\t\big)$ the quotient projection.

\begin{prop1}\label{recu-conico} A point $y\in\bord\G$ belongs to $\con\varphi$ if and only if for every pair $(x,v)\in(\bord\G-\{y\})\times\E_\t$ one has $\pp(x,y,v)\in\cal K({\df}^\varphi).$
\end{prop1}

\begin{proof} If $(x,y,v)\in\bord^2\G\times \E_\t$ is such that $y\in\con\varphi,$ then consider $r>0$ and $\g_n\to y$ conically such that $\cartan\big((\g_n)_\rho\big)\in\tube_r(\sf u_\varphi).$ By equation (\ref{ycono}), there exists $\eps$ given by Proposition \ref{CCTprop} (only depending on $c_0$ and $c_1$) such that for all $n$ $$\xi^\t(y)\in(\g_n)_\rho B_{\t,\eps}\big((\g_n)_\rho\big).$$ 

\noindent
Consequently equation (\ref{comparision-shadow}) gives \begin{equation}\label{tutti4}\big\|\bb{}(\g_n ^{-1},y)+\cartan_\t\big(\rho(\g_n)\big)\big\|=\big\|-\bb{}(\g_n,\g_n^{-1}\cdot y)+\cartan_\t\big(\rho(\g_n)\big)\big\|<K_\eps.\end{equation} 

\noindent By assumption $\cartan_\t\big(\rho(\g_n)\big)\in\tube_r(\sf u_\varphi)$ and one finds thus a divergent sequence $t_n\in\R_+$  such that \begin{equation}\label{tutti3}\| \bb{}(\g_n ^{-1},y)+t_nu_\varphi\|<K',\end{equation} for some $K'$ only depending on $r$ and $\eps.$ The sequence $$\df^\varphi_{-t_n}\g_n^{-1}(x,y,v)=\big(\g_n^{-1}x,\g_n^{-1}y,v-\bb{}(\g_n^{-1},y)-t_n u_\varphi\big)$$ is thus contained in $\big\{(z,w)\in\bord^2\G: d_{\bord\G}(z,w)>\kappa\big\}\times B(v,K'),$ for some $\kappa$ only depending on $d_{\bord\G}(x,y),$ in particular $\pp(x,y,v)\in\cal K({\df}^\varphi)$ as desired.

Reciprocally, if $\pp(x_0,y_0,v_0)\in\G\/\big(\bord^2\G\times\E_\t\big)$ belongs to $\cal K({\df}^\varphi),$ let $\BB$ be a bounded open set to which the ${\df}^\varphi$-orbit of $\pp(x_0,y_0,v_0)$ returns to unboundedly. Considering an accumulation point of the orbit points in $\BB$ we can assume that $\BB=\pp(\tilde\BB)$ for some $\tilde \BB$ of the form $$\big\{(z,w)\in\bord^2\G:d_{\bord\G}(z,w)\geq\kappa'\big\}\times B(v,c).$$ We obtain thus divergent sequences $\{\g_n\}\subset\G$ and $\{t_n\}\subset\R^+$ such that for all $n$ \begin{equation}\label{coni}d_{\bord\G}(\g_n^{-1}x_0,\g_n^{-1} y_0)>\kappa'\textrm{ and }\|\bb{}(\g_n^{-1},y_0)+t_nu_\varphi\|\leq K''.\end{equation} 

Considering subsequences we can assume that $\g_n^{-1}x_0\to x_\infty$ and $\g_n^{-1}y_0\to y_\infty.$ Necessarily $x_\infty\neq y_\infty$ since they are at least $\kappa'$ apart. The sequence $\{\g_n\}$ is thus conical, but it is still to be determined whether it converges to $x_0$ or to $y_0.$

Using the last inequality in (\ref{coni}) we deduce, since $t_n\to+\infty,$ that for all $\sroot\in\t$ $$\peso_\sroot\big(\bb{}(\g_n^{-1},y_0)\big)\to-\infty.$$ By definition of $\bb{}$ and the interpretation of the Buseman-Iwasawa cocycle via representations (equation (\ref{busnorma})) one has $\log\big(\|\Fund_\sroot\rho(\g_n^{-1})v\|/\|v\|\big)\to-\infty$ for any non-vanishing $v\in\Xi_{\Fund_\sroot}\big(\xi^\t(\g_n^{-1}y_0)\big),$ or equivalently, as $n\to\infty$ $$\frac{\|\Fund_\sroot\rho(\g_n^{-1})v\|}{\|v\|}\to 0.$$

\noindent We now use a standard linear algebra computation\footnote{\begin{lema1}[{Bochi-Potrie-S. \cite[Lemma A.3]{BPS}}]Let $A\in\GL_d(\R)$ have a gap at $\aa_1,$ then for every $v\in\R^d$ one has $$\frac{\|Av\|}{\|v\|}\geq \|A\|\sin\angle\big(\R\cdot v,U_{d-1}(A^{-1})\big).$$\end{lema1}} to conclude that $$\sin\angle\Big( \Xi_{\Fund_\sroot}\big(\xi^\t(y_0)\big),U_{d-1}\big(\Fund_\sroot(\g_n)_\rho\big)\Big)\to0.$$

\noindent By Lemma \ref{lejosplanos} one concludes that $$U_1\big(\Fund_\sroot\big(\rho(\g_n)\big)\big)\to\Xi_{\Fund_\sroot}\big(\xi^\t(y_0)\big),$$ as $n\to\infty$ for all $\sroot\in\t.$ Again by Lemma \ref{lejosplanos} one has $\g_n\to y_0$ as $n\to\infty$ (in $\G\cup\bord\G$) and thus, by conica\-li\-ty of $\{\g_n\},$ that for all $z\in\bord\G-\{y_0\}$ it holds $\g_n^{-1}z\to x_\infty.$ It follows then that $$U_{d-1}\big(\Fund_\sroot\rho(\g_n)^{-1}\big)\to \Xi_{\Fund_\sroot}^*\big(\xi^{\ii\t}(x_\infty)\big),$$ and, since $\g_n^{-1}y_0\to y_\infty\neq x_\infty,$ that $$\angle\big(\Xi_{\Fund_\sroot}\big(\xi^\t(\g_n^{-1}y_0)\big),U_{d-1}\big(\Fund_\sroot\rho(\g_n)^{-1}\big)\big)>\kappa'.$$ Since the latter lower bound holds for all $\sroot\in\t$ one concludes that $\xi^\t(\g_n^{-1}y_0)$ belongs to the Cartan basin $B_{\t,\kappa''}\big(\rho(\g_n)\big).$ Thus, as in equation (\ref{tutti4}), one has $$\big\|\bb{}(\g_n^{-1},y_0)+\cartan_\t\big(\rho(\g_n)\big)\big\|\leq K,$$ for some $K$ only depending on $\kappa''.$ The latter, together with the second inequality from equation (\ref{coni}) implies that $y_0$ is $\varphi$-conical, as desired.\end{proof}

\begin{proof}[Proof of Theorem \ref{thmB1}]

Consider a positive $\eps.$ Fix $y\in\con\varphi, x\in\bord\G-\{y\}$ and two neighborhoods $A^-$ and $A^+$ of $x$ and $y$ respectively so that for all $(z,w)\in A^-,A^+$ one has $\big|[z,w]_\varphi-[x,y]_\varphi\big|<\eps.$ Pick also an arbitrary $T>0$ so that the quotient projection $\pp$ is injective on $\tilde\BB=A^-\times A^+\times B(0,T).$ We can thus compute the measure of $\BB=\pp(\tilde\BB)$ by the formula (\ref{forBM}). 

If we let $\tilde{\cal K}({\df}^\varphi)=\pp^{-1}\big(\cal K({\df}^\varphi)\big),$ then the Lemma above asserts that $$A^-\times(A^+\cap\con\varphi)\times B(0,T)=\tilde{\cal K}({\df}^\varphi)\cap\tilde\BB.$$

\noindent 
If $|\t|\leq2$ Theorem \ref{dicoAnosov} states that $\BM^\varphi(\tilde\BB)=\BM^\varphi\big(\tilde{\cal K}({\df}^\varphi)\cap\tilde\BB\big),$ which implies, up to $e^{-\delta^\varphi\eps},$ that $$\mu^\varphi(A^+)=\mu^\varphi(A^+\cap\con\varphi).$$ Since $\eps$ is arbitrary one concludes $\mu^\varphi(\con\varphi)=1.$ On the other hand, if $|\t|\geq4$ then we have $\BM^\varphi\big(\tilde{\cal K}({\df}^\varphi)\big)=0$ so $\mu^\varphi(A^+\cap\con\varphi)=0$ and the theorem is proved.
\end{proof}

\appendix

\section{Ergodicity of skew-products with values on $\R$}\label{dim1}

We freely use notation from Proposition \ref{d=1} which we intend to prove. The proof presented here is mainly a collection of results.

We say that $\sf K$ is \emph{recurrent} if for every measurable set $A\subset\EE$ with $\nu(A)>0$ and every neighborhood $N(0)$ of $0$ in $V$ there exists $n\in\Z-\{0\}$ such that one has $$\nu\Big(A\cap\sigma^{-n}A\cap\big\{x:\sum_{k=0}^n\sf K(\sigma^ix)\in N(0)\big\}\Big)>0.$$ 

It is proven in Schmidt \cite[Theorem 5.5]{Schmidt1} that $\sf K$ is recurrent if and only if the skew-product $f^\sf K:\EE\times\R\to\EE\times\R$ is conservative (see Aaronson book \cite[\S\,1.1]{AaronsonLibro} for the definition). It is moreover a general fact that mean-zero cocycles over the reals are conservative, see \cite[Cor. 8.1.5]{AaronsonLibro} from which we state here a particular case. 

\begin{cor}Since by assumption $\int\sf Kd\nu=0,$ the cocycle $f^\sf K$ is conservative and so $\sf K$ is recurrent.\end{cor}

The proof of Proposition \ref{d=1} ends with the following theorem of Coelho (obtained by the combination of Example 2.4 and Corollary 3.4 of Coelho \cite{Coelho} ), specific to sub-shifts and equilibrium states.

\begin{thm1}[Coelho \cite{Coelho}]\label{teoyg} Assume $\sf K$ is non-arithmetic and let $\nu$ be an equilibrium state of $\sigma$ for a H\"older potential. Then $f$ is ergodic w.r.t $\Upomega_\nu$ if and only if $\sf K$ is recurrent.\end{thm1}

\section{Mixing}\label{mixingprueba} 

In this appendix we give a quick outline of the proof of Theorem \ref{mixing0}. We use small modifications of classical computations dating back at least to Babillot \cite{babillot} and appearing also in Babillot-Ledrappier \cite{babled}, Ledrappier-Sarig \cite{Ledrappier-Sarig} and more recently in Oh-Pan \cite{Oh-Pan} and Chow-Sarkar \cite{Chow-Sarkar}, where an extra parameter (an holonomy with values on a  compact group) has been added to the Ruelle operator. We thank M. Chow and P. Sarkar for pointing out an issue in the argument presented in S. \cite{orbitalcounting}, F. Ledrappier for suggesting the reference \cite{Ledrappier-Sarig} and H. Oh for pinpointing the reference \cite{Oh-Pan}.

It is first convenient to straighten the flow action by means of twisting the $r$-action. 

\begin{lema}Let $U=\R\times W,$ and let $k:\EE\to U$ be $k(x)=\big(r(x),\int_0^{r(x)}K(x,s)ds\big),$ then: \begin{itemize}\item[-] there exists $\varphi\in U^*$ such that $\varphi(k)=r>0,$\item[-]$\int kd\nu=(\int rd\nu,0)\neq0,$ \item[-]there exists a bi-H\"older homeomorphism $E:\Upsigma_r\times W\to\EE\times U/\hat k,$ where $$\hat k(x,u)=\big(\sigma(x),u-k(x)\big),$$ which is a measurable isomorphism between $\sus$ and $\nu\otimes\Leb_U/\hat k,$ that conjugates $\psi$ with the flow induced on the quotient by $$(x,u)\mapsto(x,u-t\tau),$$ where $\tau\in\int kd\nu$ is such that $\varphi(\tau)=1.$\end{itemize}                                                                                                                                                                                                                                                                                                                                                                                                                                                                             \end{lema}

\begin{proof} Define $E:\Upsigma_r\times W\to \EE\times U/\hat k$ by $E\big((x,t),w\big)=\Big(x,\big(t,w+\int_0^tK(x,s)ds\big)\Big).$ It is well defined since the above formula is equivariant, indeed one has 

\begin{alignat*}{3}\int_0^{t-r(x)}K(\sigma(x),s)ds & = \int_{r(x)}^{t}K(\sigma(x),s-r(x))ds & \\ & = \int_{r(x)}^{t}K(x,s)ds &\qquad(\textrm{$K$ is $\hat r$-invariant})\\ & = \int_{0}^{t}K(x,s)ds-\int_0^{r(x)}K(x,s)ds &
\end{alignat*} which implies

\begin{alignat*}{2} E\big(\hat r(x,t),w\big) & =\Big(\sigma(x),\big(t-r(x),w+\int_0^{t-r(x)}K(\sigma(x),s)ds\big)\Big)\\
 & = \Big(\sigma(x),\big(t-r(x),w+\int_{0}^{t}K(x,s)ds-\int_0^{r(x)}K(x,s)ds\big)\Big)\\
 & =\hat k\big(E((x,t),w)\big).
\end{alignat*} as desired. The remaining assertions follow similarly. \end{proof}

We will work from now on with this latter flow, still denoted by $\psi$ in order not to overcharge with notation. Up to Liv\v sic-cohomology we may assume that $k$ is defined on $\EE^+.$

By measure-theoretic arguments we consider $F,G:\EE^+\times U\to\R$ that we can assume have separated variables, i.e. can be written as $F(x,u)=p_F(x)v_F(u),$ with $p_F$ and $p_G$ H\"older-continuous and $v_F$ and $v_G$ smooth with compact support. We have to show that, as $t\to\infty,$ $$t^{\dim W/2}\sus\big(F\cdot G\circ\psi_t)\to\sus(F)\sus(G).$$ Tracing back the definitions one is brought up to understanding the limit as $t\to\infty$ of  \begin{equation}\label{teinque}t^{\dim W/2}\big(\int_{\EE^+\times U}\sum_{n\in\N}F(x,u)G\big(\sigma^nx,u-S_nk(x)-t\tau\big)d\nu d\Leb_U\big),\end{equation} where $S_nk(x)=\sum_{i=0}^nk(\sigma^i(x))$ is the Birkhoff sum. We focuss on the integral between brackets, only to multiply at the very end of our computation by $\sqrt{t}^{d-1},$ where $d=\dim U.$ The above integral becomes $$\int_{\EE^+\times U}\sum_{n\in\N}p_F(x)p_G\big(\sigma^nx\big)v_F(u)v_G\big(u-S_nk(x)-t\tau\big)d\nu(x) d\Leb_U(u).$$

Recall that by assumption there is $\varphi\in U^*$ so that $\varphi(k)=hr>0$ and that  $P(-hr)=0,$ so up to Liv\v sic-cohomology we can assume that $-\varphi(k)=-hr$ is \emph{normalized}, i.e. so that the Ruelle operator, defined by $$\cal L_{\varphi}\Phi(x)=\sum_{y:\sigma(y)=x}e^{-\varphi\big(k(y)\big)}\Phi(y),$$ verifies $\cal L_{\varphi}^*\nu=\nu,$ in particular, for every pair of H\"older-continuous functions $j,l$ on $\EE^+$ one has $\int_{\EE^+}j(\sigma x)l(x)d\nu=\int_{\EE^+}j(x)\big(\cal L_{\varphi}l\big)(x)d\nu.$

Denote by $\Leb_{U^*}$ the Lebesgue measure on $U^*$ defined by the Fourier inversion formula $$v_G(w)=\int_{U^*}e^{i\psi(w)}\cal Fv_G(\psi)d\Leb_{U^*}(\psi)$$ for the Fourier transform $\cal Fv_G$ if $v_G.$ As in Babillot-Ledrappier \cite[\S\,2.3]{babled} we can, and will, assume that $\cal F v_G$ is of class $\clase^N$ for some $N>(d-1)/2$ and has compact support.

We will suppress the notation $\nu,$ $\Leb_U$ and $\Leb_{U^*}$ from the integrals from now on. The desired integral, equation (\ref{teinque}), then becomes \begin{alignat}{2}\label{22} &  \int_{\EE^+\times U}\sum_{n\in\N}\big(\cal L_{\varphi}^np_F\big)(x)p_G\big(x\big)v_F(u)v_G\big(u-S_nk(x)-t\tau\big)dx du\nonumber\\= &  \int_{\EE^+\times U}\sum_{n\in\N}\big(\cal L_{\varphi}^np_F\big)(x)p_G\big(x\big)v_F(u)\int_{U^*}e^{i\psi\big(u-S_nk(x)-t\tau\big)}\cal Fv_G(\psi)d\psi dx du,\nonumber\\= & \int_{\EE^+\times U}\int_{U^*}\sum_{n\in\N}\big(\cal L_{\varphi+i\psi}^np_F\big)(x)p_G(x)v_F(u)e^{i\psi\big(u-t\tau\big)}\cal Fv_G(\psi)d\psi dx du \nonumber\\ = & \int_{\EE^+\times U}p_G(x)v_F(u)\int_{U^*}\sum_{n\in\N}\big(\cal L_{\varphi+i\psi}^np_F\big)(x)e^{i\psi\big(u-t\tau\big)}\cal Fv_G(\psi)d\psi dx du. \end{alignat} 

We seek thus to understand the nature of $\psi\mapsto\sum_{n\in\N}\cal L_{\varphi+i\psi}^n,$ for which one is brought to understand the spectral radius $\rr_\psi$ of $\cal L_{\varphi+i\psi}.$  Applying Parry-Pollicott \cite[Chapter 4]{parrypollicott} we obtain that for every $\psi\in U^*$ one has $\rr_\psi\in(0,1].$ We then distinguish two situations.

\textbf{The spectral radius $\rr_\psi$ is smaller than $1;$} in which case $$\eta\mapsto\sum_{n\in\N}\cal L_{\varphi+i\eta}^n=(1-\cal L_{\varphi+i\eta})^{-1},$$ is analytic on a neighborhood of $\psi.$

\textbf{The spectral radius $\rr_\psi$ equals $1.$} One has then the following:

\begin{thm}[{Parry-Pollicott \cite[Chapter 4]{parrypollicott}}]\label{w} One has $\rr_\psi=1$ if and only if there exists a H\"older continuous $w_\psi:\EE^+\to\SS^1$ such that for all $x\in\EE^+$ one has $$e^{i\psi(k(x))}=\lambda_\psi\frac{w_\psi(\sigma(x))}{w_\psi(x)}.$$ In this situation the function $w_\psi$ is unique up to scalars.
\end{thm}

Applying moreover Parry-Pollicott \cite[Theorem 4.5]{parrypollicott} and the perturbation theorem \cite[Prop. 4.6]{parrypollicott}, there exists a neighborhood $\cal O_\psi$ of $\psi$ such that for all $\eta\in\cal O_\psi$ one has \begin{equation}\label{desc}\cal L_{\varphi+i\eta}=\lambda_\eta Q_\eta+N_\eta \end{equation} where $Q_\eta$ is a rank-one projector, $N_\eta$ is an operator with spectral radius strictly smaller than $\rr_\eta=|\lambda_\eta|$ and such that $Q_\eta N_\eta=N_\eta Q_\eta=0.$ The above objects are analytic on $\cal O_\psi.$ The operator $M_\psi=\sum_{n\in\N}N_\psi^n$ is hence well defined and analytic on $\cal O_\psi.$ Observe that, as we have assumed $-\varphi(k)$ to be normalized, one has \begin{equation}\label{Q0}Q_0(f)(x)=\big(\int_{\EE^+}fd\nu\big)\cdot 1.\end{equation}

\begin{obs} On the other hand, since $k$ has dense group of periods on $U,$ $\psi(k):\EE^+\to\R$ has dense group of periods as soon as $\psi\neq0,$ and Parry-Pollicott \cite[Theorem 4.5]{parrypollicott} implies that  $\{\lambda_\psi^n:n\in\Z\}$ is dense in $\SS^1=\bord\DD.$ In particular $\lambda_\psi\neq1,$ unless $\psi=0.$ Consequently, if $\psi\neq0$ then for all $\eta\in \cal O_\psi$ the operator $$\frac{Q_\eta}{1-\lambda_\eta}-M_\eta$$ is well defined and analytic on $\cal O_\psi.$\end{obs}

\begin{lema}\label{analitico}If $\rr_\psi=1$ then for all $\eta\in\cal O_\psi-\{\psi\}$ it holds $\rr_\eta<1.$ Consequently $$\eta\mapsto\sum_{n\in\N}\cal L_{\varphi+i\eta}^n=(1-\cal L_{\varphi+i\eta})^{-1}=\frac{Q_\eta}{1-\lambda_\eta}+M_\eta$$ is analytic on $\cal O_\psi-\{\psi\}$ and, if $\psi\neq0,$ it extends analytically to $\cal O_\psi,$ as the right-hand-side of the equation is well defined on $\psi.$
\end{lema}

\begin{proof}As $\eta\mapsto\lambda_\eta$ is analytic on $\cal O_\psi,$ together with the above density result, it follows that there isa neighborhood (possibly smaller but) still denoted by $\cal O_\psi$ such that if $|\lambda_\eta|=1$ then $\lambda_\eta=\lambda_\psi.$ One has then applying Theorem \ref{w} that $$e^{i(\psi-\eta)(k(x))}=\lambda_\psi \frac{w_\psi(\sigma(x))}{w_\psi(x)}\lambda_\eta^{-1}\frac{w_\eta(x)}{w_\eta(\sigma(x))}=\frac{(w_\psi/w_\eta )(\sigma(x))}{(w_\psi/w_\eta )(x)}.$$ The remark and Theorem \ref{w} give $\psi-\eta=0$ and so the lemma is stablished.\end{proof}

One obtains then that the operator $\sum_{n\in\N}\cal L_{\varphi+i\psi}^n$ is well defined and varies analytically on $\psi$ except at $\psi=0.$ One is thus taken to localize the integral (\ref{22}) about $0.$ To that end one considers an auxiliary $\clase^\infty$ function $\kappa:U^*\to\R,$ supported on the neighborhood $\cal O_0$ where equation (\ref{desc}) holds, with $\kappa(0)=1,$ and we seek to understand the modified integral over $U^*$ on equation (\ref{22}) given, for $(x,u)\in\EE^+\times U,$ by  \begin{equation}\label{modif}\int_{U^*}\big(1-\kappa(\psi)\big)\sum_{n\in\N}\big(\cal L_{\varphi+i\psi}^np_F\big)(x)e^{i\psi\big(u-t\tau\big)}\cal Fv_G(\psi)d\psi.\end{equation}

Consider from \S\,\ref{E=1} the critical hypersurface $\cal Q_{k}\subset U^*.$ It follows from Babillot-Ledrappier \cite{babled} that one has $\cal Q_{k}=\PP^{-1}(0)$ and the tangent space $\sf T_\varphi\cal Q_{k}=\{\psi\in U^*:\psi(\tau)=0\}.$ For $\psi\in U^*$ let \begin{equation}\label{descpsi}\psi=s_\psi\varphi+\psi_0\end{equation} be its decomposition along $U^*=\R\varphi\oplus\sf T_\varphi\cal Q_{k}.$ Decomposing $d\psi$ as $dsd\psi_0,$ the integral (\ref{modif}) becomes

\begin{equation}\label{modif2}\int_{\R}e^{-its}\int_{\sf T_\varphi\cal Q_{k}}e^{i\psi(u)}\big(1-\kappa(\psi)\big)\sum_{n\in\N}\big(\cal L_{\varphi+i\psi}^np_F\big)(x)\cal Fv_G(\psi)d\psi_0ds=O(t^{-N})\end{equation} 

\noindent
as it is the Fourier transform of the integral over $\sf T_\varphi\cal Q_{k},$ which is, by Lemma \ref{analitico} as regular as $\cal Fv_G(\psi),$ and this function was chosen to be of class $\clase^N$ for some $N>(d-1)/2.$

We can thus focus on the integral from equation (\ref{22}) localized about 0, so it becomes, using again Lemma \ref{analitico}, \begin{alignat}{2}\label{teinque3}\int_{\EE^+\times U}\int_{U^*}\kappa(\psi)\Big(\frac{(Q_\psi p_F)x}{1-\lambda_\psi}+(M_\psi p_F)(x)\Big)p_G(x)v_F(u)e^{i\psi\big(u-t\tau\big)} & \cal Fv_G(\psi)d\psi dx du\nonumber\\ & +O(t^{-N}).\end{alignat}

We first treat the term containing $M_\psi,$ which is dealt with as we did with equation (\ref{modif}). Indeed, for the same reasons one has for all $(x,u)\in\EE^+\times U,$ the integral 

\begin{alignat}{2}\label{Marafue} & \int_{U^*}\kappa(\psi)(M_\psi p_F)(x)e^{i\psi\big(u-t\tau\big)}  \cal Fv_G(\psi)d\psi \nonumber \\ &=\int_{\R}e^{-its}\int_{\sf T_\varphi\cal Q_{k}}e^{i\psi(u)}\kappa(\psi)(M_\psi p_F)(x)  \cal Fv_G(\psi)d\psi_0ds \nonumber\\ &= O(t^{-N}).\end{alignat}

We effort then on understanding the integral \begin{equation}\label{dondetamo}\int_{U^*}\kappa(\psi)\frac{(Q_\psi p_F)x}{1-\lambda_\psi}e^{i\psi\big(u-t\tau\big)}\cal Fv_G(\psi)d\psi,\end{equation} the issue being the singularity at $\psi=0$ of $1/(1-\lambda_\psi).$ To that end, consider the function $\QQ: \sf T_\varphi\cal Q_{k}\to\R$ defined implicitly by the equation $$\QQ(\psi_0)\varphi+\psi_0\in\cal Q_{k}.$$ It is analytic, critical at 0 with $\QQ(0)=1,$ and has positive-definite Hessian at 0, $\Hess_0\QQ,$ using Taylor expansion one writes $$\QQ(\psi_0)=1+(1/2)\Hess_0\QQ(\psi_0)+O\big(\|\psi_0\|^2\big).$$

One applies the Weierstrass preparation Theorem \cite[Theorem 7.5.1]{Hormander} to express $1-\lambda_\psi$  about $0$ as $$1-\lambda_\psi=a(\psi)\Big(is_\psi-(1/2)\Hess_0\QQ(\psi_0)-O\big(\|\psi_0\|^2\big)\Big),$$ (recall the decomposition of $\psi$ from equation (\ref{descpsi})) where $a$ is real-analytic and $a(0)=h\int rd\nu$ (see Babillot-Ledrappier \cite[page 37]{babled} or Ledrappier-Sarig \cite[page 17]{Ledrappier-Sarig} for details). Whence, as in \cite[Lemma 2.3]{babled}, using the formula $1/z=-\int_0^\infty e^{Tz}dT$ one has 

$$\frac1{1-\lambda_\psi}=-\frac1{a(\psi)}\int_0^\infty e^{T\Big(is_\psi-(1/2)\Hess_0\QQ(\psi_0)-O\big(\|\psi_0\|^2\big)\Big)}dT,$$

\noindent and equation (\ref{dondetamo}) becomes, denoting by $C(\psi,x)=\kappa(\psi)\frac{(Q_\psi p_F)(x)}{a(\psi)}\cal Fv_G(\psi)$ to lighten the notation,

\begin{alignat}{2}\label{separar} & \int_{U^*}e^{i\psi(u-t\tau)}\kappa(\psi)\frac{(Q_\psi p_F)(x)}{a(\psi)}\cal Fv_G(\psi)\int_0^\infty e^{T\Big(is_\psi-(1/2)\Hess_0\QQ(\psi_0)-O\big(\|\psi_0\|^2\big)\Big)}dT d\psi_0ds \nonumber\\ &=\int_0^\infty\int_{\R}e^{-i(t-T)s}\int_{\sf T_\varphi\cal Q_{k}}C(\psi,x)e^{i\psi(u)-T\Big(\Hess_0\QQ(\psi_0)/2+O\big(\|\psi_0\|^2\big)\Big)} d\psi_0dsdT.
\end{alignat}

Using Taylor series on the $\psi$-variable we may, and will, assume that $C(\psi,x)$ is of the form $c_x(s_\psi)b(\psi_0).$ Equation (\ref{separar}) now becomes

\begin{alignat}{2}\label{separar2} &\int_0^\infty\int_{\R}e^{-i\big(t-T-\varphi(u)\big)s}c_{x}(s)ds\int_{\sf T_\varphi\cal Q_{k}}b(\psi_0)e^{i\psi_0(u)-T\Big(\Hess_0\QQ(\psi_0)/2+O\big(\|\psi_0\|^2\big)\Big)} d\psi_0dT \nonumber \\&= \int_0^\infty \cal F c_{x}\big(t-T-\varphi(u)\big)\int_{\sf T_\varphi\cal Q_{k}}b(\psi_0)e^{i\psi_0(u)-T\Big(\Hess_0\QQ(\psi_0)/2+O\big(\|\psi_0\|^2\big)\Big)} d\psi_0dT \nonumber\\ &= \int_{\varphi(u)}^\infty\cal F c_{x}\big(t-T\big)\int_{\sf T_\varphi\cal Q_{k}}b(\psi_0)e^{i\psi_0(u)-\big(T-\varphi(u)\big)\Big(\Hess_0\QQ(\psi_0)/2+O\big(\|\psi_0\|^2\big)\Big)} d\psi_0dT,\nonumber\\ &=\int_{\max(t/2,\varphi(u))}^\infty\cal F c_{x}(t-T)\int_{\sf T_\varphi\cal Q_{k}}b(\psi_0)e^{i\psi_0(u)-\big(T-\varphi(u)\big)\Big(\Hess_0\QQ(\psi_0)/2+O\big(\|\psi_0\|^2\big)\Big)} d\psi_0dT+O(t^{-N}),
\end{alignat}

\noindent
where the last equality is, as in \cite[page 27]{babled}, essentially due to the fact that $c_x$ has compact support and is of class $\clase^N.$ Applying the change of variables $\psi_0\mapsto\psi_0/\sqrt T$ the last integral becomes

\begin{alignat}{2}\label{tamoahi}&\int_{\max(t/2,\varphi(u))}^\infty\frac{\cal F c_{x}(t-T)}{{\sqrt T}^{d-1}}\int_{\sf T_\varphi\cal Q_{k}}b(\frac{\psi_0}{\sqrt T})e^{i\frac{\psi_0}{\sqrt T}(u)-\big(1-\frac{\varphi(u)}T\big)\Big(\Hess_0\QQ(\psi_0)/2+O\big(\frac{\|\psi_0\|^2}T\big)\Big)} d\psi_0dT+O(t^{-N}) \nonumber\\ &=-\int_{-\infty}^{\min(t/2,t-\varphi(u))}\frac{\cal F c_{x}(S)}{{\sqrt{t-S}}^{d-1}}\int_{\sf T_\varphi\cal Q_{k}}b(\frac{\psi_0}{\sqrt{t-S}})e^{i\frac{\psi_0}{\sqrt{t-S}}(u)-\big(1-\frac{\varphi(u)}{t-S}\big)\Big(\Hess_0\QQ(\psi_0)/2+O\big(\frac{\|\psi_0\|^2}{t-S}\big)\Big)} d\psi_0dS+O(t^{-N}).\end{alignat} 

\noindent

We finally multiply by  ${\sqrt t}^{d-1},$ take limit as $t\to\infty$ and trace back our definitions to get, as we assumed $N>(d-1)/2,$ the convergence of equation (\ref{tamoahi}) (and thus that of equation (\ref{dondetamo})) to \begin{alignat}{2}\label{finalpre}\lim_{t\to\infty}\sqrt{t}^{d-1}\int_{U^*}\kappa(\psi) & \frac{(Q_\psi p_F)x}{1-\lambda_\psi}e^{i\psi\big(u-t\tau\big)}\cal Fv_G(\psi)d\psi \nonumber\\ & =  b(0)\int_{-\infty}^\infty \cal Fc_x(S)dS \int_{\sf T_\varphi\cal Q_{k}}e^{-\frac12\Hess_0\QQ(\psi_0)}d\psi_0 \nonumber\\ & =c_x(0)b(0)\int_{\sf T_\varphi\cal Q_{k}}e^{-\frac12\Hess_0\QQ(\psi_0)}d\psi_0\nonumber\\ & = \frac{Q_0(p_F)(x)\cal F v_G(0)}{a(0)}\int_{\sf T_\varphi\cal Q_{k}}e^{-\frac12\Hess_0\QQ(\psi_0)}d\psi_0\nonumber\\&= \frac{\int_{\EE^+} p_Fd\nu\int_U v_G(u)du}{h\int_{\EE^+} rd\nu}\int_{\sf T_\varphi\cal Q_{k}}e^{-\frac12\Hess_0\QQ(\psi_0)}d\psi_0,
\end{alignat}

\noindent
where we have used equation (\ref{Q0}) and the formula $\cal Fv_G(0)=\int_U v_G(u)du.$ Observe that, as $\Hess_0\QQ$ is positive-definite, the integral $\int_{\sf T_\varphi\cal Q_{k}}e^{-\frac12\Hess_0\QQ(\psi_0)}d\psi_0$ is finite (and non-zero).

\begin{obs}\label{sqrtt} If one modifies the action $(x,u)\mapsto(x,u-t\tau)$ to consider the induced on the quotient by $(x,u)\mapsto(x,u-t\tau-\sqrt t w_0)$ for a fixed $w_0\in\ker\varphi$ then tracing the compuations one readily sees that \begin{alignat*}{2}\lim_{t\to\infty}\sqrt{t}^{d-1}\int_{U^*}\kappa(\psi) & \frac{(Q_\psi p_F)x}{1-\lambda_\psi}e^{i\psi\big(u-t\tau-\sqrt t w_0\big)}\cal Fv_G(\psi)d\psi\\ & = \frac{\int_{\EE^+} p_Fd\nu\int_U v_G(u)du}{h\int_{\EE^+} rd\nu}\int_{\sf T_\varphi\cal Q_{k}}e^{-i\psi_0(w_0)}e^{-\frac12\Hess_0\QQ(\psi_0)}d\psi_0\\ & = \frac{\int_{\EE^+} p_Fd\nu\int_U v_G(u)du}{h\int_{\EE^+} rd\nu}\cal F\mathrm H (w_0),\end{alignat*}
where we have defined $$\mathrm H(\psi_ 0)=e^{-(1/2)\Hess_0(\psi_0)}.$$ Observe that, as $\mathrm H(\psi_0)=\mathrm H(-\psi_0),$ the value of the Fourier transform $\cal F\mathrm H (w_0)\in\R.$
\end{obs}

We finally group back the equations to get the desired result. Indeed, equation (\ref{teinque}) is \begin{alignat*}{2}\sqrt{t}^{d-1} & \int_{\EE^+\times U}  \sum_{n\in\N}F(x,u)G\big(\sigma^nx,u-S_nk(x)-t\tau\big)d\nu d\Leb_U\\  & = \sqrt{t}^{d-1}\int p_G(x)v_F(u) \int\kappa(\psi)\frac{(Q_\psi p_F)x}{1-\lambda_\psi}e^{i\psi\big(u-t\tau\big)}\cal Fv_G(\psi)d\psi dudx+O(1/t)\\ &= \frac{\int_{\EE^+} p_Fd\nu\int_U v_G(u)du}{h\int_{\EE^+} rd\nu}\int p_G(x) v_F(u)dxdu\int_{\sf T_\varphi\cal Q_{k}}e^{-\frac12\Hess_0\QQ(\psi_0)}d\psi_0 +O(1/t)\\ & =\frac{\cal F\mathrm H(0)}{h\int_{\EE^+} rd\nu}\int Fd\nu d\Leb \int Gd\nu d\Leb+O(1/t), \end{alignat*} where we have used, in the first equality, equations (\ref{modif2}) and (\ref{Marafue}) and the fact that $N>(1/2)(d-1)$ and, in the second equality, the convergence from equation (\ref{finalpre}). This completes the sketch of proof.

\bibliography{hdr0}
\bibliographystyle{plain}

\author{\vbox{\footnotesize\noindent 
	Andr\'es Sambarino\\
	CNRS - Sorbonne Universit\'e \\ IMJ-PRG (CNRS UMR 7586)\\ 
	4 place Jussieu 75005 Paris France\\
	\texttt{andres.sambarino@imj-prg.fr}
\bigskip}}

\end{document}